\documentclass[10pt]{article}
\usepackage{amsthm}
\usepackage[utf8]{inputenc}
\usepackage[english]{babel}
\usepackage{amsmath}
\usepackage{amsfonts}
\usepackage{amssymb}
\usepackage{graphicx}
\usepackage{float} 
\usepackage{subcaption}
\usepackage{stackrel}
\usepackage{cite}
\usepackage{vmargin}
\usepackage{enumerate}
\usepackage{textcomp}
\usepackage[dvipsnames]{xcolor}
\usepackage{hyperref}
\usepackage{lineno}
\usepackage{comment}

\newtheorem{prop}{Proposition}
\newtheorem{lemma}{Lemma}

\newtheorem{corollary}{Corollary}
\newtheorem{theorem}{Theorem}
\newtheorem{remark}{Remark}

\numberwithin{equation}{section}
\usepackage{xcolor}

\setmargins{2.5cm} 
{1.5cm} 
{16.5cm} 
{23.42cm} 
{0pt} 
{2.75cm} 
{0pt} 
{1cm} 

\usepackage[final]{microtype}

\newcommand{\authorblock}[4][0.30\textwidth]{%
  \begin{minipage}[t]{#1}
    \centering
    \textbf{#2}\\[-0.05em]
    {\small #3\par}
    \vspace{0.12em}
    {\small\ttfamily #4}
  \end{minipage}%
}

\newcommand{\HALtitleblock}{%
\begin{center}
  \rule{0.92\textwidth}{0.8pt}\\[1.25em]
  {\LARGE\scshape
  Euler Scheme for Stochastic Functional Differential Equations\\
  Driven by Fractional Brownian Motion via Fractional Calculus Techniques\par}
  \vspace{1.15em}
  \rule{0.92\textwidth}{0.8pt}\\[1.15em]
  {\large\scshape A Preprint\par}

  \vspace{2.0em}
  \authorblock{Alexander Abreu}%
    {CIMFAV\\Facultad de Ingenier\'ia\\Universidad de Valpara\'iso\\Valpara\'iso, Chile}%
    {pedro.abreu@postgrado.uv.cl}
  \hfill
  \authorblock{H\'ector Araya}%
    {Faculty of Engineering \& Sciences\\Universidad Adolfo Ib\'a\~nez\\Vi\~na del Mar, Chile}%
    {hector.araya@uai.cl}
  \hfill
  \authorblock{Lisandro Ferm\'in\footnotemark[1]}%
    {Aix-Marseille University\\CNRS, AMSE\\Marseille, France}%
    {lisandro.fermin@univ-amu.fr}

  \vspace{2.0em}
  \authorblock[0.38\textwidth]{Johanna Garz\'on\footnotemark[2]}%
    {Departamento de Matem\'aticas\\Universidad Nacional de Colombia\\Bogot\'a, Colombia}%
    {mjgarzonm@unal.edu.co}
  \hspace{0.07\textwidth}
  \authorblock[0.38\textwidth]{Soledad Torres}%
    {CIMFAV\\Facultad de Ingenier\'ia\\Universidad de Valpara\'iso\\Valpara\'iso, Chile}%
    {soledad.torres@uv.cl}

  \vspace{2.0em}
  {\normalsize \today\par}
\end{center}
}

\begin{document}
\thispagestyle{empty}
\renewcommand{\thefootnote}{\fnsymbol{footnote}}
\HALtitleblock

\footnotetext[1]{The project leading to this publication has received funding from the French government under the ``France 2030'' investment plan managed by the French National Research Agency (reference: ANR-17-EURE-0020) and from the Excellence Initiative of Aix-Marseille University - A*MIDEX.}
\footnotetext[2]{Corresponding author.}
\setcounter{footnote}{0}
\renewcommand{\thefootnote}{\arabic{footnote}}

\vspace{0.4cm}
\begin{abstract}
We study a stochastic functional differential equation (SFDE) with memory driven by a fractional Brownian motion (fBm) with Hurst parameter $H>1/2$. An Euler-type numerical scheme is proposed and analyzed under suitable regularity conditions on the drift and diffusion coefficients using tools from fractional calculus. We prove the convergence of the scheme and derive the corresponding rate in terms of the discretization step. Numerical simulations illustrate the theoretical results and confirm the accuracy of the proposed method.
\end{abstract}

{\bf 2010 AMS Classification Numbers:} 
Primary 60H10, 60G22; 
Secondary 65C30, 65L20, 60H35, 34K50.


\section{Introduction}
Stochastic differential equations with delay or memory have been widely used to model dynamical systems whose future evolution depends not only on the present state but also on their past history. Such equations arise naturally in several areas, including mathematical biology, economics, engineering, and physics, where memory effects play a fundamental role (see, for example, \cite{Mohammed1984, Mao1997}). In particular, stochastic functional differential equations with memory allow us to describe systems with long-range historical dependence, providing a more realistic framework for modeling persistent phenomena.

In this article, we study a class of stochastic functional differential equations (SFDE)  with distributed memory driven by fractional Brownian motion. The evolution of the system depends not only on its present state but also on a functional of its past trajectory over a finite delay interval. This hereditary structure is modeled through a distributed delay term, allowing us to capture finite-memory effects within a non-Markovian stochastic framework.
More precisely, we consider the equation
\begin{equation}\label{eq1-0}
X(t)=\left\{
\begin{array}{lc}
X(0) + \displaystyle\int_0^t b(s,X(s),Y(s))\,{d}s
+ \displaystyle\int_0^t \sigma(s,X(s),Y(s))\,{d}B^H(s),     &  t \in [0,T] \\
\phi(t),     & t \in [-r,0]
\end{array}
 \right.
\end{equation}
where $B^H=(B^H(t))_{t \ge 0}$ is a fractional Brownian motion with Hurst parameter $H> 1/2$. The stochastic integral with respect to $B^H$ is understood in the sense of the generalized Lebesgue-Stieltjes integral as introduced by Zähle \cite{Zahle1998}. The drift and diffusion coefficients $b$ and $\sigma$ are assumed to be continuous functions, and the initial condition $\phi : [-r,0] \to \mathbb{R}$ is a Hölder continuous function.
The memory effect is modeled through the functional $Y(t)$ defined by
\begin{equation*}
Y(t) = \int_{-r}^0 K(t,s,X(t+s))\,{d}s
     = \int_{t-r}^t K(t,s-t,X(s))\,{d}s,
\qquad t \ge 0,
\end{equation*}
where the kernel $K : [0,T] \times [-r,0] \times \mathbb{R} \to \mathbb{R}$ is a continuous function. This formulation allows us to capture hereditary effects over a finite window of the past while maintaining functional dependence on the entire trajectory segment.

From a mathematical perspective, the presence of memory introduces a non-Markovian structure that significantly complicates both the analysis of the solutions and the design of numerical approximation schemes, since the evolution depends on the entire past trajectory.

In the classical setting, randomness is typically modeled by standard Brownian motion, leading to Markovian stochastic differential equations and well-developed theory from both analytical and numerical perspectives \cite{KloedenPlaten1992}. However, empirical evidence in many applications reveals the presence of long-range dependence and temporal correlations that cannot be adequately captured by classical Brownian noise. In this context, fractional Brownian motion (fBm), introduced by Mandelbrot and Van Ness \cite{MandelbrotVanNess1968}, has emerged as a fundamental tool for modeling processes with long-range and memory dependence.

The fractional Brownian motion is a centered Gaussian process characterized by a Hurst parameter $H \in (0,1)$, whose sample paths are Hölder continuous of any order strictly less than $H$. In the regime $H > 1/2$, the process exhibits positively correlated increments, making it particularly suitable for modeling persistent behavior \cite{MandelbrotVanNess1968, Biagini2008}. Nevertheless, when $H \neq 1/2$, the fractional Brownian motion is not a semimartingale, which prevents the direct application of the calculus It\^o and introduces substantial analytical and numerical challenges.

For stochastic differential equations driven by fractional Brownian motion with $H > 1/2$, a widely used approach consists in defining the stochastic integral in the sense of generalized Lebesgue-Stieltjes or Young integrals, exploiting the Hölder regularity of the trajectories \cite{Young1936, Zahle1998, HuNualart2009}. This framework has allowed the establishment of uniqueness and existence results for various classes of stochastic differential equations, including functional equations with memory and delay, under suitable regularity assumptions for the coefficients \cite{FerranteRovira2006, BH2011}.

The combination of fractional noise and distributed delay creates a challenging setting for numerical approximation, since temporal roughness and non-local memory effects must be controlled simultaneously.

Equations of the form (\ref{eq1-0}), involving distributed memory terms, have been previously studied in the context of stochastic functional differential equations driven by standard Brownian motion. In particular, Buckwar \cite{Buckwar2000}, \cite{Buckwar2004} investigated stochastic functional differential equations with distributed delay and proposed the $\theta$-Maruyama scheme for their numerical approximation. Strong convergence results were established under suitable Lipschitz-type conditions on the coefficients, highlighting the impact of the memory term on both the stability and convergence properties of the numerical method. 

From a numerical point of view, the approximation of solutions to stochastic differential equations driven by fractional Brownian motion has attracted increasing attention over the past decades. For equations without delay, several Euler- and Milstein-type schemes have been proposed and analyzed, producing strong convergence results and explicit convergence rates under different assumptions on coefficients and the Hurst parameter \cite{NeuenkirchNourdin2007, GyongyMillet2005}. In particular, in the case $H > 1/2$, the convergence of the Euler scheme was systematically studied in \cite{Mishura2008}, who proved strong convergence under appropriate regularity conditions using Riemann-Stieltjes integration and Hölder-type estimates. Later, Neuenkirch and Nourdin obtained further refinements and optimal convergence rates \cite{NeuenkirchNourdin2007}.

Despite these advances, most of the available numerical results concern stochastic differential equations without memory or, at most, with finite delay. The numerical analysis of stochastic functional differential equations with hereditary structures driven by fractional Brownian motion remains comparatively underdeveloped. The presence of memory significantly complicates both the analytical and numerical treatment of such systems, since the evolution of the solution depends on its past trajectory rather than on its current state alone. Existing work in this area has focused mainly on qualitative properties of solutions, such as existence, uniqueness, and stability, whereas rigorous convergence analyzes of numerical schemes and the derivation of explicit convergence rates have received less attention \cite{Caraballo2011}.

Motivated by these observations, the main objective of this paper is to propose and analyze a Euler type numerical scheme for equation (\ref{eq1-0}). In contrast to existing results for memoryless equations, such as those obtained in \cite{Mishura2008}, the proposed scheme explicitly accounts for the memory structure of the system. Under suitable regularity and growth conditions on the drift and diffusion coefficients, we establish strong convergence of the numerical method toward the exact solution and derive its corresponding rate of convergence. The analysis relies on tools from fractional calculus, Hölder-type estimates, and inequalities adapted to integrals with respect to fractional Brownian motion \cite{Biagini2008, HuNualart2009}.

The proof relies on pathwise Hölder estimates together with careful control of the delay term along the numerical trajectory.

The convergence of the Euler scheme for stochastic differential equations driven by fractional Brownian motion with Hurst parameter $H>1/2$ was studied by \cite{Mishura2008} by means of Riemann–Stieltjes integration and Hölder-type estimates; however, that analysis, as well as its subsequent refinements \cite{NeuenkirchNourdin2007}, is restricted to the memoryless case in which the coefficients depend only on the present state. The present work extends this framework to stochastic functional differential equations with distributed memory, where the drift and diffusion depend on a functional $Y(t)$ of the past trajectory over a finite delay window, modeled through a kernel $K$. Our scheme involves a discrete approximation of the delay functional via Riemann sums, which must be controlled simultaneously with the fractional noise; to this end, we develop new structural estimates that transfer the pathwise Hölder regularity of the Euler trajectory to the discretized delay term, and combine them with Zähle's generalized Stieltjes integral and fractional-calculus inequalities to obtain a priori bounds uniform in the mesh size. Under suitable Lipschitz, linear-growth, and Fréchet-differentiability assumptions, we prove strong convergence of the scheme and derive its explicit rate in terms of the discretization step $h$, providing  to the best of our knowledge the first rigorous convergence rate analysis of an Euler scheme for SFDEs driven by fBm with $H>1/2$ and distributed delay.

In order to guide the reader through the argument, the presentation is organized into three parts.
We first establish the existence and uniqueness of the solution under the stated assumptions and then derive the corresponding estimates required for the subsequent analysis (Theorem \ref{texistenciax}, Proposition \ref{cotasapriori}, Appendix \ref{Tec-Est-Sol}).  These estimates play a fundamental role, since they allow us to control the behavior of the exact solution.
In the second part  (Section \ref{sec:num}), we prove several key lemmas providing a priori estimates for the numerical method, which ensure the boundedness of the approximating sequence and constitute the main tool for the convergence analysis. 
In the third part (Section \ref{convergence}), we use these bounds to prove the main theorem of the paper, which establishes the convergence of the scheme under the assumptions imposed on the coefficients. Each step is written in such a way that the role of the hypotheses becomes explicit, which will be important for possible generalizations of the model. In order to keep the exposition clear and readable, some technical proofs are postponed to the appendices. In particular, the proofs of the auxiliary lemmas used in the convergence arguments, as well as the a priori estimate for the exact solution of the equation, are presented in the supplementary section. This organization allows us to emphasize the main ideas of the proof while keeping the technical details available to the interested reader.

The remainder of the paper is organized as follows. Section~ \ref{SFDE} introduces the main preliminaries and key results on stochastic integration that will be used throughout the paper. In Section~ \ref{Seuler}, we present the Euler scheme for SFDE \eqref{eq1}.
Section~\ref{sec:num} is devoted to the numerical analysis, where we study the convergence rate of the Euler scheme for SFDEs.
Finally, in Section \ref{numerical} numerical simulations are provided to illustrate the theoretical results and assess the accuracy and efficiency of the proposed numerical scheme.

\section{Analytical Framework for the SFDE}
\label{SFDE}

In this section, we introduce the analytical framework used throughout the paper. 
We first fix the functional setting and the notion of stochastic integration with respect to fractional Brownian motion, and then present the stochastic functional differential equation under consideration together with the assumptions on its coefficients. 
These elements will provide the basis for both the well-posedness of the model and the subsequent numerical analysis.

Throughout the paper, we work on a complete probability space 
$(\Omega, \mathcal{F}, \{\mathcal{F}_t\}_{t\ge 0}, P)$ with filtration 
$\{\mathcal{F}_t\}_{t\ge 0}$ satisfying the usual conditions.

For $\lambda\in (0,1)$, let $C^\lambda([a, b])$ be the space of continuous functions 
$\lambda$ - H\"older $f : [a, b] \to \mathbb{R}$. If $f\in C^\lambda$, we define
\[
\|f\|_\lambda:= \|f\| + \sup_{a\leq s< t \leq b} 
\frac{|f(t) - f (s)|}{(t - s)^{\lambda}},
\qquad 
\|f\|:= \sup_{a\leq  t \leq b} |f(t)|.
\]

These Hölder spaces naturally arise in the analysis of equations driven by fractional Brownian motion, since the trajectories of $B^H$ and of the solution itself are controlled in Hölder norms.

We denote by $C_r=C([-r, 0])$ the space of continuous functions 
$\xi : [-r, 0] \to \mathbb{R}$. 
For a functional coefficient $\sigma(t, x, \xi)$, we define 
$D_\xi \sigma(t, x, \xi_0)$ as the derivative Fr\'echet of 
$\sigma(t, x,\cdot)$ in the variable $\xi=\xi_0$.

The use of Fréchet derivatives reflects the functional nature of the model, where the coefficients depend on entire segments of the past trajectory rather than on pointwise values only.

\subsection{The Riemann--Stieltjes integral}

We briefly recall the definition of the pathwise stochastic integral with respect to fractional Brownian motion, which will be used throughout the paper. 
In the regime $H>1/2$, this integral can be defined deterministically using fractional calculus, which allows us to exploit the Hölder regularity of the trajectories.

The generalized Riemann--Stieltjes integral, introduced by Z\"ahle in 
\cite{Zahle1998}, with respect to fBm $B^H$ is defined by
\begin{equation}
\label{eqdefintegral}
\int_a^b f dB^H
=
\int_a^b
(D_{a+}^{\beta}f)(s)
(D_{b-}^{1-\beta}B^H_{b-})(s)
ds,
\end{equation}
where $D_{a+}^{\beta}$ and $D_{b-}^{1-\beta}$ are the left and right sided 
fractional derivatives of orders $\beta$ and $1 -\beta$, respectively, given by
\begin{equation*}
(D_{a+}^{\beta}f)(s)
=
\frac{1}{\Gamma(1-\beta)}
\left[
\frac{f(s)}{(s-a)^{\beta}}
+
\beta\int_a^s
\frac{f(s)-f(u)}{(s-u)^{1+\beta}}
du
\right]
I_{(a,b)}(s),
\end{equation*}
and  
\begin{equation*}
(D_{b-}^{1-\beta}B^H_{b-})(s)
=
\frac{\exp^{-i\pi\beta}}{\Gamma(\beta)}
\left[
\frac{B^H_{b-}(s)}{(b-s)^{1-\beta}}
+
(1-\beta)\int_s^b
\frac{B^H_{b-}(s)-B^H_{b-}(u)}{(u-s)^{2-\beta}}
du
\right]
I_{(a,b)}(s),
\end{equation*}
with the convention
\[
B^H_{b-}(s)
=
(B^H_s-B^H_b)I_{(a,b)}(s).
\]

If $f\in C^\nu([a,b])$ for $\nu+H>1$, it is shown e.g. in 
\cite{Nualart2003} that the pathwise fractional integral 
(\ref{eqdefintegral}) exists for any $\beta\in (1-H, \nu)$. 
Moreover, for $\beta<1/2$, the following estimate holds:
\begin{equation}
\label{eqcotaintegral}
\left|\int_a^b f dB^H\right|
\leq 
C(\omega)
\left[
\int_a^b\frac{|f(s)|}{(s-a)^{\beta}}ds
+
\int_a^b\int_a^s
\frac{|f(s)-f(u)}{(s-u)^{\beta+1}}
duds
\right],
\end{equation}
where
\begin{equation}
\label{eqdefc0}
C(\omega)
=
C\cdot 
\sup_{a<s<b}
|D_{b-}^{1-\beta}B^H_{b-}(s)|,
\end{equation}
\(
1-H < \beta < \frac{1}{2}.
\)

Estimate \eqref{eqcotaintegral} will play a central role in the sequel, as it allows us to control stochastic integrals in terms of Hölder-type norms of the integrand. This type of bound is the key analytical tool that replaces classical Itô isometry arguments in the fractional setting.

\subsection{The stochastic functional differential equation }\label{sec:SFDE}

We study a class of stochastic functional differential equations with memory driven by fractional Brownian motion. We consider a SFDE of the form 
\begin{equation}\label{eq1}
X(t)=\left\{
\begin{array}{lc}
X(0) + \displaystyle\int_0^t b(s,X(s),Y(s))\,{d}s
+ \displaystyle\int_0^t \sigma(s,X(s),Y(s))\,{d}B^H(s),     &  t \in [0,T] \\
\phi(t),     & t \in [-r,0]
\end{array}
 \right.
\end{equation}
where $B^H=(B^H(t))_{t \ge 0}$ is a fractional Brownian motion with Hurst parameter $H> 1/2$. The stochastic integral with respect to $B^H$ is understood in the sense of the generalized Lebesgue-Stieltjes integral as introduced by Zähle \cite{Zahle1998}. The drift and diffusion coefficients $b,\sigma : [0,T] \times \mathbb{R} \times \mathbb{R} \to \mathbb{R}$ are assumed to be continuous functions, and the initial condition $\phi : [-r,0] \to \mathbb{R}$ is assumed to be Hölder continuous.
The memory effect is modeled through the functional $Y(t)$ defined by
\begin{equation}\label{eq2}
Y(t) = \int_{-r}^0 K(t,s,X(t+s))\,{d}s
     = \int_{t-r}^t K(t,s-t,X(s))\,{d}s,
\qquad t \ge 0,
\end{equation}
where the kernel $K : [0,T] \times [-r,0] \times \mathbb{R} \to \mathbb{R}$ is a continuous function.

We consider the following assumptions on the coefficients $b$, $\sigma$, the initial condition $\phi$ and the Kernel $K$ in (\ref{eq1}) and (\ref{eq2}).

The assumptions imposed throughout the paper are standard in the analysis of stochastic functional differential equations, but we briefly comment on their role and necessity. The Lipschitz-type conditions guarantee the well-posedness of the equation and allow us to obtain uniform bounds, while the growth conditions ensure that the solution does not explode in finite time. \\
The regularity assumptions on the coefficients are also required in order to apply the stochastic calculus tools used in the proofs, in particular when dealing with the memory term and the fractional noise. 

{\bf Hypothesis $\phi$} 
There exists a positive constant $L$ such that  
\begin{align}
    |\phi(t) - \phi(s)| &\leq L|t-s|^{H-\rho} \label{eqphi}, \quad 0 <\rho <H.
\end{align}

{\bf Hypothesis 1} 
There exist positive constants $L_1, L_2$ and $L_3$ such that the following properties hold.
\begin{align}
    |b(t,x,y)-b(s,x,y)|+|\sigma(t,x,y)-\sigma(s,x,y)| &\leq L_1|t-s|. \label{eq3}\\
    |b(t,x_1,y_1)-b(t,x_2,y_2)|+|\sigma(t,x_1,y_1)-\sigma(t,x_2,y_2)| &\leq L_2(|x_1-x_2|+|y_1-y_2|). \label{eq4}\\
    |b(t,x,y)|+|\sigma(t,x,y)| &\leq L_3(1+|x|+|y|). \label{eq5}
\end{align}

{\bf Hypothesis 2}
The function $\sigma(t, \xi)$ is Fréchet differentiable in the variable $\xi$. Also, there exist positive constants $M_1, M_2$ and $M_3$ such that for all $\xi, \eta \in C_r$ and $t\in [0, T]$, 
\begin{align}
    |D_\xi\sigma(t, \xi)|_{\mathcal{L}(C_r, \mathbb{R})}   &\leq M_1. \nonumber \\
 |D_\xi\sigma(t, \xi) - D_\xi\sigma(t, \eta)|_{\mathcal{L}(C_r, \mathbb{R})}   &\leq M_2\|\xi-\eta\|. \nonumber \\
    |D_\xi\sigma(t, \xi) - D_\xi\sigma(s, \xi)|_{\mathcal{L}(C_r, \mathbb{R})}   &\leq M_3|t-s|. \label{HM2}
\end{align}

{\bf Hypothesis 3} 
The function $K(t,s, \xi)$ is Fr\'echet differentiable in the variable $\xi$. Moreover, there exist positive constants $K_1, K_2$ and $K_3$ such that
\begin{align}
    |K(t,s,x)-K(u,v,x)| +  &\leq K_1(|t-u|+|s-v|). \nonumber \\
    |K(t,s,x)-K(t,s,y)| &\leq K_2|x-y|. \nonumber \\
    |K(t,s,x)| &\leq K_3(1+|x|). \label{HK1}
\end{align}

The following theorem establishes the existence and uniqueness of the solution to (\ref{eq1}) (the proof is analogous to those given in \cite{BH2011} and \cite{DonW2016}), and is therefore left to the reader. 

\begin{theorem}
\label{texistenciax}
Under hypotheses \ref{eqphi} to \ref{HK1}. If $\phi$ is a stochastic process whose trajectories belong to the space $C^{H-\rho}[0,T]$ a.s. for any $0<\rho<H$, then Equation (\ref{eq1}) has a unique solution $X\in C^{H-\rho}[0,T]$ a.s.,
\begin{equation}\label{deltaX}
     |X(u)-X(v)|\leq C(\omega)|u-v|^{H-\rho}.
\end{equation}
\end{theorem}

We obtain the next a priori estimate of the process $X$, the proof is given in Appendix \ref{Tec-Est-Sol}.  
 Such bounds will be crucial in controlling both the exact solution and the error terms arising in the discretization scheme.

\begin{prop}\label{cotasapriori}
 Under Hypotheses 1 and 3, for any $\varepsilon>0$ and $0<\rho<H$ there exists $h_0>0$ and $\Omega_{\varepsilon,h_0,\rho}\subset \Omega$ such that $\mathbb{P}(\Omega_{\varepsilon,h_0,\rho})>1-\varepsilon$ and for each $\omega\in\Omega_{\varepsilon,h_0,\rho}$, $h<h_0$
     $$|{X}(t)|< C(\omega)$$
     and for $0\leq s < t\leq T$,
     $$|{X}(t)-{X}(s)|\leq C(\omega)(t-s)^{H-\rho}.$$
\end{prop}

\subsection{Structural Estimates for the Delay Functional and Fractional Integral}

In this subsection, we establish several structural estimates that will be used repeatedly in the sequel. 
The first result concerns the regularity properties of the delay functional $Y(t)$ defined in (\ref{eq2}), 
while the second recalls the fundamental bound associated with the fractional integral with respect to $B^H$. 
Isolating these estimates allows us to streamline subsequent arguments and to avoid repeating technical computations. The proofs are deferred to the Appendix \ref{Aux-lemmas}.

The following lemma shows how the delay functional $Y(t)$ inherits regularity from the process $X$, while also capturing the contribution of the past trajectory over the memory window.

\begin{lemma}[Structural estimate for the delay functional]\label{lemmaY}
Let $Y$ be defined by (\ref{eq2}) and assume Hypothesis 3 holds. 
Then for all $0\le s<t\le T$,
\[
|Y(t)-Y(s)|
\le C_1 |t-s|
+ K_3 \int_{s-r}^{t-r} |X(u)|\,du.
\]
\end{lemma}


\section{Euler scheme for SFDE}\label{Seuler}

In this section, we introduce the Euler approximation of the SFDE and describe how the fractional noise and the memory term are discretized. Particular attention is paid to the delay functional, whose approximation involves the past trajectory of the numerical solution.

We assume, without loss of generality, that there exists $N_T \in \mathbb{N}$ such that $T = N_T r$. 
Let $h := T/N$ for some $N \in \mathbb{N}$ and consider the uniform partition of $[-r,T]$ given by
\[
t_n = nh, \qquad n = -N, \dots, -1, 0, 1, \dots, N\cdot N_T.
\]
For simplicity, we take $T=r$, so that $N_T=1$.

\medskip

The Euler scheme associated with \eqref{eq1} is defined recursively by
\begin{equation*}\label{euler}
\begin{cases}
\Tilde{X}^h(t_n) = \phi(t_n), & n = -N, \dots, 0, \\
\Tilde{X}^h(t_{n+1})
= \Tilde{X}^h(t_n)
+ h\, b(t_n,\Tilde{X}^h(t_n),\Tilde{Y}(t_n))
+ \sigma(t_n,\Tilde{X}^h(t_n),\Tilde{Y}(t_n)) \Delta B^H(t_n),
& n = 0,\dots,N-1,
\end{cases}
\end{equation*}
where
\[
\Delta B^H(t_n) = B^H(t_{n+1}) - B^H(t_n).
\]

In contrast to the classical Markovian setting, the scheme depends not only on the current value $\Tilde{X}^h(t_n)$ but also on the discrete approximation of the delay functional $\Tilde{Y}(t_n)$, which encodes the past trajectory of the process.
\medskip

The continuous interpolation is defined  by $\Tilde{X}^h(t)$ for $t\in[-r, 0]$, and for $t \in [t_n,t_{n+1})$ by
\begin{align*}
\Tilde{X}^h(t)
&= \Tilde{X}^h(t_n)
+ b(t_n,\Tilde{X}^h(t_n),\Tilde{Y}(t_n))(t-t_n) + \sigma(t_n,\Tilde{X}^h(t_n),\Tilde{Y}(t_n))(B^H(t)-B^H(t_n)).
\end{align*}

This continuous-time interpolation allows us to rewrite the discrete scheme in an integral form, which is more suitable for comparison with the exact solution and for the derivation of error estimates.

Then for $t\ge0$,
\begin{equation}\label{EulerInterpolado}
\Tilde{X}^h(t)
\!=\!
X(0)
+
\int_0^t b(C_h(u),\Tilde{X}^h(C_h(u)),\Tilde{Y}(C_h(u)))\,du
+
\int_0^t \sigma(C_h(u),\Tilde{X}^h(C_h(u)),\Tilde{Y}(C_h(u)))\,dB^H(u),
\end{equation}
where
\begin{equation}\label{Chu}
C_h(u)=\left\lfloor \frac{u}{h} \right\rfloor h.
\end{equation}
Clearly, if $u \in [t_i,t_{i+1})$ then $C_h(u)=t_i$.

The operator $C_h(u)$ selects the last discretization point before $u$, ensuring that the coefficients are evaluated along the piecewise constant approximation of the trajectory.

\medskip
We define a discrete approximation of the delay functional $Y(t)$, given by \eqref{eq2}, by means of Riemann sums, as follows
\begin{equation}\label{AproxY}
\Tilde{Y}(t)
:= \sum_{i=-N}^{-1} h\, K(t,t_i,\Tilde{X}^h(t+t_i)).
\end{equation}

Equivalently,
\begin{align*}\label{AproxY2}
\Tilde{Y}(t)
&= \sum_{i=-N}^{-1}
\int_{t_i}^{t_{i+1}}
K(t,C_h(s),\Tilde{X}^h(t+C_h(s)))\,ds= \int_{-r}^{0}
K(t,C_h(s),\Tilde{X}^h(t+C_h(s)))\,ds.
\end{align*}

This representation makes explicit the link between the exact delay term and its discrete counterpart, and will be instrumental in the convergence analysis, where the discrepancy induced by the discretization of both the kernel and the trajectory must be carefully controlled.


\section{Estimates for the numerical scheme}\label{sec:num}

In this section, we derive a priori estimates for the numerical scheme that will be used in the convergence analysis.

First, we introduce some technical lemmas needed to estimate $|\Tilde{X}^h(t)|$ for $0\leq t\leq T$. 

We define the function $\Tilde{F}: [-r,T]\rightarrow \mathbb{R}$ by
\begin{equation}
\label{Ftilde}
    \Tilde{F}(t)=\sup_{u\in[-r,t]}|\Tilde{X}^h(u)|.
\end{equation}

The following result is the discrete structural counterpart of 
Lemma \ref{lemmaY}. It shows that the Euler approximation 
preserves the incremental structure of the delay functional 
uniformly with respect to the mesh size.  For clarity, we defer the proofs of the lemmas to the Appendix \ref{Aux-lemmas}.

\begin{lemma}\label{cotadeYtilde-Xtilde}
Let $\Tilde{Y}(t)$ be the discrete approximation of the integral $Y(t)$, given by \eqref{AproxY}. Then
\begin{equation*}
    |\Tilde{Y}(C_h(u))| \leq rK_3(1+\Tilde{F}(u-h)).
\end{equation*}
\end{lemma}

Observe that the constants in Lemma \ref{cotadeYtilde-Xtilde} are independent of $h$.
Hence the discrete delay functional satisfies the same structural
increment bound uniformly in the discretization parameter.

The following lemma gives a relation between $\Delta {\Tilde{Y}}$ and $\Delta {\Tilde{X}}$. This relation will be used to transfer regularity estimates from $\Tilde{X}^h$ to the delay term.

\begin{lemma}\label{cotadeDYtilde-DXtilde}
Let $\Tilde{Y}$ be the approximation defined by \eqref{AproxY}. Then
\begin{equation*}
    |\Tilde{Y}(u)-\Tilde{Y}(v)|  \leq rK_1|u-v| + K_2 \int_{-r}^{0} |\Tilde{X}^h(u+C_h(s))-\Tilde{X}^h(v+C_h(s))|{d}s,
\end{equation*}
for $0\leq v\leq u\leq T$.
\end{lemma}

We next control the local discretization error between grid points.

\begin{lemma}\label{cota-Xtilde}
Let $\Tilde{X}^h(t)$ be the Euler scheme given by \eqref{EulerInterpolado}. Then, for $v \in [0,T]$
\begin{equation*}
     |\Tilde{X}^h(v)-\Tilde{X}^h(C_h(v))| \leq C(\omega) \left(1+|\Tilde{X}^h(C_h(v))|+rK_3(1+\Tilde{F}(v-h))\right)(v-C_h(v))^{H-\rho},
\end{equation*}
where $C(\omega)$ is  a constant independent of $h$.
\end{lemma}

\begin{remark}\label{u-cu-beta}
For all $0 \le i \le N-1$ and $C_h$ defined in \eqref{Chu}, we have $
\int_{t_i}^{t_{i+1}} (u-C_h(u))^{-\beta}du
=
\frac{1}{1-\beta}h^{1-\beta}.
$
Consequently,
\begin{align}\label{u-chu}
\int_0^t (u-C_h(u))^{-\beta}du
&\le \sum_{i=0}^{N-1} \int_{t_i}^{t_{i+1}} (u-C_h(u))^{-\beta}du =\frac{N}{1-\beta}h^{1-\beta}
= \frac{T}{1-\beta}h^{-\beta}.
\end{align}
\end{remark}

\begin{lemma}\label{X-CHU-v}
Let $\Tilde{X}^h(t)$ be the Euler scheme given by \eqref{EulerInterpolado}. Then, for    $0\leq v \leq C_h(u)$, 
\begin{align*}
\label{eq14-0}
&\int_0^{C_h(u)} |\Tilde{X}^h(C_h(u))-\Tilde{X}^h(v)|(u-v)^{-\beta-1}{d}v \\
 \leq &C(\omega)\left(1 + \int_0^{C_h(u)} |\Tilde{X}^h(w)|(u-w)^{-2\beta}{d}w + \int_0^{C_h(u)} \Tilde{F}(w)(u-w)^{-2\beta}{d}w \right. \notag\\
 &+ \int_0^{C_h(u)} (u-w)^{-\beta}\int_0^{C_h(w)} \dfrac{|\Tilde{X}^h(C_h(w))-\Tilde{X}^h(z)|}{(w-z)^{\beta+1}}{d}z{d}w + \Tilde{F}(C_h(u))h^{H-\rho-\beta} \notag\\
 & \left. + \int_0^{C_h(u)} (u-w)^{-\beta}\int_0^{C_h(w)} \int_{-r}^0 \dfrac{ |\Tilde{X}^h(C_h(w)+C_h(s))-\Tilde{X}^h(C_h(z)+C_h(s))|{d}s}{(w-z)^{\beta+1}}{d}z{d}w\right).\notag
\end{align*}
\end{lemma}

\begin{lemma}\label{Cota- I(w)}
Let $\Tilde{X}^h(t)$ be the Euler scheme given by \eqref{EulerInterpolado} and  for $w \in [0,T]$ define
\[
I(w):=\int_{0}^{C_h(w)}(w-z)^{-\beta-1}\int_{-r}^{0}
\big|\widetilde X^h(C_h(w)+C_h(s))-\widetilde X^h(C_h(z)+C_h(s))\big|dsdz.
\]  
Then, for $w \in [0,T]$
\begin{align*}
 I(w) \leq & C_1(\omega) \left( 1+ \widetilde F(C_h(w))  +  h ^{H - \rho}(w- C_h(w))^{-\beta}  +  h  (w-C_h(w))^{-\beta}  \widetilde F(C_h(w)) \right. \nonumber \\
& + \left.
\int_0^{C_h(w)} (w-z)^{-\beta-1} \int_{-C_h(z)}^{0}
\big|\widetilde X^h(C_h(w)+C_h(s))-\widetilde X^h(C_h(z)+C_h(s))\big| dsdz
\right).
\end{align*}
\end{lemma}

\begin{lemma}\label{dif-dif}
Let $\Tilde{X}^h(t)$ be the Euler scheme given by \eqref{EulerInterpolado}, $z, w$ such that $0<z<C_h(w)$, and  $s \in [-r,0]$, then

\begin{align*}\label{DIFF-2}
&    \int_0^{C_h(w)}\int_{-C_h(z)}^0 \dfrac{ |\tilde{X}^h(C_h(w)+C_h(s))-\tilde{X}^h(C_h(z)+C_h(s))|}{(w-z)^{\beta+1}} {d}s{d}z  \nonumber \\
\leq &C(\omega)\left( 1 + h^{1 - \beta} + h^{H-\rho -\beta} + h^{H-\rho -\beta}\tilde{F}(w)  +  h^{1-\beta}  \widetilde F(C_h(w))  +  \widetilde F(C_h(w)) w^{1-\beta} 
 + \int_{0}^{C_h(w)} \frac{|\Tilde{F}^h(C_h(z))|}{(w-z)^{2\beta}}{d}z  \right. \nonumber \\
& + \int_{0}^{C_h(w)} \int_{0}^{C_h(u)}
(w-u)^{-\beta}\,
\frac{\big|\tilde{X}^h(C_h(u))-\tilde{X}^h(v)\big|}{(u-v)^{\beta+1}}
dvdu \nonumber \\ 
&  + \int_{0}^{C_h(w)} (w-u)^{-\beta}\, 
\int_0^{C_h(u)} (u-v)^{-\beta-1} \left. \int_{-C_h(v)}^{0}
\big|\widetilde X^h(C_h(u)+C_h(s))-\widetilde X^h(C_h(v)+C_h(s))\big| dsdv du\right).
\end{align*}

\end{lemma}
We now establish uniform bounds for the Euler scheme. These estimates are the key ingredient for the convergence analysis and rely on the previous lemmas.

\begin{theorem}
\label{teoCotasapriori}
According to Hypotheses 1 and 3, for any $\varepsilon>0$ and $0<\rho<H$ there exist $h_0>0$ and $\Omega_{\varepsilon,h_0,\rho}\subset \Omega$ such that $\mathbb{P}(\Omega_{\varepsilon,h_0,\rho})>1-\varepsilon$, and for each $\omega\in\Omega_{\varepsilon,h_0,\rho}$, $h<h_0$ and $C(\omega)>0$, such that
     $$|\Tilde{X}^h(t)|< C(\omega)$$
     and for $0\leq s < t\leq T$,
     $$|\Tilde{X}^h(t)-\Tilde{X}^h(s)|\leq C(\omega)(t-s)^{H-\rho}.$$
     
\end{theorem}

\begin{proof}
From the Euler scheme \eqref{EulerInterpolado} we have the following,
\[
\Tilde{X}^h(t)
=
X(0)
+
\int_0^t b(C_h(u),\Tilde{X}^h(C_h(u)),\Tilde{Y}(C_h(u)))\,du
+
\int_0^t \sigma(C_h(u),\Tilde{X}^h(C_h(u)),\Tilde{Y}(C_h(u)))\,dB^H_u.
\]
Applying inequality \ref{eqcotaintegral} to the function
$
f(u)
=
\sigma(C_h(u),\Tilde{X}^h(C_h(u)),\Tilde{Y}(C_h(u))),
$
we obtain 
\begin{align*}
|\Tilde{X}^h(t)|\leq &  |X(0)|\! + \!\int_0^t \!|b(C_h(u),\Tilde{X}^h(C_h(u)),\Tilde{Y}(C_h(u)))| {d}u \!+\! C_{\beta}C(\omega)\!\!\int_0^t\! \dfrac{|\sigma(C_h(u),\Tilde{X}^h(C_h(u)),\Tilde{Y}(C_h(u)))|}{u^{\beta}}{d}u \\
&\!+\! C_{\beta}C(\omega)\!\!\int_0^t \! \int_0^u \!\dfrac{|\sigma(C_h(u),\Tilde{X}^h(C_h(u)),\Tilde{Y}(C_h(u)))-\sigma(C_h(v),\Tilde{X}^h(C_h(v)),\Tilde{Y}(C_h(v)))|}{(u-v)^{\beta+1}}{d}v{d}u.
\end{align*}

Due to the linear growth property of $b$ and $\sigma$, the Lipschitz property of $\sigma$ (see Hypothesis 1 \eqref{eqdefc0}, \eqref{eq3}. and \eqref{eq4})  we derive 
\begin{align*}
|\Tilde{X}^h(t)|  \leq & |X(0)| + \int_0^t L_3(1+|\Tilde{X}^h(C_h(u))|+|\Tilde{Y}(C_h(u))|){d}u \\
&+ C_{\beta}C(\omega)\int_0^t \dfrac{L_3(1+|\Tilde{X}^h(C_h(u))|+|\Tilde{Y}(C_h(u))|)}{u^{\beta}}{d}u + C_{\beta}C(\omega)\int_0^t\int_0^u \dfrac{L_1|C_h(u)-C_h(v)|}{(u-v)^{\beta+1}}{d}v{d}u \\
&+ C_{\beta}C(\omega)\int_0^t\int_0^u \dfrac{L_2|\Tilde{X}^h(C_h(u))-\Tilde{X}^h(C_h(v))|+L_2|\Tilde{Y}(C_h(u))-\Tilde{Y}(C_h(v))|}{(u-v)^{\beta+1}}{d}v{d}u.
\end{align*}

Using Lemma \ref{cotadeYtilde-Xtilde} to control the delay term,
and splitting the inner integral over $[0,u]$ as
\(
[0,u] = [0,C_h(u)] \cup [C_h(u),u],
\)
we obtain
\begin{align*}
|\Tilde{X}^h(t)|
\leq & |X(0)|\!
+\! L_3 t
+ L_3\int_0^t |\Tilde{X}^h(C_h(u))|\,du
+ L_3 rK_3 t
+ L_3 rK_3 \int_0^t \Tilde{F}(u-h)\,du \\
&\!+ \!C_\beta C(\omega)L_3 \!\int_0^t \! \frac{du}{u^\beta}
+ C_\beta C(\omega)L_3 \!\int_0^t \!
\frac{|\Tilde{X}^h(C_h(u))|}{u^\beta}\,du + C_\beta C(\omega)L_3 rK_3\int_0^t 
\frac{1+\Tilde{F}(u-h)}{u^\beta}\,du \\
&\!+ \! C_\beta C(\omega)\! \int_0^t \! \int_0^u \! 
\frac{L_1|C_h(u)-C_h(v)|}{(u-v)^{\beta+1}}\,dv\,du \\
&\!+\! C_\beta C(\omega)\! \int_0^t \! \int_0^u \!
\frac{L_2|\Tilde{X}^h(C_h(u))-\Tilde{X}^h(C_h(v))|
+L_2|\Tilde{Y}(C_h(u))-\Tilde{Y}(C_h(v))|}
{(u-v)^{\beta+1}}\,dv\,du.
\end{align*}

For $v\in[C_h(u),u]$ we have $C_h(v)=C_h(u)$, so the integral over $[C_h(u),u]$ vanishes. Collecting the previous bounds, and using $
\int_0^t u^{-\beta}du
\leq C {t^{1-\beta}},
$
we have that
\begin{align*}
&|\Tilde{X}^h(t)|
\leq  |X(0)|
+ L_3(1+rK_3)\left(T
+ C_\beta C(\omega)\frac{T^{1-\beta}}{1-\beta}\right)
+ L_3\int_0^t |\Tilde{X}^h(C_h(u))|\,du
+ L_3 rK_3 \int_0^t \Tilde{F}(u-h)\,du
\nonumber\\
&+ C_\beta C(\omega)L_3\int_0^t 
\frac{|\Tilde{X}^h(C_h(u))|}{u^\beta}\,du
+ C_\beta C(\omega)L_3 rK_3\int_0^t 
\frac{\Tilde{F}(u-h)}{u^\beta}\,du
\nonumber\\
&+ C_\beta C(\omega)
\int_0^t\int_0^{C_h(u)}
\frac{
L_1|C_h(u)-C_h(v)|
+L_2|\Tilde{X}^h(C_h(u))-\Tilde{X}^h(C_h(v))|
+L_2|\Tilde{Y}(C_h(u))-\Tilde{Y}(C_h(v))|
}{(u-v)^{\beta+1}}\,dv\,du.
\end{align*}


Therefore,
\begin{align}
\label{cotaXht}
|\Tilde{X}^h(t)|&\leq |X(0)| + L_3(1+rK_3)\left(T + C_{\beta}C(\omega )\dfrac{T^{1-\beta}}{1-\beta}\right) + L_3(T^{\beta} + C_{\beta}C(\omega ))\int_0^t \dfrac{|\Tilde{X}^h(C_h(u))|}{u^{\beta}}{d}u \nonumber \\
&\,+ L_3rK_3(T^{\beta} + C_{\beta}C(\omega ))\int_0^t \dfrac{\Tilde{F}(u-h)}{u^{\beta}}{d}u + C_{\beta}C(\omega )\int_0^t\int_0^{C_h(u)} \dfrac{L_1|C_h(u)-C_h(v)|}{(u-v)^{\beta+1}}{d}v{d}u \nonumber \\
&\,+ \!C_{\beta}C(\omega ) \!\! \int_0^t \! \int_0^{C_h(u)} \!\!\dfrac{L_2|\Tilde{X}^h(C_h(u))-\Tilde{X}^h(v)|}{(u-v)^{\beta+1}}{d}v{d}u \!+\!\! C_{\beta}C(\omega )\!\int_0^t \! \int_0^{C_h(u)} \!\!\dfrac{L_2|\Tilde{X}^h(v)-\Tilde{X}^h(C_h(v))|}{(u-v)^{\beta+1}}{d}v{d}u\nonumber \\
&\,+ C_{\beta}C(\omega )\int_0^t\int_0^{C_h(u)} \dfrac{L_2|\Tilde{Y}(C_h(u))-\Tilde{Y}(C_h(v))|}{(u-v)^{\beta+1}}{d}v{d}u \nonumber \\
&= C'(\beta,\omega)( 1+I_1 +  I_2 + I_3 + I_4 + I_5 +  I_6),
\end{align}
where $C'(\beta,\omega)=(L_1+L_2+L_3(1+rK_3))(1+T^{\beta})(1+C_{\beta} C(\omega))$, and
\begin{align*}\label{Is}
    I_1&=\int_0^t \dfrac{|\Tilde{X}^h(C_h(u))|}{u^{\beta}}{d}u,\quad 
    I_2=\int_0^t \dfrac{\Tilde{F}(u-h)}{u^{\beta}}{d}u,\quad
    I_3=\int_0^t\int_0^{C_h(u)} \dfrac{|C_h(u)-C_h(v)|}{(u-v)^{\beta+1}}{d}v{d}u, \nonumber \\
    I_4&=\int_0^t\int_0^{C_h(u)} \dfrac{|\Tilde{X}^h(C_h(u))-\Tilde{X}^h(v)|}{(u-v)^{\beta+1}}{d}v{d}u,\quad
    I_5=\int_0^t\int_0^{C_h(u)} \dfrac{|\Tilde{X}^h(v)-\Tilde{X}^h(C_h(v))|}{(u-v)^{\beta+1}}{d}v{d}u,\nonumber \\
    I_6&=\int_0^t\int_0^{C_h(u)} \dfrac{|\Tilde{Y}(C_h(u))-\Tilde{Y}(C_h(v))|}{(u-v)^{\beta+1}}{d}v{d}u.
\end{align*}

To estimate $I_3$, we use $C_h(u)-C_h(v)\leq u-v+h$ if $u >v$, from \eqref{R6} and Remark  [\eqref{u-cu-beta} eq-\eqref{u-chu}], then is easy to see that there exists a constant $C_3>0$ depending only on $T$ and $\beta$
such that
    \begin{equation}
    \label{eqI3}
        I_3 \leq  \dfrac{1}{(1-\beta)(2-\beta)} t^{2-\beta} + \dfrac{h}{\beta} \int_0^t (u-C_h(u))^{-\beta}{d}u \leq  C_3
    \end{equation}

From Lemma \eqref{X-CHU-v}
\begin{align*}
I_{4}=&\int_0^t\int_0^{C_h(u)} |\Tilde{X}^h(C_h(u))-\Tilde{X}^h(v)|(u-v)^{-\beta-1}{d}v du \\
 \leq &C_4(\omega)\left(1  + \int_0^t \Tilde{F}(u)u^{1-2\beta}du  + \Tilde{F}(C_h(t))h^{H-\rho-\beta}  \right. \notag\\
 &+ \int_0^t\int_0^{C_h(u)} (u-w)^{-\beta}\int_0^{C_h(w)} \dfrac{|\Tilde{X}^h(C_h(w))-\Tilde{X}^h(z)|}{(w-z)^{\beta+1}}{d}z{d}w du \notag\\
 & \left. + \int_0^t\int_0^{C_h(u)} (u-w)^{-\beta}\int_0^{C_h(w)} \int_{-r}^0 \dfrac{ |\Tilde{X}^h(C_h(w)+C_h(s))-\Tilde{X}^h(C_h(z)+C_h(s))|{d}s}{(w-z)^{\beta+1}}{d}z{d}wdu\right).\notag
\end{align*}
From Lemma \eqref{Cota- I(w)}, \eqref{R6} and   Remark  [\eqref{u-cu-beta} eq-\eqref{u-chu}]
\begin{align}
&I_4\leq C_4(\omega)\left(1  + \int_0^t \Tilde{F}(u)u^{1-2\beta}du  + \Tilde{F}(C_h(t))h^{H-\rho-\beta}  \right. \notag\\
 &+ \int_0^t\int_0^{C_h(u)} (u-w)^{-\beta}\int_0^{C_h(w)} \dfrac{|\Tilde{X}^h(C_h(w))-\Tilde{X}^h(z)|}{(w-z)^{\beta+1}}{d}z{d}w du \notag\\
 & \left. + \int_0^t \int_0^{C_h(u)} (u-w)^{-\beta} \left[1+ \widetilde F(C_h(w))  +  h ^{H - \rho}(w- C_h(w))^{-\beta}  +  h  (w-C_h(w))^{-\beta}  \widetilde F(C_h(w)) \right] dw du\right.\notag \\
 & + \left. \int_0^t \int_0^{C_h(u)} (u-w)^{-\beta}  \int_0^{C_h(w)} (w-z)^{-\beta-1} \int_{-C_h(z)}^{0}
\big|\widetilde X^h(C_h(w)+C_h(s))-\widetilde X^h(C_h(z)+C_h(s))\big| dsdz dwdu 
 \right) \notag\\
\label{eqI4-1}
&\leq C_4(\omega)\left(1  + \int_0^t \Tilde{F}(u)u^{1-2\beta}du  + \Tilde{F}(C_h(t))h^{H-\rho-\beta}  + \int_0^t u^{1-\beta} \widetilde F(u) du  + h ^{H - \rho - \beta} + h^{1-\beta}    \widetilde F(t) \right. \notag\\
 &+ \int_0^t\int_0^{C_h(u)} (u-w)^{-\beta}\int_0^{C_h(w)} \dfrac{|\Tilde{X}^h(C_h(w))-\Tilde{X}^h(z)|}{(w-z)^{\beta+1}}{d}z{d}w du \notag\\
 & + \left. \int_0^t \int_0^{C_h(u)} (u-w)^{-\beta}  \int_0^{C_h(w)} (w-z)^{-\beta-1} \int_{-C_h(z)}^{0}
\big|\widetilde X^h(C_h(w)+C_h(s))-\widetilde X^h(C_h(z)+C_h(s))\big| dsdz dwdu 
 \right)
\end{align}
  For $I_5$, we use Lemma \ref{cota-Xtilde} and \eqref{Ftilde} to deduce
    \begin{align}\label{eqI5}
        I_5&\leq C_5(\omega)\int_0^t\int_0^{C_h(u)} \dfrac{ \left(1+|\Tilde{X}^h(C_h(v))|+rK_3(1+\Tilde{F}(v-h))\right)(v-C_h(v))^{H-\rho}}{(u-v)^{\beta+1}}{d}v{d}u \nonumber \\
        &\leq C_5(\omega) h^{H - \rho}  \int_0^t\int_0^{C_h(u)} \dfrac{1+\Tilde{F}(C_h(v))+rK_3(1+\Tilde{F}(v-h))}{(u-v)^{\beta+1}}{d}v{d}u \nonumber \\
         & \leq C_5(\omega)h^{H-\rho-\beta}   + C(\omega)  h^{H-\rho - \beta} \Tilde{F}(t). 
        \end{align}
We now analyze $I_6$, using Lemma \ref{cotadeDYtilde-DXtilde} and the upper bound of $I_3$ given in \eqref{eqI3}, we can get 

    \begin{align*}
        I_6 \leq & \int_{0}^t\int_{0}^{C_h(u)} \dfrac{rK_1|C_h(u)-C_h(v)| + K_2 \int_{-r}^{0} |\Tilde{X}^h(C_h(u)+C_h(s))-\Tilde{X}^h(C_h(v)+C_h(s))|{d}s}{(u-v)^{\beta+1}}{d}v{d}u \nonumber \\
        \leq & rK_1I_3 + K_2\int_{0}^t\int_{0}^{C_h(u)} \int_{-r}^{0} \dfrac{|\Tilde{X}^h(C_h(u)+C_h(s))-\Tilde{X}^h(C_h(v)+C_h(s))|}{(u-v)^{\beta+1}}{d}s{d}v{d}u \nonumber \\
       \leq & C + K_2\int_{0}^t\int_{0}^{C_h(u)} \int_{-r}^{0} \dfrac{|\Tilde{X}^h(C_h(u)+C_h(s))-\Tilde{X}^h(C_h(v)+C_h(s))|}{(u-v)^{\beta+1}}{d}s{d}v{d}u.
    \end{align*}
Now, from Lemma  \eqref{Cota- I(w)} 
\begin{align*}
 I_6 \leq & C_6(\omega) \left( 1+ \int_0^t \widetilde F(C_h(u)) du  +  h ^{H - \rho} \int_0^t (u- C_h(u))^{-\beta} du  \right. +  h \int_0^t  (u-C_h(u))^{-\beta}   \widetilde F(C_h(u)) du  \nonumber \\
& + \left.
\int_0^t \int_0^{C_h(u)} (u-z)^{-\beta-1} \int_{-C_h(z)}^{0}
\big|\widetilde X^h(C_h(u)+C_h(s))-\widetilde X^h(C_h(z)+C_h(s))\big| dsdz du
\right).
\end{align*}
From  Remark [\eqref{u-cu-beta} eq-\eqref{u-chu}]
\begin{align*}
 I_6 \leq & C_6(\omega) \left( 1+ \int_0^t \widetilde F(C_h(u)) du  +  h ^{H - \rho - \beta} +  h^{1-\beta} \widetilde F(C_h(t)) \right.  \nonumber \\
& + \left.
\int_0^t \int_0^{C_h(u)} (u-z)^{-\beta-1} \int_{-C_h(z)}^{0}
\big|\widetilde X^h(C_h(u)+C_h(s))-\widetilde X^h(C_h(z)+C_h(s))\big| dsdz du
\right).
\end{align*}
From Lemma \eqref{dif-dif}
\begin{align}
 \label{eqI6}
& I_6 \leq C_6(\omega) \left( 1+ \int_0^t \widetilde F(C_h(u)) du  +  h ^{H - \rho - \beta} \widetilde F(C_h(t))  +  h^{1-\beta} \widetilde F(C_h(t)) +  \int_0^t \widetilde F(C_h(w)) w^{1-\beta}dw \right.  \nonumber \\
& + \int_0^t \int_{0}^{C_h(u)} \frac{|\Tilde{F}^h(C_h(z))|}{(u-z)^{2\beta}}{d}z du  + \int_0^t \int_{0}^{C_h(u)} \int_{0}^{C_h(w)}
(u-w)^{-\beta}\,
\frac{\big|\tilde{X}^h(C_h(w)-\tilde{X}^h(v)\big|}{(w-v)^{\beta+1}}
dvdwdu \nonumber \\ 
&  +  \left. \int_0^t \int_{0}^{C_h(u)} (u-w)^{-\beta}\, 
\int_0^{C_h(w)} (w-v)^{-\beta-1} \int_{-C_h(v)}^{0}
\big|\widetilde X^h(C_h(w)+C_h(s))-\widetilde X^h(C_h(v)+C_h(s))\big| dsdvdwdu\right)
\end{align}

Using \eqref{cotaXht}, \eqref{eqI3}, \eqref{eqI4-1}, \eqref{eqI5}, \eqref{eqI6}  \eqref{u-chu} and from \eqref{Ftilde} and Equation (8) in \cite{Mishura2008}, considering $H-\rho -\beta >0$,
 we can obtain
\begin{align}\label{eq13-1}
 \tilde{F}(t) \leq & C_0(\omega)\left(1 + h^{H-\rho - \beta} \tilde{F}(t) + \int_0^t \dfrac{|\tilde{X}^h(C_h(u))|}{u^{\beta}}\,du + \int_0^t\dfrac{\tilde{F}(u)}{u^{\beta}}\,du \right.  \\
& +  \int_0^t\!\!\int_{0}^{C_h(u)}\!\!\!\! \int_{0}^{C_h(w)}(u-w)^{-\beta}\,
\frac{\big|\tilde{X}^h(C_h(w))-\tilde{X}^h(v)\big|}{(w-v)^{\beta+1}}
dvdwdu 
\nonumber \\ 
&+\left.\!\!\int_0^t \!\!\int_{0}^{C_h(u)}\!\!\!\!\! (u-w)^{-\beta}\!\!\!
\int_0^{C_h(w)} \!\!\!\!\!(w-v)^{-\beta-1}\!\!\!\int_{-C_h(v)}^{0}\!\!\!\!
\big|\widetilde X^h(C_h(w)+C_h(s))-\widetilde X^h(C_h(v)+C_h(s))\big| dsdvdwdu  \right)\nonumber
\end{align}
By Lemma \ref{X-CHU-v} ,
\begin{align}\label{I1-bound}
\int_0^{C_h(t)} &
\frac{|\tilde{X}^h(C_h(t))-\tilde{X}^h(v)|}{(t-v)^{\beta+1}}\,dv 
\nonumber\\
 \le&
C_0(\omega)\Big(
1
+ h^{H-\rho-\beta}\tilde F(t)
+ \int_0^{C_h(t)} (t-w)^{-\beta}\tilde F(w)dw \nonumber\\
&+ \int_0^{C_h(t)} (t-w)^{-\beta}
\int_0^{C_h(w)}
\frac{|\tilde X^h(C_h(w))-\tilde X^h(z)|}{(w-z)^{\beta+1}}dz\,dw
\nonumber\\
&+ \int_0^{C_h(t)} (t-w)^{-\beta}
\int_0^{C_h(w)}
\int_{-C_h(z)}^0
\frac{|\tilde X^h(C_h(w)+C_h(s))-\tilde X^h(C_h(z)+C_h(s))|}
{(w-z)^{\beta+1}}ds\,dz\,dw
\Big).
\end{align}
From Lemma \ref{dif-dif},
\begin{align}\label{I2-bound}
&\int_0^{C_h(t)} \int_{-C_h(v)}^0 
\frac{|\tilde{X}^h(C_h(t)+C_h(s))-\tilde{X}^h(C_h(v)+C_h(s))|}
{(t-v)^{\beta+1}}
\,ds\,dv\nonumber\\
\le &
C_0(\omega)\Big(
1
+ h^{H-\rho-\beta}\tilde F(t)
+ \int_0^{C_h(t)} (t-u)^{-\beta}\tilde F(u)du
\\
&+ \int_0^{C_h(t)} (t-u)^{-\beta}
\int_0^{C_h(u)}
\frac{|\tilde X^h(C_h(u))-\tilde X^h(v)|}{(u-v)^{\beta+1}}dv\,du
\nonumber\\
&+ \int_0^{C_h(t)} (t-u)^{-\beta}
\int_0^{C_h(u)}
\int_{-C_h(v)}^0
\frac{|\tilde X^h(C_h(u)+C_h(s))-\tilde X^h(C_h(v)+C_h(s))|}
{(u-v)^{\beta+1}}ds\,dv\,du
\Big).
\nonumber
\end{align}

We define the function $\varphi(t)$ by
\begin{align}\label{phi}
\varphi(t):=&\tilde{F}(t)
+ \int_0^{C_h(t)} 
\frac{|\tilde{X}^h(C_h(t))-\tilde{X}^h(v)|}{(t-v)^{\beta+1}}\,dv
\nonumber\\
&+ \int_0^{C_h(t)} \int_{-C_h(v)}^0 
\frac{|\tilde{X}^h(C_h(t)+C_h(s))-\tilde{X}^h(C_h(v)+C_h(s))|}
{(t-v)^{\beta+1}}
\,ds\,dv.
\end{align}

Summing \eqref{eq13-1}, \eqref{I1-bound} and \eqref{I2-bound}, we obtain
\begin{align}
\varphi(t)
\le &
C_0(\omega)\Bigg(
1
+ h^{H-\rho-\beta}\tilde F(t)
+ \int_0^t u^{-\beta}\tilde F(u)du
+ \int_0^t (t-u)^{-\beta}\tilde F(u)du\nonumber\\
&+ \int_0^t u^{-\beta}
\int_0^{C_h(u)}
\frac{|\tilde X^h(C_h(u))-\tilde X^h(v)|}{(u-v)^{\beta+1}}dv\,du
+ \int_0^t (t-u)^{-\beta}
\int_0^{C_h(u)}
\frac{|\tilde X^h(C_h(u))-\tilde X^h(v)|}{(u-v)^{\beta+1}}dv\,du
\nonumber\\
&+ \int_0^t u^{-\beta}
\int_0^{C_h(u)}
\int_{-C_h(v)}^0
\frac{|\tilde X^h(C_h(u)+C_h(s))-\tilde X^h(C_h(v)+C_h(s))|}
{(u-v)^{\beta+1}}ds\,dv\,du
\nonumber\\
&+ \int_0^t (t-u)^{-\beta}
\int_0^{C_h(u)}
\int_{-C_h(v)}^0
\frac{|\tilde X^h(C_h(u)+C_h(s))-\tilde X^h(C_h(v)+C_h(s))|}
{(u-v)^{\beta+1}}ds\,dv\,du
\Bigg).
\nonumber
\end{align}

Hence
\begin{align*}
\varphi(t)
\le
C_0(\omega)\Big(
1
+ h^{H-\rho-\beta}\tilde F(t)
+ \int_0^t (u^{-\beta}+(t-u)^{-\beta})\varphi(u)du
\Big).
\end{align*}
By definition of $\varphi(t)$, we have \(
\tilde{F}(t)\le \varphi(t)
\). Then
\begin{align*}
\varphi(t)
\le
C_0(\omega)\Big(
1
+ h^{H-\rho-\beta} \varphi(t)
+ \int_0^t (u^{-\beta}+(t-u)^{-\beta})\varphi(u)du
\Big).
\end{align*}

Since $H-\rho-\beta>0$, we may choose $h_0>0$ and a set
$\Omega_{\varepsilon,h_0,\rho}$ with
$\mathbb P(\Omega_{\varepsilon,h_0,\rho})>1-\varepsilon$
such that for $h<h_0$ and $\omega\in\Omega_{\varepsilon,h_0,\rho}$,
\[
C_0(\omega)h^{H-\rho-\beta}\le \frac12.
\]
Therefore,
\begin{align*}
\varphi(t)
\le
C_0(\omega)\Bigg(1
+
\int_0^t (u^{-\beta}+(t-u)^{-\beta})\varphi(u)du\Bigg).
\end{align*}

The kernel $u^{-\beta}+(t-u)^{-\beta}$ is integrable on $[0,t]$
for $\beta<1$. By the fractional Gronwall inequality (see Equation (8) in \cite{Mishura2008}), there exists $C(\omega)>0$ such that
\[
\sup_{t\in[0,T]}\varphi(t)\le C(\omega).
\]

In particular,
\[
|\tilde X^h(t)|\le C(\omega),
\quad \forall t\in[0,T].
\]

Let $0\le s<t\le T$. Subtracting \eqref{EulerInterpolado}
at times $t$ and $s$, we obtain
\begin{align*}
\tilde X^h(t)-\tilde X^h(s)
&=
\int_s^t b(C_h(u),\tilde X^h(C_h(u)),\tilde Y(C_h(u)))\,du 
+
\int_s^t \sigma(C_h(u),\tilde X^h(C_h(u)),\tilde Y(C_h(u)))\,dB^H(u).
\end{align*}

We have established the uniform bound of $\tilde X^h$, and of $\tilde Y$ follows from Lemma \ref{cotadeYtilde-Xtilde}.

Because the coefficients $b$ and $\sigma$ have linear growth,
there exists $C'(\omega)>0$ such that for all $u\in[0,T]$,
\begin{align*}
    |b(C_h(u),\tilde X^h(C_h(u)),\tilde Y(C_h(u)))| &\leq L_3(1+|\tilde{X}^h(t)|+|\tilde{Y}(t)|) \leq C'(\omega) \\
    |\sigma(t,\tilde{X}^h(t),\tilde{Y}(t))| &\leq L_3(1+|\tilde{X}^h(t)|+|\tilde{Y}(t)|) \leq C'(\omega).
\end{align*}

Since the integrand is bounded, we have
\[
\left|
\int_s^t b(C_h(u),\tilde X^h(C_h(u)),\tilde Y(C_h(u)))\,du
\right|
\le
C'(\omega)(t-s).
\]

The integrand
\[
u \mapsto \sigma(C_h(u),\tilde X^h(C_h(u)),\tilde Y(C_h(u)))
\]
is uniformly bounded. Moreover, almost surely,
$B^H$ is $(H-\rho)$-Hölder continuous. Using the fractional integral estimate, inequality \ref{eqcotaintegral}, we obtain
\[
\left|
\int_s^t \sigma(C_h(u),\tilde X^h(C_h(u)),\tilde Y(C_h(u)))\,dB^H(u)
\right|
\le
C(\omega)(t-s)^{H-\rho}.
\]

Combining both estimates yields
\[
|\tilde X^h(t)-\tilde X^h(s)|
\le
C(\omega)\big[(t-s)+(t-s)^{H-\rho}\big].
\]

Since $H\in(1/2,1)$ and $\rho>0$ is chosen such that $H-\rho>0$,
we have $H-\rho<1$, hence $1-(H-\rho)>0$. Therefore,
for $0<t-s\le T$,
\[
(t-s)
=
(t-s)^{H-\rho}(t-s)^{1-(H-\rho)}
\le
T^{1-(H-\rho)}(t-s)^{H-\rho}.
\]

Consequently, for $C(\omega)=T^{1-(H-\rho)}C'(\omega)$, we obtain
\[
|\tilde X^h(t)-\tilde X^h(s)|
\le
C(\omega)(t-s)^{H-\rho},
\qquad 0\le s<t\le T.
\]
\end{proof}

The following corollary is an immediate consequence of Lemma \ref{cotadeDYtilde-DXtilde} and Theorem \ref{teoCotasapriori}. It will be used in the following section.
\begin{corollary}
\label{cor1}
    Let ${Y}$ and $\Tilde{Y}$ defined by \eqref{eq2} and \eqref{AproxY}, respectively. Under hypotheses  of Theorem \ref{teoCotasapriori}, for $0\leq v\leq u\leq T$,
\begin{equation*}
    |{Y}(u)-{Y}(v)| + |\Tilde{Y}(u)-\Tilde{Y}(v)|  \leq  rK_1|u-v| + rK_2C(\omega)|u-v|^{H-\rho}.
    \end{equation*}
\end{corollary}


\section{Convergence of the Euler scheme}
\label{convergence}
In this section,  we prove that the Euler scheme introduced in Section \ref{Seuler} converges to the solution of \eqref{eq1}.

To quantify the error between the exact solution and its approximation, we introduce increment-type quantities measuring their local discrepancies.

 Denoting
\begin{equation}
\label{Delta}
    \Delta_{u,s}(X,\Tilde{X}^h)=|X(s)-\Tilde{X}^h(s)-X(u)+\Tilde{X}^h(u)|
\end{equation}
  and
\begin{equation}
\label{DeltaY}
    \Delta_{u,s}(Y,\Tilde{Y})=|Y(s)-\Tilde{Y}(s)-Y(u)+\Tilde{Y}(u)|.
\end{equation}  

\begin{theorem}
\label{teoconv}
    Under Hypotheses 1, 2 and 3, for  any $\varepsilon>0$, there exists $h_0>0$ and a set $\Omega_{\varepsilon,h_0,\rho}\subset \Omega$ such that $\mathbb{P}(\Omega_{\varepsilon,h_0,\rho})>1-\varepsilon$, and for each $\omega\in\Omega_{\varepsilon,h_0,\rho}$ and all $h<h_0$,
\begin{align}\label{eq-U}
   U_h:=&\sup_{y\in[0,T]} \left(|X(y)-\Tilde{X}^h(y)|+\int_0^{y} \frac{ \Delta_{v,y}(X,\Tilde{X}^h)} {(y-v)^{\beta + 1} } dv + \int_0^{C_h(y)}\int_{-C_h(v)}^0\frac{\Delta_{v+s,y+s}(X,\Tilde{X}^h)}{(y-v)^{\beta+1}}dsdv\right)\nonumber\\
  \leq& C(\omega)h^{2H-1}, 
\end{align}
for $\beta\in (1-H, 1/2)$ and small enough $\rho>0$ such that $1/2<H-\rho<H$, where $C(\omega)$ does not depend on $h$ and $\varepsilon,$ but may depend on $\rho.$
\end{theorem}

The proof is based on controlling the difference between the exact solution and the numerical scheme through increment-type quantities. In particular, we introduce suitable measures of local error that allow us to handle both the fractional integral terms and the delay contribution in a unified way. The control of the delay terms follows the structural estimate of the functional $Y$ given in Lemma \ref{lemmaY}, together with its discrete counterpart.

Before proving Theorem \ref{teoconv}, we establish a sequence of technical lemmas controlling separately each component of the error.

\begin{lemma}
\label{lemZh}
 Under the conditions of Theorem \ref{teoconv}, taking
    \begin{equation}
    \label{Zh}
       Z^h(t):=\sup_{y\in[0,t]}|X(y)-\Tilde{X}^h(y)|. 
    \end{equation}
Then   
\begin{align*}
    Z^h(t)     &\leq C(\omega)(h^{H-\rho-\beta} + \int_0^t Z^h(u)u^{-\beta}{d}u + \int_0^t\int_0^u \Delta_{v,u}(X,\Tilde{X}^h)(u-v)^{-\beta-1}{d}v{d}u \nonumber\\
    & \qquad \qquad +  \int_0^t\int_0^u \Delta_{v,u}(Y,\Tilde{Y})(u-v)^{-\beta-1}{d}v{d}u).
\end{align*}
\end{lemma}
\begin{proof}

    \begin{align}
    \label{Zh-2}
    Z^h(t)&\leq \sup_{y\in[0,t]} \int_0^y |b(u,X(u),Y(u))-b(C_h(u),\Tilde{X}^h(C_h(u)),\Tilde{Y}(C_h(u)))|{d}u \nonumber\\
 &\quad + \sup_{y\in[0,t]}\left|\int_0^y (\sigma(u,X(u),Y(u))-\sigma(C_h(u),\Tilde{X}^h(C_h(u)),\Tilde{Y}(C_h(u)))){d}B^H(u)\right|\nonumber \\
 &\leq I_1 + I_2 + I_3 + I_4 + I_5 + I_6,
 \end{align}
where
 \begin{align*}
     I_1&=\int_0^t |b(u,X(u),Y(u))-b(u,\Tilde{X}^h(u),\Tilde{Y}(u))|{d}u, \\
     I_2&=\int_0^t |b(u,\Tilde{X}^h(u),\Tilde{Y}(u))-b(C_h(u),\Tilde{X}^h(u),\Tilde{Y}(u))|{d}u, \\
     I_3&=\int_0^t |b(C_h(u),\Tilde{X}^h(u),\Tilde{Y}(u))-b(C_h(u),\Tilde{X}^h(C_h(u)),\Tilde{Y}(C_h(u)))|{d}u, \\
     I_4&=\sup_{y\in[0,t]}\left|\int_0^y (\sigma(u,X(u),Y(u))-\sigma(u,\Tilde{X}^h(u),\Tilde{Y}(u))){d}B^H(u)\right|, \\
     I_5&=\sup_{y\in[0,t]} \left|\int_0^y (\sigma(u,\Tilde{X}^h(u),\Tilde{Y}(u))-\sigma(C_h(u),\Tilde{X}^h(u),\Tilde{Y}(u))){d}B^H(u)\right|, \\
     I_6&=\sup_{y\in[0,t]} \left|\int_0^y (\sigma(C_h(u),\Tilde{X}^h(u),\Tilde{Y}(u))-\sigma(C_h(u),\Tilde{X}^h(C_h(u)),\Tilde{Y}(C_h(u)))){d}B^H(u)\right|.
\end{align*}

From \eqref{eq4}, 
\begin{equation}\label{ConvI-1}
    I_1\leq L_2\int_0^t \left(|X(u)-\Tilde{X}^h(u)|+|Y(u)-\Tilde{Y}(u)|\right){d}u \leq L_2\int_0^t \left(Z^h(u)+|Y(u)-\Tilde{Y}(u)|\right){d}u.
\end{equation}

The control of the term involving $Y$ follows the structural estimate provided by Lemma \ref{lemmaY}, combined with the discretization of the delay functional.

According to \eqref{eq2}, \eqref{AproxY} and Hypothesis 3 \eqref{HK1},
\begin{align}\label{DeltaYtilde}
    &|Y(u)-\Tilde{Y}(u)|=\left|\int_{-r}^0 K(u,s,X(u+s)){d}s - \sum_{i=-N}^{-1} h\cdot K(u,t_i,\Tilde{X}^h(u+t_i))\right| \nonumber \\
    &\leq \int_{-r}^0 \left(K_1|s-C_h(s)|+K_2|X(u+s)-\Tilde{X}^h(u+C_h(s))|\right){d}s  \nonumber \\
    &\leq \int_{-r}^0 \left(K_1|s-C_h(s)|+K_2(|X(u+s)-\Tilde{X}^h(u+s)|+|\Tilde{X}^h(u+s)-\Tilde{X}^h(u+C_h(s))|)\right){d}s, \nonumber 
    \end{align}
by Theorem \ref{teoCotasapriori}, together with inequality 
$$|X(u+s)-\Tilde{X}^h(u+s)|\leq \sup_{w \in [-r,u]}|X(w)-\Tilde{X}^h(w)|=\sup_{w \in [0,u]}|X(w)-\Tilde{X}^h(w)|=Z^h(u)$$ for all $s\in[-r,0]$ we obtain
\begin{align}
 |Y(u)-\Tilde{Y}(u)|     &\leq \int_{-r}^0 \left(K_1|s-C_h(s)|+K_2|X(u+s)-\Tilde{X}^h(u+s)|+K_2C(\omega)|s-C_h(s)|^{H-\rho}\right){d}s  \nonumber \\
     & \leq Ch+ C(\omega)h^{H-\rho} + C Z^h(u).
\end{align}
Thus, from \eqref{ConvI-1}
\begin{align}
\label{ConvI-1-1}
    I_1
    &\leq Ch +C(\omega)h^{H-\rho} + C\int_0^t Z^h(u){d}u. 
\end{align} 
Now, using \eqref{eq3}, we obtain the following  
\begin{equation}\label{ConvI-2}
I_2\leq \int_0^t L_1|u-C_h(u)|{d}u \leq Ch.    
\end{equation}

By Hypothesis 1 \eqref{eq4}, and using Theorem \ref{teoCotasapriori} and Corollary \ref{cor1},
\begin{align}\label{ConvI-3}
    I_3&\leq L_2\int_0^t \left(|\Tilde{X}^h(u)-\Tilde{X}^h(C_h(u))|+|\Tilde{Y}(u)-\Tilde{Y}(C_h(u))|\right){d}u \nonumber \\
    &\leq L_2\int_0^t \left(C(\omega)|u-C_h(u)|^{H-\rho}+rK_1|u-C_h(u)| +rK_2C(\omega)|u-C_h(u)|^{H-\rho}\right){d}u \nonumber \\
    &\leq Ch + C(\omega)h^{H-\rho}.
\end{align}
Using  the estimate \eqref{eqcotaintegral}, we  get
\begin{align}\label{ConvI-4}
     I_4 \leq & \sup_{y\in[0,t]} C(\omega )\left(\int_0^y \dfrac{|\sigma(u,X(u),Y(u))-\sigma(u,\Tilde{X}^h(u),\Tilde{Y}(u))|}{u^{\beta}}{d}u \right.\nonumber \\
    &\!\!\left.+ \!\!\int_0^y\!\!\int_0^u\!\! \dfrac{|\sigma(u,X(u),Y(u))\!-\!\sigma(u, \Tilde{X}^h(u),\Tilde{Y}(u))\!-\!\sigma(v,X(v),Y(v))\!+\!\sigma(v, \Tilde{X}^h(v),\Tilde{Y}(v))|}{(u-v)^{\beta+1}}\!{d}v{d}u\!\!\right) \nonumber \\
    \leq & I_{4-1} + I_{4-2}.
\end{align}
Using \eqref{eq4} and \eqref{DeltaYtilde}, we obtain,
\begin{align}\label{ConvI-7-1}
I_{4-1}&=\sup_{y\in[0,t]} C(\omega ) \int_0^y \dfrac{|\sigma(u,X(u),Y(u))-\sigma(u,\Tilde{X}^h(u),\Tilde{Y}(u))|}{u^{\beta}}{d}u \nonumber \\
 & \leq C(\omega ) \int_0^t\dfrac{|\sigma(u,X(u),Y(u))-\sigma(u,\Tilde{X}^h(u),\Tilde{Y}(u))|}{u^{\beta}}{d}u \nonumber \\
&\leq C(\omega ) \int_0^t\dfrac{L_2r(K_1h+ K_2C(\omega)h^{H-\rho})+ L_2(1+rK_2)Z^h(u)}{u^{\beta}}{d}u \nonumber \\
&\leq C(\omega)h + C(\omega)h^{H-\rho}  + C(\omega)\int_0^t Z^h(u)u^{-\beta}{d}u.\
\end{align}

To estimate $I_{4-2}$, we use the following corollary, whose proof is given in Appendix \ref{Aux-lemmas}.

\begin{corollary}\label{Cor-cota-nualart}
Under Hypothesis 1 (\eqref{eq3} to \eqref{eq5}), if $|x_1|,|x_2|,|x_3|,|x_4|\leq C,$ for some positive constant $C$, then
\begin{align*}
&|\sigma(t_1,x_1,y_1)-\sigma(t_1,x_2,y_2)-\sigma(t_2,x_3,y_3)+\sigma(t_2,x_4,y_4)|\\
\leq  &L_2|x_1-x_2-x_3+x_4| + L_1|x_1-x_2||t_1-t_2| + L_2|x_1-x_2||x_2 - x_4| + L_2|x_1-x_2||x_1 - x_3| \\
    & + L_2|x_1-x_2||y_2 - y_4| + L_2|x_1-x_2||y_1 - y_3| + L_2|y_1-y_2-y_3+y_4| + L_1|y_1-y_2||t_1-t_2| \\
    & + L_2|y_1-y_2||x_2-x_4| + L_2|y_1-y_2||x_1 - x_3| + L_2|y_1-y_2||y_2 - y_4| + L_2|y_1-y_2||y_1 - y_3|.
\end{align*}
\end{corollary} 
Thus, 
\begin{align}\label{ConvI-8-1}
I_{4-2}&\leq \! C(\omega ) \!\!\int_0^t\!\!\int_0^u\!\! \dfrac{|\sigma(u,X(u),Y(u))\!-\!\sigma(u, \Tilde{X}^h(u),\Tilde{Y}(u))\!-\!\sigma(v,X(v),Y(v))\!+\!\sigma(v, \Tilde{X}^h(v),\Tilde{Y}(v))|}{(u-v)^{\beta+1}}{d}v{d}u\nonumber \\
&\leq C(\omega ) \sum_{k=7}^{18} I_k,
\end{align}
where
\begin{align*}
    I_7 &= \int_0^t\int_0^u |X(u)-\Tilde{X}^h(u)-X(v)+\Tilde{X}^h(v)|(u-v)^{-\beta-1} {d}v{d}u \nonumber \\
    I_{8} &=\int_0^t\int_0^u |X(u)-\Tilde{X}^h(u)||u-v|(u-v)^{-\beta-1} {d}v{d}u \nonumber \\
    I_{9} &= \int_0^t\int_0^u |X(u)-\Tilde{X}^h(u)||\Tilde{X}^h(u)-\Tilde{X}^h(v)|(u-v)^{-\beta-1} {d}v{d}u \nonumber \\
    I_{10} &=  \int_0^t\int_0^u |X(u)-\Tilde{X}^h(u)||X(u)-X(v)|(u-v)^{-\beta-1} {d}v{d}u \nonumber \\
    I_{11} &= \int_0^t\int_0^u |X(u)-\Tilde{X}^h(u)||\Tilde{Y}(u)-\Tilde{Y}(v)|(u-v)^{-\beta-1} {d}v{d}u \nonumber \\
    I_{12} &=  \int_0^t\int_0^u |X(u)-\Tilde{X}^h(u)||Y(u)-Y(v)|(u-v)^{-\beta-1} {d}v{d}u \nonumber \\
   I_{13} & =      \int_0^t\int_0^u|Y(u)-\Tilde{Y}(u)-Y(v)+\Tilde{Y}(v)| |(u-v)^{-\beta-1} {d}v{d}u \nonumber \\
    I_{14} &= \int_0^t\int_0^u|Y(u)-\Tilde{Y}(u)||u-v| |(u-v)^{-\beta-1} {d}v{d}u \nonumber \\
    I_{15} &=  \int_0^t\int_0^u|Y(u)-\Tilde{Y}(u)||\Tilde{X}^h(u)-\Tilde{X}^h(v)| |(u-v)^{-\beta-1} {d}v{d}u \nonumber \\
    I_{16} &=  \int_0^t\int_0^u|Y(u)-\Tilde{Y}(u)||X(u) - X(v)| |(u-v)^{-\beta-1} {d}v{d}u \nonumber \\ 
    I_{17}& = \int_0^t\int_0^u|Y(u)-\Tilde{Y}(u)||\Tilde{Y}(u) - \Tilde{Y}(v)| |(u-v)^{-\beta-1} {d}v{d}u \nonumber  \\
    I_{18} &= \int_0^t\int_0^u |Y(u)-\Tilde{Y}(u)||Y(u) - Y(v)||(u-v)^{-\beta-1}{d}v{d}u.
\end{align*}   
From \eqref{Delta} and \eqref{DeltaY}
\begin{equation}\label{ConvI-9}
I_7 + I_{13}= \int_0^t\int_0^u \Delta_{v,u}(X,\Tilde{X}^h)(u-v)^{-\beta-1} {d}v{d}u + \int_0^t\int_0^u \Delta_{v,u}(Y,\Tilde{Y})(u-v)^{-\beta-1} {d}v{d}u.
\end{equation}
By \eqref{Zh},
\begin{equation}\label{ConvI-10}
I_{8}\leq \int_0^t\int_0^u Z^h(u)(u-v)^{-\beta}{d}v{d}u \leq C\int_0^t Z^h(u){d}u.
\end{equation}
Moreover, by Theorems \ref{texistenciax} and \ref{teoCotasapriori},
\begin{equation} \label{ConvI-11}  
I_{9} + I_{10}\leq C(\omega)\int_0^t\int_0^u Z^h(u)(u-v)^{H-\rho-\beta-1}{d}v{d}u \leq C(\omega)\int_0^t Z^h(u)u^{H-\rho-\beta}{d}u\leq C(\omega)\int_0^t Z^h(u){d}u.
\end{equation}
From Corollary \ref{cor1} and \eqref{Zh}, 
\begin{align}\label{ConvI-13}
I_{11}+I_{12} &=  \int_0^t\int_0^u |X(u)-\Tilde{X}^h(u)||\Tilde{Y}(u)-\Tilde{Y}(v)|(u-v)^{-\beta-1} {d}v{d}u \nonumber  \\
&\leq \int_0^t\int_0^u Z^h(u)(rK_1|u-v|+rK_2C(\omega)|u-v|^{H-\rho})(u-v)^{-\beta-1}{d}v{d}u \nonumber \\
&\leq C(\omega)\int_0^t Z^h(u){d}u.
\end{align}
Using \eqref{DeltaYtilde}, Theorem \ref{texistenciax} and Theorem \ref{teoCotasapriori}
 \begin{align}\label{ConvI-16}
    I_{14}+ I_{15}+ I_{16} =&\int_0^t\int_0^u|Y(u)-\Tilde{Y}(u)|\left[|u-v|+|\Tilde{X}^h(u)-\Tilde{X}^h(v)|+|X(u) - X(v)|\right](u-v)^{-\beta-1} {d}v{d}u \nonumber \\
    \leq &C(\omega)\int_0^t\int_0^u (h + h^{H-\rho} + Z^h(u))\left[|u-v|+  2(u-v)^{H-\rho}\right](u-v)^{-\beta-1} {d}v{d}u \nonumber \\
    \leq&  C(\omega)(h + h^{H-\rho} + \int_0^t Z^h(u){d}u).
\end{align}
In the same manner, from \eqref{DeltaYtilde} and Corollary \ref{cor1}
\begin{align}\label{ConvI-19} 
    I_{17}+I_{18}& = \int_0^t\int_0^u|Y(u)-\Tilde{Y}(u)|\left[|\Tilde{Y}(u) - \Tilde{Y}(v)|+ |Y(u) - Y(v)|\right] (u-v)^{-\beta-1} {d}v{d}u \nonumber \\
      & \leq C(\omega)(h + h^{H-\rho} + \int_0^t  Z^h(u) {d}u).
\end{align}

Finally, from \eqref{ConvI-8-1} and the bounds \eqref{ConvI-9} to \eqref{ConvI-19}
\begin{align}\label{ConvI-8-2}
I_{4-2} &\leq C(\omega)\left(\int_0^t \int_0^u (\Delta_{v,u}(X,\Tilde{X}^h)+  \Delta_{v,u}(Y,\Tilde{Y}))(u-v)^{-\beta-1}{d}v{d}u + \int_0^t Z^h(u){d}u + h + h^{H-\rho} \right).
\end{align}
Therefore, from \eqref{ConvI-4}, \eqref{ConvI-7-1} and \eqref{ConvI-8-2} we obtain
\begin{align}\label{ConvI4-1}
I_4 \leq& C(\omega)\left(h + h^{H-\rho} + \int_0^t Z^h(u){d}u + \int_0^t Z^h(u)u^{-\beta}{d}u \right.\nonumber \\ 
&+ \left. \int_0^t \int_0^u (\Delta_{v,u}(X,\Tilde{X}^h) + \Delta_{v,u}(Y,\Tilde{Y}))(u-v)^{-\beta-1} {d}v{d}u  \right).
\end{align}

Let us now return to $I_5$ in \eqref{Zh-2}, from \eqref{eqcotaintegral},
\begin{align*}
    I_5 &\leq C(\omega) \left(\int_0^t \left|\sigma(u,\Tilde{X}^h(u),\Tilde{Y}(u))-\sigma(C_h(u),\Tilde{X}^h(u),\Tilde{Y}(u))\right|u^{-\beta}{d}u  \right. \\
    &\quad + \int_0^t \int_0^u \left|\sigma(u,\Tilde{X}^h(u),\Tilde{Y}(u))-\sigma(C_h(u),\Tilde{X}^h(u),\Tilde{Y}(u)) \right. \\
    &\quad \quad \quad \quad -\left. \left(\sigma(v,\Tilde{X}^h(v),\Tilde{Y}(v))-\sigma(C_h(v),\Tilde{X}^h(v),\Tilde{Y}(v)) \right)\right| 
     \left.(u-v)^{-\beta-1}{d}v{d}u \right)\\
    &= I_{5-1} + I_{5-2}, 
\end{align*}
where
\begin{align*}
I_{5-1}&=  \int_0^t |\sigma(u,\Tilde{X}^h(u),\Tilde{Y}(u))-\sigma(C_h(u),\Tilde{X}^h(u),\Tilde{Y}(u))|u^{-\beta}{d}u, \\
I_{5-2}&=  \int_0^t \int_0^u |\sigma(u,\Tilde{X}^h(u),\Tilde{Y}(u))-\sigma(C_h(u),\Tilde{X}^h(u),\Tilde{Y}(u))-(\sigma(v,\Tilde{X}^h(v),\Tilde{Y}(v)) \\
&\qquad \qquad \qquad \qquad \qquad \qquad + \sigma(C_h(v),\Tilde{X}^h(v),\Tilde{Y}(v)))| (u-v)^{-\beta-1}{d}v{d}u.
\end{align*}
 \text{From Hypothesis 1  and \eqref{eq3},}
\begin{equation}\label{ConvI(5-1)}
   I_{5-1} \leq L_1\int_0^t |u-C_h(u)|u^{-\beta}{d}u 
    \leq L_1\int_0^t h u^{-\beta}{d}u. 
    \leq L_1h. 
\end{equation}
To study $I_{5-2}$, we split the integral over $[0,u]$ into $[0,C_h(u)]$ and $[C_h(u) , u]$. Using Hypothesis 1, \eqref{eq3}, Teorema \eqref{teoCotasapriori}, Corollary \eqref{cor1}, \eqref{u-chu} and Remark \eqref{u-cu-beta}, we  obtain 
\begin{align}
\label{ConvI(5-2)}
    I_{5-2}   
    \leq & L_1\int_0^t \int_0^{C_h(u)} (|u-C_h(u)|+|v-C_h(v)|)(u-v)^{-\beta-1} {d}v{d}u \nonumber\\
   &\ + \int_0^t \int_{C_h(u)}^u (|\sigma(u,\Tilde{X}^h(u),\Tilde{Y}(u))-\sigma(v,\Tilde{X}^h(v),\Tilde{Y}(v))|\nonumber\\
    &\qquad \qquad \qquad \quad +|\sigma(C_h(u),\Tilde{X}^h(u),\Tilde{Y}(u))     -\sigma(C_h(u),\Tilde{X}^h(v),\Tilde{Y}(v)))|(u-v)^{-\beta-1}{d}v{d}u \nonumber\\
   \leq &C h^{1-\beta} + C\int_0^t \int_{C_h(u)}^u \left(|u-v| + |\Tilde{X}^h(u)-\Tilde{X}^h(v)| + |\Tilde{Y}(u)- \Tilde{Y}(v)| \right)(u-v)^{-\beta-1}{d}v{d}u\nonumber\\
   \leq &C h^{1-\beta} + C(\omega)\int_0^t \int_{C_h(u)}^u \left(2|u-v| + 2|u-v|^{H-\rho} \right)(u-v)^{-\beta-1}{d}v{d}u\nonumber\\
   \leq &C(\omega)h^{1-\beta}+C(\omega)h^{H-\rho-\beta}.
    \end{align}
Thus, from \eqref{ConvI(5-1)} and \eqref{ConvI(5-2)} we obtain
\begin{equation}
    \label{ConvI-5}
    I_5\leq  C(\omega)h^{H-\rho-\beta}.
\end{equation}

Similarly to $I_5$, we have for $I_6$ that
\begin{align*}
    I_6 &\leq C(\omega) (  I_{6-1} + I_{6-2}),
\end{align*}
where
\begin{align*}\label{ConvI(6-1-2)}
I_{6-1}&=      \int_0^t |\sigma(C_h(u),\Tilde{X}^h(u),\Tilde{Y}(u))-\sigma(C_h(u),\Tilde{X}^h(C_h(u)),\Tilde{Y}(C_h(u)))|u^{-\beta}{d}u, \nonumber  \\
I_{6-2}&= \int_0^t \int_0^u \left|  \sigma(C_h(u),\Tilde{X}^h(u),\Tilde{Y}(u))-\sigma(C_h(u),\Tilde{X}^h(C_h(u)),\Tilde{Y}(C_h(u))) \right. \\
&\quad - \left. \left(\sigma(C_h(v),\Tilde{X}^h(v),\Tilde{Y}(v))-\sigma(C_h(v),\Tilde{X}^h(C_h(v)),\Tilde{Y}(C_h(v))\right)\right|(u-v)^{-\beta-1}{d}v{d}u. \nonumber
\end{align*}
Using Hypothesis 1 \eqref{eq4}, Theorem \eqref{teoCotasapriori} and Corollary \ref{cor1} we have,
\begin{equation}\label{ConvI(6-1}
        I_{6-1} \leq L_2\int_0^t (|\Tilde{X}^h(u)-\Tilde{X}^h(C_h(u))|+|\Tilde{Y}(u)-\Tilde{Y}(C_h(u))|)u^{-\beta}{d}u 
    \leq  C(\omega)h^{H-\rho}.
\end{equation}
Again, splitting the integral over $[0,u]$ into $[0,C_h(u)]$ and $[C_h(u) , u]$, from Hypothesis 1 \eqref{eq4}, Theorem \ref{teoCotasapriori}, Corollary \ref{cor1} and \eqref{u-chu}, we have
\begin{align}\label{ConvI(6-2)}
    I_{6-2} 
    &\leq C(\omega)h^{H-\rho-\beta} + C(\omega)h^{1-\beta}.
    \end{align}
Therefore, from \eqref{ConvI(6-1} and \eqref{ConvI(6-2)}
\begin{equation}
    \label{ConvI(6-3)}
    I_6 \leq C(\omega)h^{H-\rho-\beta}.
\end{equation}
Summarizing, from \eqref{Zh-2}, \eqref{ConvI-1-1}, \eqref{ConvI-2}, \eqref{ConvI-3}, \ref{ConvI4-1}, \ref{ConvI-5}, \eqref{ConvI(6-3)}, for $\omega \in \Omega_{\varepsilon,h_0,\rho}$ (see Theorem \ref{teoCotasapriori}), and recalling that $1-H < \beta < \frac{1}{2}$ and that $\rho > 0$ is sufficiently small so that $H-\rho-\beta > 0$, 
\begin{align}
\label{Zh-3}
    Z^h(t) &\leq C(\omega) \left(h + h^{H-\rho} + h^{1-\beta} + h^{H-\rho-\beta} + \int_0^t Z^h(u){d}u + \int_0^t Z^h(u)u^{-\beta}{d}u \right. \nonumber  \\
    &\qquad \qquad  + \left. \int_0^t\int_0^u \Delta_{v,u}(X,\Tilde{X}^h)(u-v)^{-\beta-1}{d}v{d}u + \int_0^t\int_0^u \Delta_{v,u}(Y,\Tilde{Y})(u-v)^{-\beta-1}{d}v{d}u \right) \nonumber\\
    &\leq C(\omega) \left( h^{H-\rho-\beta} + \int_0^t Z^h(u)u^{-\beta}{d}u + \int_0^t\int_0^u \Delta_{v,u}(X,\Tilde{X}^h)(u-v)^{-\beta-1}{d}v{d}u\right. \nonumber\\
    & \qquad \qquad +  \left. \int_0^t\int_0^u \Delta_{v,u}(Y,\Tilde{Y})(u-v)^{-\beta-1}{d}v{d}u\right).
\end{align}
\end{proof}
We continue our analysis by studying the term in \eqref{Zh-3}, $ \displaystyle \int_0^t\int_0^u \Delta_{v,u}(Y,\Tilde{Y})(u-v)^{-\beta-1}{d}v$ in the following Lemma.

\begin{lemma}
\label{lemcotaDY}
Under the hypotheses of  Theorem \ref{teoconv}.
Let  $\Delta_{v,u}(X,\Tilde{X}^h)$ and $\Delta_{v,u}(Y,\Tilde{Y})$ given in \eqref{Delta}  and    \eqref{DeltaY}, respectively, then
\begin{align}
   \label{intDeltaY-1}
    \int_0^t\int_0^u \Delta_{v,u}(Y,\Tilde{Y})  (u-v)^{-\beta-1}dvdu \leq & C(\omega)h^{H-\rho-\beta}+ C(\omega)\int_0^t Z^h(u) du +  C(\omega)\int_0^t \kappa(u) du,
\end{align}
where 
\begin{equation}
\label{kappa}
    \kappa(u)=\int_0^{C_h(u)}\int_{-C_h(v)}^0(u-v)^{-\beta-1}\Delta_{v+s,u+s}(X,\Tilde{X}^h)dsdv.
\end{equation}
\end{lemma}

 \begin{proof} 
 We combine the structural properties of the functional $Y$ from Lemma \ref{lemmaY} with the discretization errors introduced by the Euler scheme.
 
 From \eqref{DeltaY}
\begin{align*}
    \Delta_{v,u}(Y,\Tilde{Y})
    \leq & \int_{-r}^0 |K(u,s,X(u+s))-K(u, C_h(s),\Tilde{X}^h(u+C_h(s)))\\
    &\qquad -K(v,s,X(v+s)) +K(v, C_h(s),\Tilde{X}^h(v+C_h(s)))|{d}s,     
\end{align*}
by properly applying Corollary \eqref{Cor-cota-nualart} to the function $K$, we obtain
\begin{equation}
\label{intDeltaY0}
  \int_0^u \Delta_{v,u}(Y,\Tilde{Y})  (u-v)^{-\beta-1}dv \leq C \sum_{i=1}^7  M_i(u),
\end{equation}
where
\begin{align}
      M_1(u)  &= \int_0^u\int_{-r}^0 |u-v||s-C_h(s)|(u-v)^{-\beta-1} {d}s{d}v \nonumber \\
    M_2(u) &=  \int_0^u\int_{-r}^0 |u-v||X(v+s)-\Tilde{X}^h(v+C_h(s))|(u-v)^{-\beta-1} {d}s{d}v \nonumber \\ 
       M_3(u) &= \int_0^u\int_{-r}^0 |u-v||X(u+s)-\Tilde{X}^h(u+C_h(s))|(u-v)^{-\beta-1} {d}s{d}v \nonumber \\
    M_4(u)& = \int_0^u\int_{-r}^0|X(u+s)-X(v+s)-\Tilde{X}^h(u+C_h(s))+\Tilde{X}^h(v+C_h(s))|(u-v)^{-\beta-1} {d}s{d}v \nonumber \\
\label{covM}
    M_5(u) &= \int_0^u\int_{-r}^0|X(u+s)-X(v+s)||s-C_h(s)| (u-v)^{-\beta-1} {d}s{d}v\nonumber \\
   M_6(u) &= \int_0^u\int_{-r}^0|X(u+s)-X(v+s)||X(v+s)-\Tilde{X}^h(v+C_h(s))| (u-v)^{-\beta-1} {d}s{d}v \nonumber \\
    M_7(u) &=  \int_0^u\int_{-r}^0|X(u+s)-X(v+s)||X(u+s)-\Tilde{X}^h(u+C_h(s))|(u-v)^{-\beta-1} {d}s{d}v.
\end{align}

For $M_1$ we have, 
\begin{align}\label{covM1}
 M_1(u)  &\leq  hr\int_0^u (u-v)^{-\beta}{d}vdu \leq Ch. 
\end{align}    
From Theorem \ref{teoCotasapriori} and recalling that $X(s)=\Tilde{X}^h(s)$ for $-r\leq s\leq 0$, as both correspond to the initial condition, and following the analysis on the different subintervals as in Lemma \eqref{Cota- I(w)}, we obtain
\begin{align}\label{covM2}
   M_2(u)  \leq  & \int_0^u (u-v)^{-\beta}\int_{-r}^0 \left[|X(v+s)-\Tilde{X}^h(v+s)|+ |\Tilde{X}^h(v+s)-\Tilde{X}^h(v+C_h(s))|\right] dsdv\nonumber\\
    \leq & \int_0^u (u-v)^{-\beta}\int_{-v}^0  Z^h(v+s){d}s{d}v  + \int_0^u (u-v)^{-\beta}  \,\left\{\int_{-r}^{-v} \left|\phi(v+s)-\phi(v+C_h(s))\right|{d}s \right. \nonumber\\
    &+ \int_{-v}^{-C_h(v)} \left[\left|\Tilde{X}^h(v+s)-\Tilde{X}^h(0)|+ |\Tilde{X}^h(0)-\phi(v+C_h(s))\right|\right]{d}s \nonumber\\
    &+ \left. \int_{-C_h(v)}^0 |\Tilde{X}^h(v+s)-\Tilde{X}^h(v+C_h(s))| {d}s \right\}{d}v\nonumber\\
    \leq & C\int_0^u (u-v)^{-\beta} Z^h(v){d}v   +C(\omega) \int_0^u (u-v)^{-\beta}\left\{ \int_{-r}^{-v} (s-C_h(s))^{H-\rho}{d}s \right. \nonumber\\
    &+ \int_{-v}^{-C_h(v)} \left. [(v-C_h(v))^{H-\rho}+ (s-C_h(s))^{H-\rho}]{d}s+\int_{-C_h(v)}^0 (s-C_h(s))^{H-\rho}{d}s \right\}{d}v\nonumber\\
    \leq & C Z^h(u) \int_0^u (u-v)^{-\beta}{d}v + C(\omega)h^{H-\rho}\int_0^u (u-v)^{-\beta} dv\nonumber\\
    \leq & CZ^h(u) + C(\omega)h^{H-\rho}.
   \end{align} 
   Analogously,
\begin{align}\label{covM3}
    M_3(u)  \leq  &\int_0^u (u-v)^{-\beta}\int_{-r}^0 \left[|X(u+s)-\Tilde{X}^h(u+s)|+ |\Tilde{X}^h(u+s)-\Tilde{X}^h(u+C_h(s))|\right]{d}s{d}v\nonumber\\
    \leq & C Z^h(u)+ C(\omega)h^{H-\rho}.
   \end{align} 
From Theorem \eqref{texistenciax} and since $H-\rho -\beta>0$,
\begin{align}\label{covM5}
    M_5(u)  \leq& h\int_0^u (u-v)^{-\beta-1}\left[\int_{-r}^{-u} |\phi(u+s)-\phi(v+s)|ds\right.\nonumber\\
    &+ \left.\int_{-u}^{-v} |{X}(u+s)-X(0)|+|{X}(0)-\phi(v+s)|ds+ \int_{-v}^{0} |{X}(u+s)-X(v+s)|{d}s\right]{d}v\nonumber\\
    \leq&  C(\omega)h\int_0^u (u-v)^{-\beta-1}4(u-v)^{H-\rho}{d}v
    \leq   C(\omega)h.  
   \end{align}
We now study $M_6$. With an analysis similar to that used for $M_2$, 
\begin{align}
    M_6(u)  \leq &\int_0^u (u-v)^{-\beta-1}\left\{ \int_{-r}^{-u} |\phi(u+s)-\phi(v+s)||\phi(v+s)-\phi(v+C_h(s))|ds\right.\nonumber\\
    &+\int_{-u}^{-v} \left( |{X}(u+s)-X(0)|+|{X}(0)-\phi(v+s)|\right) |\phi(v+s)-\phi(v+C_h(s))|ds\nonumber\\
    &+ \int_{-v}^{-C_h(v)} |{X}(u+s)-X(v+s)|\left[ |X(v+s)-X(0)|+|X(0)-\phi(v+C_h(s))|\right] {d}s\nonumber\\
    &+ \left. \int_{-C_h(v)}^0  |{X}(u+s)-X(v+s)| \left[|X(v+s)-\Tilde{X}^h(v+s)|+|\Tilde{X}^h(v+s)-\Tilde{X}^h(v+C_h(s))|\right] \right\}{d}v\nonumber\\
     \leq & C(\omega)\int_0^u (u-v)^{-\beta-1}\left\{\int_{-r}^{-u} (u-v)^{H-\rho}(s-C_h(s))^{H-\rho}ds+
  \int_{-u}^{-v} 2 (u-v)^{H-\rho}(s-C_h(s))^{H-\rho}\right.ds\nonumber\\
  &+ \int_{-v}^{-C_h(v)} (u-v)^{H-\rho}[(v-C_h(v))^{H-\rho}+(s-C_h(s))^{H-\rho}]{d}s \nonumber\\
  &+ \left. \int_{-C_h(v)}^0  (u-v)^{H-\rho}[Z^h(v+s)+(s-C_h(s))^{H-\rho}] \right\}{d}v\nonumber.
  \end{align}  
 Hence, 
\begin{align}\label{covM6}
M_6(u)     &\leq C(\omega)h^{H-\rho} + C(\omega)Z^h(u).
   \end{align} 
   
 Analogously to $M_6(u)$,
 \begin{align}\label{covM7}
    M_7(u)      &\leq C(\omega)h^{H-\rho} + C(\omega)Z^h(u).
     \end{align}
To study $M_4(u)$, we introduce the following notation 
\begin{equation}
\label{upsilon}
    \Upsilon(u,v,s)=|X(u+s)-X(v+s)-\Tilde{X}^h(u+C_h(s))+\Tilde{X}^h(v+C_h(s))|.
\end{equation}
Then, by splitting the interval $[-r,0]$ into subintervals we get
\begin{align}
    \label{M4-0}
 M_4(u)  =&\int_0^u\int_{-r}^0\Upsilon(u,v,s)(u-v)^{-\beta-1} {d}s{d}v \nonumber \\
=&\int_0^{C_h(u)}\int_{-r}^{-u}\Upsilon(u,v,s)(u-v)^{-\beta-1} {d}s{d}v +\int_{C_h(u)}^u\int_{-r}^{-u}\Upsilon(u,v,s)(u-v)^{-\beta-1} {d}s{d}v\nonumber\\
&+\int_0^{C_h(u)}\int_{-u}^{-C_h(u)}\Upsilon(u,v,s)(u-v)^{-\beta-1} {d}s{d}v +\int_{C_h(u)}^u\int_{-u}^{-C_h(u)}\Upsilon(u,v,s)(u-v)^{-\beta-1} {d}s{d}v\nonumber\\
&+\int_0^{C_h(u)}\int_{-C_h(u)}^0\Upsilon(u,v,s)(u-v)^{-\beta-1} {d}s{d}v +\int_{C_h(u)}^u\int_{-C_h(u)}^0\Upsilon(u,v,s)(u-v)^{-\beta-1} {d}s{d}v\nonumber\\
=& \sum_{i=1}^6 A_i(u).
\end{align}
We consider each of these cases separately.\\
\noindent
{\bf{Case 1: ($A_1$ and $A_2$)}} 
We have $0\leq v\leq u$ and $-r\leq s \leq -u$. Then $u+s$, $v+s$, $u+C_h(s)$ and $v+C_h(s)$ are non positive; thus
$$\Upsilon(u,v,s)=|\phi(u+s)-\phi(v+s)-\phi(u+C_h(s))+\phi(v+C_h(s))|.$$
Using the H\"older continuity of $\phi$ we obtain
\begin{align}
    \label{A-1}
    A_1(u)\leq &\int_0^{C_h(u)}\int_{-r}^{-u}(u-v)^{-\beta-1}[|\phi(u+s)-\phi(u+C_h(s))|+|\phi(v+s)-\phi(v+C_h(s))|] {d}s{d}v\nonumber\\
    \leq &C(\omega)\int_0^{C_h(u)}\int_{-r}^{-u}(u-v)^{-\beta-1}2(s-C_h(s))^{H-\rho} {d}s{d}v\nonumber\\
    \leq &C(\omega)h^{H-\rho}(u-C_h(u))^{-\beta}
\end{align}
and
\begin{align}
    \label{A-2}
    A_2(u)\leq &\int_{C_h(u)}^u\int_{-r}^{-u}(u-v)^{-\beta-1}[|\phi(u+s)-\phi(v+s)|+|\phi(u+C_h(s))-\phi(v+C_h(s))|] {d}s{d}v\nonumber\\
    \leq &C(\omega)\int_{C_h(u)}^u\int_{-r}^{-u}(u-v)^{-\beta-1}2(u-v)^{H-\rho} {d}s{d}v\nonumber\\
    \leq &C(\omega)h^{H-\rho-\beta}.
\end{align}

\noindent
{\bf{Case 2: ($A_3$ and $A_4$)}} 
We have $-u\leq s \leq -C_h(u)$, hence $u+s\ge 0$ and  $u+C_h(s)\le 0$.  Then $|u+C_h(s)|=-u-C_h(s)\leq s-C_h(s)$.
For $A_3$, we have $0\leq v\leq C_h(u)$, then  $v+s\le 0$ and  $v+C_h(s)\le 0$. Consequently, 
$$\Upsilon(u,v,s)=|X(u+s)-\phi(v+s)-\phi(u+C_h(s))+\phi(v+C_h(s))|.$$
Theorem \ref{texistenciax} and the H\"older continuity of $\phi$ implies
\begin{align}
    \label{A-3}
    A_3(u)\leq &\int_0^{C_h(u)}\int_{-u}^{-C_h(u)}(u-v)^{-\beta-1}[|X(u+s)-X(0)|+|\phi(v+s)-\phi(v+C_h(s))|\nonumber\\
    & \hspace{5.5cm}+|X(0)-\phi(u+C_h(s))|] {d}s{d}v\nonumber\\
    \leq &C(\omega)\int_0^{C_h(u)}\int_{-u}^{-C_h(u)}(u-v)^{-\beta-1}[(u+s)^{H-\rho} + (s-C_h(s))^{H-\rho} +|u+C_h(s)|^{H-\rho}] {d}s{d}v\nonumber\\
    \leq &C(\omega)\int_0^{C_h(u)}\int_{-u}^{-C_h(u)}(u-v)^{-\beta-1}\left[(u-C_h(u))^{H-\rho} + 2(s-C_h(s))^{H-\rho} \right] {d}s{d}v\nonumber\\
    \leq &C(\omega)h^{H-\rho+1}(u-C_h(u))^{-\beta}.
\end{align}
Now, we divide the integral in $A_4(u)$ as follows
\begin{align}
    \label{A-4}
    A_4(u)= &\int_{C_h(u)}^u\int_{-u}^{-v}(u-v)^{-\beta-1}\Upsilon(u,v,s){d}s{d}v+\int_{C_h(u)}^u\int_{-v}^{-C_h(u)}(u-v)^{-\beta-1}\Upsilon(u,v,s){d}s{d}v\nonumber\\ 
    =& A_{4-1}(u) + A_{4-2}(u).
\end{align}
For $A_{4-1}(u)$, $C_h(u)<v<u$ and $-u\leq s\leq -v$, then $v+s\le 0$ and $v+C_h(s)\le 0$, since  $C_h(s) < s$. Analogously to $A_3(u)$,
\begin{align}
    \label{A-4-1}
    A_{4-1}(u)\leq &\int_{C_h(u)}^u\int_{-u}^{-v}(u-v)^{-\beta-1}\left[|X(u+s)-X(0)|+|\phi(v+s)-\phi(v+C_h(s))| \right.\nonumber\\
    & \hspace{5.5cm} \left. +|X(0)-\phi(u+C_h(s))|\right] {d}s{d}v\nonumber\\
    \leq &C(\omega)\int_{C_h(u)}^u\int_{-u}^{-v}(u-v)^{-\beta-1}\left[(u+s)^{H-\rho} + (s-C_h(s))^{H-\rho} +|u+C_h(s)|^{H-\rho}\right] {d}s{d}v\nonumber\\
    \leq & C(\omega)\int_{C_h(u)}^u\int_{-u}^{-v}(u-v)^{-\beta-1}[(u-C_h(u))^{H-\rho} + (s-C_h(s))^{H-\rho} +(s-C_h(s))^{H-\rho}] {d}s{d}v\nonumber\\
    \leq &C(\omega)h^{H-\rho}\int_{C_h(u)}^u(u-v)^{-\beta}{d}v
    \leq C(\omega)h^{H-\rho}(u-C_h(u))^{1-\beta}\leq C(\omega)h^{H-\rho+1-\beta}.
\end{align}
For $A_{4-2}(u)$, $C^h(u)<v<u$ and $-v\leq s\leq -C_h(u)$, then $v+s\ge 0$ and $v+C^h(s)\le 0$ then
\begin{align*}
    \Upsilon(u,v,s)&=|X(u+s)-X(v+s)-\phi(u+C_h(s))+\phi(v+C_h(s))|\\
    &\leq |X(u+s)-X(v+s)|+|\phi(u+C_h(s))-\phi(v+C_h(s))|\leq  C(\omega)(u-v)^{H-\rho}.
\end{align*}
Hence,
\begin{align}
    \label{A-4-2}
    A_{4-2}(u)\leq  &C(\omega)\int_{C_h(u)}^u(u-v)^{H-\rho-\beta-1}\int_{-v}^{-C_h(u)}{d}s{d}v\leq  C(\omega)h^{H-\rho-\beta}.
    \end{align}
Thus, from \eqref{A-4}, \eqref{A-4-1} and \eqref{A-4-2},
\begin{align}
    \label{A-4-3}
    A_{4}(u)\leq &C(\omega)h^{H-\rho-\beta}.
\end{align}

\noindent
{\bf{Case 3: ($A_5$ and $A_6$)}} 
We have $0\le v\le u$ and $-C_h(u)\le s \le 0$, then $u+s\ge 0$ and $u+C_h(s)\ge 0$.
We divide the integral in $A_5(u)$ as follows,
\begin{align}
    \label{A-5}
    A_5(u)= &\int_0^{C_h(u)}\int_{-C_h(u)}^{-v}(u-v)^{-\beta-1}\Upsilon(u,v,s){d}s{d}v+\int_0^{C_h(u)}\int_{-v}^{-C_h(v)}(u-v)^{-\beta-1}\Upsilon(u,v,s){d}s{d}v\nonumber\\
    &+ \int_0^{C_h(u)}\int_{-C_h(v)}^0(u-v)^{-\beta-1}\Upsilon(u,v,s){d}s{d}v\nonumber\\ 
    =& A_{5-1}(u) + A_{5-2}(u) + A_{5-3}(u).
\end{align}
For $A_{5-1}(u)$,  $-C_h(u)\le s \le -v$, then $v+s\le 0 $ and $v+C_h(s)\le 0$. From Theorem \ref{teoCotasapriori}, \eqref{Zh} and the  Hölder continuity of $\phi$,
\begin{align*}
    \Upsilon(u,v,s)&=|X(u+s)-\phi(v+s)-\Tilde{X}^h(u+C_h(s))+\phi(v+C_h(s))|\\
    &\leq |X(u+s)-\Tilde{X}^h(u+s)|+|\Tilde{X}^h(u+s)-\Tilde{X}^h(u+C_h(s))|+|\phi(v+s)-\phi(v+C_h(s))|\\
    &\leq Z^h(u+s)+ C(\omega)(s-C_h(s))^{H-\rho}.
\end{align*}
Therefore, by $C_h(u)-v \leq u-v$,
\begin{align}
    \label{A-5-1}
    A_{5-1}(u)\leq &\int_0^{C_h(u)}\int_{-C_h(u)}^{-v}(u-v)^{-\beta-1}[Z^h(u+s)+ C(\omega)(s-C_h(s))^{H-\rho}]{d}s{d}v\nonumber\\
    \leq &C[Z^h(u)+ C(\omega)h^{H-\rho}]\int_0^{C_h(u)}(u-v)^{-\beta}{d}v\nonumber\\
    \leq &C Z^h(u)  +  C(\omega)h^{H-\rho}.
\end{align}
For $A_{5-2}(u)$, $-v\leq s \le -C_h(v)$, then $-C_h(u)\leq C_h(s)\leq -v \le s \le -C_h(v)$. Thus, $v+s\ge 0$, $v+ C_h(s)\leq 0$ and $|v+ C_h(s)|=-v - C_h(s) \leq u-v$. Hence,
\begin{align*}
    \Upsilon(u,v,s)&=|X(u+s)-X(v+s)-\Tilde{X}^h(u+C_h(s))+\phi(v+C_h(s))|\\
    &\leq |X(u+s)-X(v+s)|+|\Tilde{X}^h(u+C_h(s))-\Tilde{X}^h(0)|+|\Tilde{X}^h(0)-\phi(v+C^h(s))|\\
    & \leq C(\omega)[(u-v)^{H-\rho} + (u+C_h(s))^{H-\rho}+ |v+C_h(s)|^{H-\rho}]\\
    & \leq C(\omega)(u-v)^{H-\rho}.
\end{align*}
Then
\begin{align}
    \label{A-5-2}
    A_{5-2}(u)\leq & C(\omega)\int_0^{C_h(u)}\int_{-v}^{-C_h(v)}(u-v)^{H-\rho-\beta-1}{d}s{d}v\leq  C(\omega)h.
\end{align}
For $A_{5-3}(u)$, $-C_h(v)\leq s \le 0$, then $v+s\ge 0 $ and $v+C_h(s)\ge 0$, by the definition of $\Upsilon$ given in \eqref{upsilon}
\begin{align*}
     \Upsilon(u,v,s)=&|X(u+s)-X(v+s)-\Tilde{X}^h(u+C_h(s))+\Tilde{X}^h(v+C_h(s))|\\
    \leq & |X(u+s)-X(v+s)-\Tilde{X}^h(u+s)+\Tilde{X}^h(v+s)|\\
    & + |\Tilde{X}^h(u+s)-\Tilde{X}^h(u+C_h(s))|+|\Tilde{X}^h(v+C_h(s))-\Tilde{X}^h(v+s)|\\
    = & \Delta_{v+s,u+s}(X,\Tilde{X}^h) + |\Tilde{X}^h(u+s)-\Tilde{X}^h(u+C_h(s))|+|\Tilde{X}^h(v+C_h(s))-\Tilde{X}^h(v+s)|.
\end{align*}
By the definition of $\kappa$  given in \eqref{kappa},   Theorem \ref{teoCotasapriori} and \eqref{deltaX}  
\begin{align}
    \label{A-5-3}
    A_{5-3}(u) 
    \leq & \kappa(u) 
    + \int_0^{C_h(u)}\int_{-C_h(v)}^0(u-v)^{-\beta-1}[|\Tilde{X}^h(u+s)-\Tilde{X}^h(u+C_h(s))|\nonumber\\
    &\hspace{6cm}+|\Tilde{X}^h(v+C_h(s))-\Tilde{X}^h(v+s)|]dsdv\nonumber\\
   \leq& \kappa(u) + C(\omega)h^{H-\rho} \int_0^{C_h(u)}\int_{-C_h(v)}^0(u-v)^{-\beta-1}dsdv\nonumber\\
   \leq& \kappa(u) + C(\omega)h^{H-\rho} (u-C_{h}(u))^{-\beta}
\end{align}
Hence, from \eqref{A-5}, \eqref{A-5-1}, \eqref{A-5-2}, and \eqref{A-5-3}, we obtain 
\begin{align}
    \label{A-5-0}
    A_{5}(u) \leq & C Z^h(u) +  C(\omega)h^{H-\rho} + \kappa(u) +  C(\omega)h^{H-\rho} (u-C^{h}(u))^{-\beta}.
\end{align}

\

For $A_{6}$, $-C_h(u)\leq s \le 0$ and $C_h(u)\leq v \le u$, then $v+s\ge 0 $ and $v+C_h(s)\ge 0$. Then,
\begin{align*}
    \Upsilon(u,v,s)
    &\leq |X(u+s)-X(v+s)|+|\Tilde{X}^h(u+C_h(s))-\Tilde{X}^h(v+C_h(s))|\leq C(\omega)(u-v)^{H-\rho}.
\end{align*}
Hence,
\begin{align}
    \label{A-6}
    A_{6}(u)\leq & C(\omega)\int_{C_h(u)}^u\int_{-C_h(u)}^{0}(u-v)^{-\beta-1}(u-v)^{H-\rho}{d}s{d}vdu
    \leq  C(\omega)\int_{C_h(u)}^u(u-v)^{H-\rho-\beta-1}{d}vdu\nonumber\\
    \leq &  C(\omega)h^{H-\rho-\beta}.
\end{align}
Using the equality \eqref{M4-0} and the inequalities \eqref{A-1}, \eqref{A-2}, \eqref{A-3}, \eqref{A-4-3}, \eqref{A-5-0} and \eqref{A-6}, we obtain
\begin{align}
\label{M4-2}
    M_4(u)\leq  C Z^h(u) +  C(\omega)h^{H-\rho-\beta} + \kappa(u) +  C(\omega)h^{H-\rho} (u-C^{h}(u))^{-\beta}.
\end{align}
In summary, from \eqref{intDeltaY0}
\begin{equation}
\label{intDeltaY}
    \int_0^t\int_0^u \Delta_{v,u}(Y,\Tilde{Y})  (u-v)^{-\beta-1}dv du\leq C \sum_{i=1}^7\int_0^t  M_i(u)du
\end{equation}
and by bounds \eqref{covM1} to \eqref{covM7} and \eqref{M4-2}, together with \eqref{u-chu}, we obtain the result \eqref{intDeltaY-1}.
\end{proof}

\vspace{0.5cm}


 Let us now study $\int_0^u \Delta_{v,u}(X,\Tilde{X}^h)(u-v)^{-\beta-1}{d}v$ in the following lemma.

\begin{lemma}
\label{lemaDelta}
 Under the conditions of  Theorem \ref{teoconv}. Let $0\leq v <u\leq T$ and $\Delta_{v,u}(X,\Tilde{X}^h)$ be given in \eqref{Delta}, defining
\begin{equation*}\label{theta}
    \theta(u):=\int_0^u \Delta_{v,u}(X,\Tilde{X}^h)(u-v)^{-\beta-1}{d}v,
\end{equation*}
then
   \begin{equation*}\label{eq10}
\theta(u)\leq C(\omega)\left(h^{H-\rho-\beta} + \int_0^u (Z^h(s)+\theta(s)+\kappa(s))(u-s)^{-\beta}{d}s+ \int_0^u Z^h(s)(u-s)^{-2\beta}{d}s\right).    
\end{equation*}
\end{lemma}
\begin{proof}
    Let $0\leq v <u\leq T$,  by \eqref{eqcotaintegral},
\begin{align}
\label{deltaX0}
    \Delta_{v,u}(X,\Tilde{X}^h) \leq & \int_v^u |b(s,X(s),Y(s))-b(C_h(s),\Tilde{X}^h(C_h(s)),\Tilde{Y}(C_h(s)))|{d}s \nonumber\\
    &+ C(\omega) \left(\int_v^u |\sigma(s,X(s),Y(s))-\sigma(C_h(s),\Tilde{X}^h(C_h(s)),\Tilde{Y}(C_h(s)))| (s-v)^{-\beta}{d}s\right. \nonumber \\
    & + \int_v^u\int_v^s |\sigma(s,X(s),Y(s))-\sigma(C_h(s),\Tilde{X}^h(C_h(s)),\Tilde{Y}(C_h(s)))  \nonumber\\
    &\left.-(\sigma(z,X(z),Y(z))-\sigma(C_h(z),\Tilde{X}^h(C_h(z)),\Tilde{Y}(C_h(z))))|(s-z)^{-\beta-1}{d}z{d}s \right).   
\end{align}
Then 
\begin{equation}
\label{thetau1}
    \theta_u=\int_0^u \Delta_{v,u}(X,\Tilde{X}^h)  (u-v)^{-\beta-1}dv \leq \sum_{i=1}^9 J_i,
\end{equation}
where
\begin{align*}
    J_1(u)&= \int_0^u \int_v^u |b(s,X(s),Y(s))-b(s,\Tilde{X}^h(s),\Tilde{Y}(s))|{d}s (u-v)^{-\beta-1}{d}v, \nonumber \\
    J_2(u)&= \int_0^u\int_v^u |b(s,\Tilde{X}^h(s),\Tilde{Y}(s))-b(C_h(s),\Tilde{X}^h(s),\Tilde{Y}(s))|{d}s (u-v)^{-\beta-1}{d}v, \nonumber \\
    J_3(u)&= \int_0^u\int_v^u |b(C_h(s),\Tilde{X}^h(s),\Tilde{Y}(s))-b(C_h(s),\Tilde{X}^h(C_h(s)),\Tilde{Y}(C_h(s)))|{d}s(u-v)^{-\beta-1}{d}v, \nonumber  \\
    J_4(u) &= \int_0^u C(\omega) \int_v^u |\sigma(s,X(s),Y(s))-\sigma(s,\Tilde{X}^h(s),\Tilde{Y}(s))) (s-v)^{-\beta}{d}s(u-v)^{-\beta-1}{d}v, \nonumber \\
    J_5(u)&= \int_0^u C(\omega) \int_v^u |\sigma(s,\Tilde{X}^h(s),\Tilde{Y}(s))-\sigma(C_h(s),\Tilde{X}^h(s),\Tilde{Y}(s))) (s-v)^{-\beta}{d}s(u-v)^{-\beta-1}{d}v, 
\end{align*}   
\begin{align}
\label{ConvJ1-9}
    J_6(u)&= \int_0^u C(\omega) \int_v^u |\sigma(C_h(s),\Tilde{X}^h(s),\Tilde{Y}(s))-\sigma(C_h(s),\Tilde{X}^h(C_h(s)),\Tilde{Y}(C_h(s)))) (s-v)^{-\beta}{d}s(u-v)^{-\beta-1}{d}v, \nonumber \\
    J_7(u)&= \int_0^u C(\omega)\int_v^u\int_v^s |\sigma(s,X(s),Y(s))-\sigma(s,\Tilde{X}^h(s),\Tilde{Y}(s)) \nonumber  \\
    &\quad -(\sigma(z,X(z),Y(z))-\sigma(z,\Tilde{X}^h(z),\Tilde{Y}(z)))|(s-z)^{-\beta-1}{d}z{d}s(u-v)^{-\beta-1}{d}v, \nonumber \\
    J_8(u)&= \int_0^u C(\omega) \int_v^u\int_v^s |\sigma(s,\Tilde{X}^h(s),\Tilde{Y}(s))-\sigma(C_h(s),\Tilde{X}^h(s),\Tilde{Y}(s)) \nonumber \\
    &\quad -(\sigma(z,\Tilde{X}^h(z),\Tilde{Y}(z))-\sigma(C_h(z),\Tilde{X}^h(z),\Tilde{Y}(z)))|(s-z)^{-\beta-1}{d}z{d}s(u-v)^{-\beta-1}{d}v, \nonumber  \\
    J_9(u)&= \int_0^u C(\omega) \int_v^u\int_v^s |\sigma(C_h(s),\Tilde{X}^h(s),\Tilde{Y}(s))-\sigma(C_h(s),\Tilde{X}^h(C_h(s)),\Tilde{Y}(C_h(s))) \nonumber  \\
    &\quad -(\sigma(C_h(z),\Tilde{X}^h(z),\Tilde{Y}(z))-\sigma(C_h(z),\Tilde{X}^h(C_h(z)),\Tilde{Y}(C_h(z))))|(s-z)^{-\beta-1}{d}z{d}s\nonumber (u-v)^{-\beta-1}{d}v.
    \nonumber \\
\end{align}
From Hypothesis 1 \eqref{eq4}, changing the limits of integration and  \eqref{DeltaYtilde}, we obtain for $J_1$ in \eqref{ConvJ1-9}
\begin{align}\label{jota1}
    J_1(u) &\leq L_2 \int_0^u \int_v^u (|X(s)-\Tilde{X}^h(s)| + |Y(s)-\Tilde{Y}(s)|) {d}s (u-v)^{-\beta-1}{d}v\nonumber  \\
    &\leq C(\omega) \int_0^u \int_v^u (Z^h(s) + h + h^{H-\rho}) {d}s (u-v)^{-\beta-1}{d}v  \nonumber\\
    &\leq C(\omega)\int_0^u ( Z^h(s) + h + h^{H-\rho})\int_0^s (u-v)^{-\beta-1} {d}v {d}s\nonumber \\
    &\leq  C(\omega)\int_0^u Z^h(s)(u-s)^{-\beta}  {d}s + C(\omega)(h + h^{H-\rho}).
\end{align}
and
\begin{equation}\label{jota2}
    J_2(u) \leq L_1\int_0^u \int_v^u |s-C_h(s)| {d}s (u-v)^{-\beta-1}{d}v \leq L_1 h\int_0^u (u-v)^{-\beta}{d}v \leq Ch.
\end{equation}
By Theorem \ref{teoCotasapriori} and Corollary \ref{cor1},
\begin{align}
\label{jota3}
    J_3(u) &\leq L_2\int_0^u \int_v^u (|\Tilde{X}^h(s)-\Tilde{X}^h(C_h(s))| + |\Tilde{Y}(s)-\Tilde{Y}(C_h(s))|) {d}s (u-v)^{-\beta-1}{d}v \nonumber\\
    &\leq L_2\int_0^u \int_v^u (C(\omega) h^{H-\rho} + rK_1h + C(\omega)h^{H-\rho} ){d}s (u-v)^{-\beta-1}{d}v\nonumber \\
    &\leq C(\omega)(h^{H-\rho}+h).
\end{align}
For the next calculus, recall that
\begin{equation}
\label{Co}
    \int_0^s (s-v)^{-\beta}(u-v)^{-\beta-1}{d}v\leq C_0 (u-s)^{-2\beta},
\end{equation}    
where $C_0=\int_0^{\infty} (1+y)^{-\beta-1}y^{-\beta}{d}y$.

Using Hypothesis 1 \eqref{eq4}, \eqref{DeltaYtilde}, \eqref{Co} and since $\beta <1/2$, we obtain
\begin{align}
\label{jota4}
    J_4 &\leq L_2 C(\omega)\int_0^u\int_v^u (|X(s)-\Tilde{X}^h(s)|+|Y(s)-\Tilde{Y}(s)|)(s-v)^{-\beta}{d}s(u-v)^{-\beta-1}{d}v\nonumber \\
    &\leq C(\omega)\int_0^u\int_v^u (Z^h(s)+ rK_1h + rK_2Z^h(s))(s-v)^{-\beta}{d}s(u-v)^{-\beta-1}{d}v \nonumber\\
    &\leq C(\omega)\int_0^u\int_v^u Z^h(s)(s-v)^{-\beta}{d}s(u-v)^{-\beta-1}{d}v + C(\omega)h\int_0^u\int_v^u (s-v)^{-\beta}{d}s(u-v)^{-\beta-1}{d}v\nonumber \\
    &\leq C(\omega)\int_0^uZ^h(s)\int_0^s (s-v)^{-\beta}(u-v)^{-\beta-1}{d}v{d}s + C(\omega)h\int_0^u (u-v)^{-2\beta}{d}v \nonumber\\
    &\leq C(\omega)\int_0^u Z^h(s)(u-s)^{-2\beta}{d}s + C(\omega)h,
\end{align}

We estimate $J_5$ similarly to how we obtained $J_2$, and $J_6$ similarly to how we obtained $J_3$, then
\begin{equation}
    \label{J5-1}
    J_5\leq C(\omega)h
\end{equation}
and
\begin{equation}
    \label{J6-1}
    J_6\leq C(\omega)(h^{H-\rho}+h).
\end{equation}
To study $J_7$, we apply Corollary \eqref{Cor-cota-nualart}, then
\begin{equation}
\label{J7}
    J_7 \leq C(\omega)\sum_{k=1}^{5} J_{7-k}, 
\end{equation}
where
\begin{align*}
    J_{7-1}&= \int_0^u \int_v^u\int_v^s |X(s)-\Tilde{X}^h(s)-X(z)+\Tilde{X}^h(z)| (s-z)^{-\beta-1}{d}z{d}s(u-v)^{-\beta-1}{d}v, \\
      J_{7-2}&= \int_0^u \int_v^u\int_v^s |X(s)-\Tilde{X}^h(s)||s-z| (s-z)^{-\beta-1}{d}z{d}s(u-v)^{-\beta-1}{d}v, \\ 
      J_{7-3}&=  \int_0^u \int_v^u\int_v^s  |X(s)-\Tilde{X}^h(s)|\left[|\Tilde{X}^h(s)-\Tilde{X}^h(z)|+ |X(s)-X(z)|\right.\\
      &\qquad \qquad \qquad\left. |\Tilde{Y}(s)-\Tilde{Y}(z)|+|Y(s)-Y(z)|\right]      
      (s-z)^{-\beta-1}{d}z{d}s(u-v)^{-\beta-1}{d}v,\\
      J_{7-4}&=\int_0^u \int_v^u\int_v^s |Y(s)-\Tilde{Y}(s)-Y(z)+\Tilde{Y}(z)| (s-z)^{-\beta-1}{d}z{d}s(u-v)^{-\beta-1}{d}v,\\
    J_{7-5}&=  \int_0^u \int_v^u\int_v^s  |Y(s)-\Tilde{Y}(s)|\left[|s-z|+|\Tilde{X}^h(s)-\Tilde{X}^h(z)|+ |X(s)-X(z)|\right.\\
       &\qquad \qquad \qquad\left. |\Tilde{Y}(s)-\Tilde{Y}(z)|+|Y(s)-Y(z)|\right]      
      (s-z)^{-\beta-1}{d}z{d}s(u-v)^{-\beta-1}{d}v.
\end{align*}
By \eqref{Delta} and changing the limits of integration, we write
\begin{align*}
    J_{7-1}&=\int_0^u \int_0^s \Delta_{z,s}(X,\Tilde{X})(s-z)^{-\beta-1}\int_0^z (u-v)^{-\beta-1}{d}v{d}z{d}s \\
    &\leq \int_0^u \int_0^s (u-z)^{-\beta}\Delta_{z, s}(X,\Tilde{X})(s-z)^{-\beta-1} {d}z{d}s \\
    &= \int_0^u (u-s)^{-\beta} \theta_s{d}s,
    \end{align*}
and   
    \begin{align*}
    J_{7-2}&\leq \int_0^u \int_v^u Z^h(s)\int_v^s (s-z)^{-\beta}{d}z{d}s(u-v)^{-\beta-1}{d}v \\
    &\leq \int_0^u \int_v^u Z^h(s) (s-v)^{1-\beta}{d}s(u-v)^{-\beta-1}{d}v \leq C\int_0^u Z^h(s)\int_0^s  (u-v)^{-\beta-1}{d}v{d}s\\ 
    &\leq C\int_0^u Z^h(s) (u-s)^{-\beta}{d}s.
      \end{align*}
            Due to Theorem \ref{teoCotasapriori}, Corollary \ref{cor1} and \eqref{deltaX},
    \begin{align*}
    J_{7-3}& \leq C(\omega)\int_0^u \int_v^u \int_v^s Z^h(s)\left[2(s-z)^{H-\rho}+ 2rK_1(s-z)+ 2rK_2(s-z)^{H-\rho}\right]\\
    &\hspace{4cm}\times(s-z)^{-\beta-1} {d}z{d}s(u-v)^{-\beta-1}{d}v \\
    &\leq C(\omega)\int_0^u \int_v^u Z^h(s)\int_v^s  (s-z)^{H-\rho-\beta-1}{d}z{d}s(u-v)^{-\beta-1}{d}v \\
    &\leq C(\omega)\int_0^u Z^h(s)(u-s)^{-\beta}{d}s.
      \end{align*}
Using Theorems \ref{texistenciax} and \ref{teoCotasapriori}, Corollary \ref{cor1} and \eqref{DeltaYtilde}; and calculations analogous to those done for $J_{7-3}$, we get
\begin{align*}
    J_{7-5}&\leq C(\omega)\int_0^u \int_v^u \int_v^s [h+h^{H-\rho}+Z^h(s)] (|s-z| + |s-z|^{H-\rho}) (s-z)^{-\beta-1}{d}z{d}s(u-v)^{-\beta-1}{d}v \\
   &\leq C(\omega)\int_0^u \int_v^u Z^h(s)\int_v^s (s-z)^{H-\rho-\beta-1}{d}z{d}s(u-v)^{-\beta-1}{d}v + C(\omega)h^{H-\rho} \\
   &\leq C(\omega)\int_0^u Z^h(s)(u-s)^{-\beta}{d}s + C(\omega)h^{H-\rho}.
\end{align*}

Now, we study the term $J_{7-4}$, similarly to $J_{7-1}$, by \eqref{DeltaY}
\begin{align*}
J_{7-4}&=\int_0^u \int_0^s \Delta_{z, s}(Y,\Tilde{Y})(s-z)^{-\beta-1}\int_0^z (u-v)^{-\beta-1}{d}v{d}z{d}s \nonumber\\
    &\leq \int_0^u (u-s)^{-\beta}\int_0^s \Delta_{z, s}(Y,\Tilde{Y})(s-z)^{-\beta-1} {d}z{d}s.
    \end{align*}
Recalling the calculation made in Lemma \ref{lemcotaDY}, applying the inequality \eqref{intDeltaY} to $\int_0^s \Delta_{z, s}(Y,\Tilde{Y})(s-z)^{-\beta-1} {d}z$, we have
\begin{align}
\label{J74cota}
J_{7-4} &\leq \int_0^u (u-s)^{-\beta}\sum_{i=1}^7  M_i(s) ds,\end{align}
where $M_i(s)$ are the corresponding integrals given in \eqref{covM}.

Using the bounds found in \eqref{covM1} to \eqref{covM7},  we obtain
\begin{align}
    \label{M3Y}
    &\int_0^u (u-s)^{-\beta} [M_1(s)+M_2+ M_3(s)+ M_5(s) +M_6(s)+M_7(s)] ds 
    \leq C\int_0^u  Z^h(s)(u-s)^{-\beta} ds  + C(\omega)h^{H-\rho}.
\end{align}

We will study now $\int_0^u (u-s)^{-\beta}  M_4(s) ds$, from \eqref{M4-2},
\begin{align}
    \label{M4Y} 
   & \int_0^u (u-s)^{-\beta}  M_4(s) ds 
   \leq \int_0^u (u-s)^{-\beta}[C Z^h(s) +  C(\omega)h^{H-\rho-\beta} + \kappa(s) +  C(\omega)h^{H-\rho} (s-C^{h}(s))^{-\beta}] ds,
\end{align}

where $\kappa(s)$ is defined in \eqref{kappa}. We will study the integral of the last term in \eqref{M4Y},
\begin{align*}
    \int_0^u (u-s)^{-\beta}(s-C_h(s))^{-\beta} \!ds=\! \int_0^{C_h(u)}\!\!\! (u-s)^{-\beta}(s-C_h(s))^{-\beta} ds\!+\!\int_{C_h(u)}^u\!\!\! (u-s)^{-\beta}(s-C_h(s))^{-\beta} ds
\end{align*}
From the proof of Lemma 2.2. in \cite{Mishura2008}, we have that
\begin{equation*}
    \int_0^{C_h(u)} (u-s)^{-\beta}(s-C_h(s))^{-\beta} ds \leq C h^{-\beta}.
\end{equation*}
Now 
\begin{align*}
    \int_{C_h(u)}^u (u-s)^{-\beta}(s-C_h(s))^{-\beta} ds& \leq  h^{-\beta}\int_{C_h(u)}^u (u-s)^{-\beta} ds\leq C h^{-\beta}(u-C_h(u))^{1-\beta}\leq Ch^{1-2\beta}.
\end{align*}
Hence, we obtain that
\begin{align}
\label{eqcotaint1}
    \int_0^u (u-s)^{-\beta}(s-C_h(s))^{-\beta} ds \leq Ch^{-\beta}.
\end{align}
Therefore, from \eqref{M4Y}     
\begin{align}
    \label{M4Y-1} 
   & \int_0^u (u-s)^{-\beta}  M_4(s) ds \le  C \int_0^u Z^h(s)(u-s)^{-\beta}  ds + \int_0^u \kappa(s)(u-s)^{-\beta}  ds + C(\omega)h^{H-\rho-\beta} 
\end{align}

Hence, from \eqref{J74cota}, \eqref{M3Y} and \eqref{M4Y-1}, we obtain
\begin{align*}
    J_{7-4}\leq   C(\omega)h^{H-\rho-\beta} + c\int_0^u Z^h(s)(u-s)^{-\beta}  ds +\int_0^u \kappa(s)(u-s)^{-\beta}  ds. 
\end{align*}

Finally, from \eqref{J7}   
  \begin{align}
  \label{J7-1}
J_7 \leq&  C(\omega)\left(h^{H-\rho-\beta} + \int_0^u \theta_s(u-s)^{-\beta} {d}s+ \int_0^u Z^h(s)(u-s)^{-\beta}{d}s + \int_0^u \kappa(s)(u-s)^{-\beta}ds \right).
\end{align}
  
   We now study $J_8$,
\begin{align*}
\label{J8}
    J_8(u)&\leq \int_0^u C(\omega) \int_v^u\int_0^s |\sigma(s,\Tilde{X}^h(s),\Tilde{Y}(s))-\sigma(C_h(s),\Tilde{X}^h(s),\Tilde{Y}(s)) \nonumber \\
    &\quad -(\sigma(z,\Tilde{X}^h(z),\Tilde{Y}(z))-\sigma(C_h(z),\Tilde{X}^h(z),\Tilde{Y}(z)))|(s-z)^{-\beta-1}{d}z{d}s(u-v)^{-\beta-1}{d}v \nonumber \\ 
    &\leq    C(\omega) [\int_0^u \int_v^u\int_0^{C_h(s)} \cdots dzdsdv +\int_0^u \int_v^u\int_{C_h(s)}^s\cdots dzdsdv]   \nonumber\\
    &= C(\omega) [J_{8-1}(u)+J_{8-2}(u)].
\end{align*}
By  Hypothesis 1, \eqref{eq3} - \eqref{eq5},
\begin{align*}
    J_{8-1}(u)&= \int_0^u \int_v^u\int_0^{C_h(s)} |\sigma(s,\Tilde{X}^h(s),\Tilde{Y}(s))-\sigma(C_h(s),\Tilde{X}^h(s),\Tilde{Y}(s)) \nonumber \\
    &\quad -(\sigma(z,\Tilde{X}^h(z),\Tilde{Y}(z))-\sigma(C_h(z),\Tilde{X}^h(z),\Tilde{Y}(z)))|(s-z)^{-\beta-1}{d}z{d}s(u-v)^{-\beta-1}{d}v\\ \nonumber 
    &\leq C\int_0^u\int_v^u\int_0^{C_h(s)} [|s- C_h(s)|+|z-C_h(z)|](s-z)^{-\beta-1}{d}z{d}s(u-v)^{-\beta-1}{d}v\\ \nonumber 
    & \leq  C h \int_0^u\int_v^u\int_0^{C_h(s)} (s-z)^{-\beta-1}{d}z{d}s(u-v)^{-\beta-1}{d}v \nonumber\\
    &\leq C h\int_0^u\int_v^u(s-C_h(s))^{-\beta}{d}s(u-v)^{-\beta-1}{d}v,
\end{align*}
making the change of variables and using the inequality \eqref{eqcotaint1}, 
\begin{align*}
    J_{8-1}(u)& \leq C h \int_0^u\int_0^s (u-v)^{-\beta-1}{d}v (s-C_h(s))^{-\beta}{d}s \leq C h \int_0^u (u-s)^{-\beta} (s-C_h(s))^{-\beta}{d}s\\
    & \leq C h^{1-\beta}. 
\end{align*}
To bound $ J_{8-2}(u)$, by hypothesis 1 and  recalling that  if $ z\in[C_h(s),s]$, then $C_h(z)=C_h(s)$
\begin{align*}
    J_{8-2}(u)=& \int_0^u \int_v^u\int_{C_h(s)}^s |\sigma(s,\Tilde{X}^h(s),\Tilde{Y}(s))-\sigma(C_h(s),\Tilde{X}^h(s),\Tilde{Y}(s)) \nonumber \\
    &\quad -(\sigma(z,\Tilde{X}^h(z),\Tilde{Y}(z))-\sigma(C_h(s),\Tilde{X}^h(z),\Tilde{Y}(z)))|(s-z)^{-\beta-1}{d}z{d}s(u-v)^{-\beta-1}{d}v\nonumber \\ 
    &\leq  \int_0^u \int_v^u\int_{C_h(s)}^s (s-z)^{-\beta-1}(u-v)^{-\beta-1} [|\sigma(s,\Tilde{X}^h(s),\Tilde{Y}(s))- \sigma(s,\Tilde{X}^h(z),\Tilde{Y}(z))|\nonumber \\
    &\qquad+| \sigma(s,\Tilde{X}^h(z),\Tilde{Y}(z))- \sigma(z,\Tilde{X}^h(z),\Tilde{Y}(z))|\nonumber \\ 
    &\qquad +|\sigma(C_h(s),\Tilde{X}^h(z),\Tilde{Y}(z))  -(\sigma(z,\Tilde{X}^h(z),\Tilde{Y}(z))|]{d}z{d}s{d}v\nonumber \\ 
    &\leq C \int_0^u \int_v^u\int_{C_h(s)}^s [|s-z|+|\Tilde{X}^h(s)-\Tilde{X}^h(z)|+|\Tilde{Y}(s)-\Tilde{Y}(z)|](s-z)^{-\beta-1}(u-v)^{-\beta-1}{d}z{d}s{d}v \nonumber
    \end{align*}
    and from Theorem \ref{teoCotasapriori} and Corollary \ref{cor1}
    \begin{align*}
    J_{8-2}(u)&\leq C(\omega) \int_0^u \int_v^u\int_{C_h(s)}^s (s-z)^{H-\rho-\beta-1}(u-v)^{-\beta-1}{d}z{d}s{d}v \nonumber \\
    &\leq  C(\omega) \int_0^u \int_v^u (s-C_h(s))^{H-\rho-\beta}(u-v)^{-\beta-1}{d}s{d}v \nonumber \\
    &\leq  C(\omega)h^{H-\rho-\beta} \int_0^u (u-v)^{-\beta}{d}v \nonumber \\
      &\leq  C(\omega)h^{H-\rho-\beta}.
    \end{align*}
    Hence, we obtain that 
    \begin{equation}
        \label{J8f}
        J_8 \leq  C(\omega)h^{H-\rho-\beta}. 
    \end{equation}
Similarly to $J_8$,
 \begin{align*}
 \label{J9}
 J_9 &\leq  C(\omega) \int_0^u \int_v^u\int_0^s |\sigma(C_h(s),\Tilde{X}^h(s),\Tilde{Y}(s))-\sigma(C_h(s),\Tilde{X}^h(C_h(s)),\Tilde{Y}(C_h(s))) \nonumber  \\
    &\quad -(\sigma(C_h(z),\Tilde{X}^h(z),\Tilde{Y}(z))-\sigma(C_h(z),\Tilde{X}^h(C_h(z)),\Tilde{Y}(C_h(z))))|(s-z)^{-\beta-1}{d}z{d}s(u-v)^{-\beta-1}{d}v\nonumber \\
    &= C(\omega) [\int_0^u \int_v^u\int_0^{C_h(s)} \cdots dzdsdv + \int_0^u \int_v^u\int_{C_h(s)}^s \cdots dzdsdv\nonumber \\
     &= C(\omega) [J_{9-1}(u)+J_{9-2}(u)],
\end{align*}
where
\begin{align*}
    &J_{9-1}(u)= \int_0^u \int_v^u\int_0^{C_h(s)} |\sigma(C_h(s),\Tilde{X}^h(s),\Tilde{Y}(s))-\sigma(C_h(s),\Tilde{X}^h(C_h(s)),\Tilde{Y}(C_h(s))) \nonumber  \\
    & -(\sigma(C_h(z),\Tilde{X}^h(z),\Tilde{Y}(z))-\sigma(C_h(z),\Tilde{X}^h(C_h(z)),\Tilde{Y}(C_h(z))))|(s-z)^{-\beta-1}{d}z{d}s(u-v)^{-\beta-1}{d}v\nonumber \\
    &\leq C\int_0^u\int_v^u\int_0^{C_h(s)} [|s- C_h(s)|^{H-\rho}+|z-C_h(z)|^{H-\rho}](s-z)^{-\beta-1}{d}z{d}s(u-v)^{-\beta-1}{d}v\\ \nonumber 
        &\leq   C h^{H-\rho}\int_0^u\int_v^u\int_0^{C_h(s)} (s-z)^{-\beta-1}{d}z{d}s(u-v)^{-\beta-1}{d}v\\ \nonumber
    & \leq  C h^{H-\rho-\beta}
\end{align*}
and
\begin{align*}
    &J_{9-2}(u)= \int_0^u \int_v^u\int_{C_h(s)}^s |\sigma(C_h(s),\Tilde{X}^h(s),\Tilde{Y}(s))-\sigma(C_h(s),\Tilde{X}^h(C_h(s)),\Tilde{Y}(C_h(s))) \nonumber  \\
    &\quad -(\sigma(C_h(z),\Tilde{X}^h(z),\Tilde{Y}(z))-\sigma(C_h(z),\Tilde{X}^h(C_h(z)),\Tilde{Y}(C_h(z))))|(s-z)^{-\beta-1}{d}z{d}s(u-v)^{-\beta-1}{d}v\nonumber \\
    &= \int_0^u \int_v^u\int_{C_h(s)}^s |\sigma(C_h(s),\Tilde{X}^h(s),\Tilde{Y}(s))- \sigma(C_h(s),\Tilde{X}^h(z),\Tilde{Y}(z)| (s-z)^{-\beta-1}{d}z{d}s(u-v)^{-\beta-1}{d}v\nonumber \\
   & \leq C(\omega) \int_0^u \int_v^u\int_{C_h(s)}^s (s-z)^{H-\rho-\beta-1}(u-v)^{-\beta-1}{d}z{d}s{d}v \nonumber \\
        &\leq  C(\omega)h^{H-\rho-\beta}.
    \end{align*}
    Thus,
    \begin{equation}
        \label{J9f}
        J_9 \leq  C(\omega)h^{H-\rho-\beta}. 
    \end{equation}

In summary, by \eqref{thetau1} and bounds from \eqref{jota1} to  \eqref{jota3}, \ref{jota4} to \ref{J6-1} and \ref{J7-1} to  \eqref{J9f},  we obtain that
\begin{align*}\label{thetau}
\theta(u) \leq& C(\omega)\left( h^{H-\rho-\beta}+ \int_0^u Z^h(s)(u-s)^{-\beta}{d}s  + \int_0^u Z^h(s)(u-s)^{-2\beta}{d}s + \int_0^u \theta(s)(u-s)^{-\beta}{d}s \right.\nonumber \\ 
& + \left. \int_0^u \kappa(s)(u-s)^{-\beta}ds\right).
\end{align*}
Hence, we obtain the result.
\end{proof}

We will study now $\kappa(u)$.

\begin{lemma}
\label{lemkappa}
 Under hypotheses of  Theorem \ref{teoconv}. Let $0\leq u\leq T$ and $\kappa(u)$ defined by \eqref{kappa}, 
   \begin{equation*}\label{eqkappa}
\kappa(u)\leq C(\omega)\left(h^{H-\rho-\beta} + \int_0^u (Z^h(s)+\theta(s)+\kappa(s))(u-s)^{-\beta}{d}s+ \int_0^u Z^h(s)(u-s)^{-2\beta}{d}s\right).    
\end{equation*}
\end{lemma}

\begin{proof}
We follow the proof of Lemma \ref{lemaDelta}, calculous are similar except some technical changes. From \eqref{deltaX0} and \eqref{thetau1},
\begin{align*}
\label{deltaXch}
    \Delta_{v+s, u+s}(X,  \Tilde{X}^h) \leq &  \int_{v+s}^{u+s}|b(w,X(w),Y(w))-b(C_h(w),\Tilde{X}^h(C_h(w)),\Tilde{Y}(C_h(w)))|{d}w \nonumber\\
    &+ C(\omega) \left(\int_{v+s}^{u+s} |\sigma(s,X(w),Y(w))-\sigma(C_h(w),\Tilde{X}^h(C_h(w)),\Tilde{Y}(C_h(w)))| (w-v)^{-\beta}{d}w\right. \nonumber \\
    & + \int_{v+s}^{u+s}\int_{v+s}^w |\sigma(w,X(w),Y(w))-\sigma(C_h(w),\Tilde{X}^h(C_h(w)),\Tilde{Y}(C_h(w)))  \nonumber\\
    &\left.-(\sigma(z,X(z),Y(z))-\sigma(C_h(z),\Tilde{X}^h(C_h(z)),\Tilde{Y}(C_h(z))))|(w-z)^{-\beta-1}{d}z{d}w \right).   
\end{align*}
Then 
\begin{equation}
\label{kappad1}
    \kappa(u)= \int_0^{C_h(u)}\int_{-C_h(v)}^0(u-v)^{-\beta-1} \Delta_{v+s, u+s}(X,\Tilde{X}^h)  (u-v)^{-\beta-1}dv \leq \sum_{i=1}^9 \tilde{J}_i,
\end{equation}
where
\begin{align}
    \tilde{J}_1(u)&= \int_0^{C_h(u)}\int_{-C_h(v)}^0(u-v)^{-\beta-1}  \int_{v+s}^{u+s} |b(w,X(w),Y(w))-b(w,\Tilde{X}^h(w),\Tilde{Y}(w))|{d}w ds{d}v, \nonumber \\
    \tilde{J}_2(u)&=  \int_0^{C_h(u)}\int_{-C_h(v)}^0(u-v)^{-\beta-1}  \int_{v+s}^{u+s} |b(w,\Tilde{X}^h(w),\Tilde{Y}(w))-b(C_h(w),\Tilde{X}^h(w),\Tilde{Y}(w))|{d}w ds{d}v, \nonumber \\
    \tilde{J}_3(u)&=  \int_0^{C_h(u)}\int_{-C_h(v)}^0(u-v)^{-\beta-1}  \int_{v+s}^{u+s}|b(C_h(w),\Tilde{X}^h(w),\Tilde{Y}(w))\nonumber \\
    & \quad -b(C_h(w),\Tilde{X}^h(C_h(w)),\Tilde{Y}(C_h(w)))|{d}wds{d}v, \nonumber  \\
    \tilde{J}_4(u) &=  \int_0^{C_h(u)}\int_{-C_h(v)}^0(u-v)^{-\beta-1}  \int_{v+s}^{u+s} |\sigma(w,X(w),Y(w))-\sigma(w,\Tilde{X}^h(w),\Tilde{Y}(w))) (w-v)^{-\beta}{d}w ds{d}v, \nonumber \\
    \tilde{J}_5(u)&= C(\omega) \int_0^{C_h(u)}\int_{-C_h(v)}^0(u-v)^{-\beta-1}  \int_{v+s}^{u+s} |\sigma(w,\Tilde{X}^h(w),\Tilde{Y}(w)) \nonumber \\
    & \quad -\sigma(C_h(w),\Tilde{X}^h(w),\Tilde{Y}(w))) (w-v)^{-\beta}{d}w ds{d}v, \nonumber \\
    \label{jotab}
        \tilde{J}_6(u)&= C(\omega) \int_0^{C_h(u)}\int_{-C_h(v)}^0(u-v)^{-\beta-1}  \int_{v+s}^{u+s}|\sigma(C_h(w),\Tilde{X}^h(w),\Tilde{Y}(w))\nonumber \\
    & \quad-\sigma(C_h(w),\Tilde{X}^h(C_h(w)),\Tilde{Y}(C_h(w)))) (w-v)^{-\beta}{d}w ds {d}v, \nonumber \\
   \tilde{J}_7(u)&= C(\omega) \int_0^{C_h(u)}\int_{-C_h(v)}^0(u-v)^{-\beta-1}  \int_{v+s}^{u+s}\int_{v+s}^w |\sigma(w,X(w),Y(w))-\sigma(w,\Tilde{X}^h(w),\Tilde{Y}(w)) \nonumber  \\
    &\quad -(\sigma(z,X(z),Y(z))-\sigma(z,\Tilde{X}^h(z),\Tilde{Y}(z)))|(w-z)^{-\beta-1}{d}z{d}w ds{d}v, \nonumber \\
   \tilde{J}_8(u)&= C(\omega) \int_0^{C_h(u)}\int_{-C_h(v)}^0(u-v)^{-\beta-1}  \int_{v+s}^{u+s}\int_{v+s}^w |\sigma(w,\Tilde{X}^h(w),\Tilde{Y}(w))-\sigma(C_h(w),\Tilde{X}^h(w),\Tilde{Y}(w)) \nonumber \\
    &\quad -(\sigma(z,\Tilde{X}^h(z),\Tilde{Y}(z))-\sigma(C_h(z),\Tilde{X}^h(z),\Tilde{Y}(z)))|(w-z)^{-\beta-1}{d}z{d}w ds{d}v, \nonumber  \\
    \tilde{J}_9(u)&= C(\omega)  \int_0^{C_h(u)}\int_{-C_h(v)}^0(u-v)^{-\beta-1}  \int_{v+s}^{u+s}\int_{v+s}^w |\sigma(C_h(w),\Tilde{X}^h(w),\Tilde{Y}(w))\nonumber  \\
    & \quad-\sigma(C_h(w),\Tilde{X}^h(C_h(w)),\Tilde{Y}(C_h(w))) 
    -(\sigma(C_h(z),\Tilde{X}^h(z),\Tilde{Y}(z))\nonumber\\
    &\quad -\sigma(C_h(z),\Tilde{X}^h(C_h(z)),\Tilde{Y}(C_h(z))))|(w-z)^{-\beta-1}{d}z{d}w ds{d}v.
\end{align}

From Hypothesis 1 and \eqref{DeltaYtilde},  changing the limits of integration 
\begin{align}
\label{jotab1}
    \tilde{J}_1(u) &\leq C(\omega)\int_0^{u}\int_{-v}^0(u-v)^{-\beta-1}  \int_{v+s}^{u+s} |X(w)-\Tilde{X}^h(w)|+|Y(w) -\Tilde{Y}(w))|{d}w ds{d}v, \nonumber \\
    &\leq C(\omega) \int_0^{u} [Z^h(w)  + h^{H-\rho}]\int_{w-u}^0\int_{-s}^{w-s} (u-v)^{-\beta-1}  dvdsdw\notag \\
        &\leq  C(\omega)\int_0^{u}Z^h(w)dw  + C(\omega) h^{H-\rho}
\end{align}
and
\begin{align}
\label{jotab2}
   \tilde{J}_2(u) &\leq \int_0^{C_h(u)}\int_{-C_h(v)}^0(u-v)^{-\beta-1}  \int_{v+s}^{u+s} |w-C_h(w)|{d}w ds{d}v \leq Ch.
\end{align}
Also,
\begin{align}
\label{jotab3}
    \tilde{J}_3(u) &\leq\int_0^{C_h(u)}\int_{-C_h(v)}^0(u-v)^{-\beta-1}  \int_{v+s}^{u+s} |\Tilde{X}^h(w)-\Tilde{X}^h(C_h(w))|+|\Tilde{Y}(w) -\Tilde{Y}(C_h(w))|{d}w ds{d}v, \nonumber \\
    &\leq  C(\omega)h^{H-\rho}.
\end{align}
Similarly to \eqref{jota4}, using the change of variable,  the inequality \eqref{Co}  and $\beta<1/2$, we obtain
\begin{align}
\label{jotab4}
    \tilde{J}_4(u) &\leq C(\omega) \int_0^{u}\int_{-v}^0(u-v)^{-\beta-1}  \int_{v+s}^{u+s}  [|X(w)-\Tilde{X}^h(w)|+|Y(w) -\Tilde{Y}(w))|](w-v)^{-\beta}{d}w ds{d}v \nonumber \\
    &\leq C(\omega) \int_0^{u}[Z^h(w)+h^{H-\rho}]\int_0^w\int_{\max\{-v, w-u\}}^0(u-v)^{-\beta-1}(w-v)^{-\beta}dsdvdw \nonumber \\
    &\leq   C(\omega)\int_0^{u}Z^h(w)(u-w)^{-2\beta}dw  +  C(\omega) h^{H-\rho}.
\end{align}
Analogously to $\tilde{J}_2(u)$ and $\tilde{J}_3(u)$, we obtain
\begin{equation}\label{J5-J6}
    \tilde{J}_5(u)\leq Ch  \quad {and} \quad \tilde{J}_6(u)\leq C(\omega)h^{H-\rho}.
\end{equation}

To study $\tilde{J}_7$, we apply Corollary \eqref{Cor-cota-nualart}, then
\begin{equation}
\label{Jb7}
    \tilde{J}_7 \leq C(\omega)\sum_{k=1}^{5} \tilde{J}_{7-k}, 
\end{equation}
where
\begin{align*}
    \tilde{J}_{7-1}&= \!\int_0^{C_h(u)}\!\!\!\int_{-C_h(v)}^0\!\!\!\!(u-v)^{-\beta-1}\!\!  \int_{v+s}^{u+s}\int_{v+s}^w \!\! |X(w)-\Tilde{X}^h(w)-X(z)+\Tilde{X}^h(z)| (w-z)^{-\beta-1}{d}z{d}w ds{d}v, \\
      \tilde{J}_{7-2}&= \int_0^{C_h(u)}\int_{-C_h(v)}^0(u-v)^{-\beta-1}  \int_{v+s}^{u+s}\int_{v+s}^w  |X(w)-\Tilde{X}^h(w)||w
      -z| (w-z)^{-\beta-1}{d}z{d}w ds{d}v, \\ 
      \tilde{J}_{7-3}&=  \int_0^{C_h(u)}\int_{-C_h(v)}^0(u-v)^{-\beta-1}  \int_{v+s}^{u+s}\int_{v+s}^w  |X(w)-\Tilde{X}^h(w)|\left[|\Tilde{X}^h(w)-\Tilde{X}^h(z)|\right.\\
      &\qquad \qquad \qquad\left. + |X(w)-X(z)| + |\Tilde{Y}(w)-\Tilde{Y}(z)|+|Y(w)-Y(z)|\right]     
      (w-z)^{-\beta-1}{d}z{d}w ds{d}v,\\
      \tilde{J}_{7-4}&=\!\int_0^{C_h(u)}\!\!\int_{-C_h(v)}^0\!\!\!\!(u-v)^{-\beta-1}  \!\!\int_{v+s}^{u+s}\!\int_{v+s}^w  |Y(w)-\Tilde{Y}(w)-Y(z)+\Tilde{Y}(z)| (w-z)^{-\beta-1}{d}z{d}w ds{d}v,\\
    \tilde{J}_{7-5}&=  \int_0^{C_h(u)}\int_{-C_h(v)}^0(u-v)^{-\beta-1}  \int_{v+s}^{u+s}\int_{v+s}^w  |Y(w)-\Tilde{Y}(w)|\left[|w-z|+|\Tilde{X}^h(w)-\Tilde{X}^h(z)|\right.\\
       &\qquad \qquad \qquad\left. + |X(w)-X(z)|+ |\Tilde{Y}(w)-\Tilde{Y}(z)|+|Y(w)-Y(z)|\right]    
      (w-z)^{-\beta-1}{d}z{d}w ds{d}v.
\end{align*}

By \eqref{Delta} and changing the limits of integration, we have
\begin{align*}
    \tilde{J}_{7-1}&\leq  \int_0^{u}\int_{-v}^0(u-v)^{-\beta-1}  \int_{v+s}^{u+s}\int_{v+s}^w  \Delta_{z, w}(X,\Tilde{X})(w-z)^{-\beta-1}dzdwdsdv, \\
    &=\int_0^u \int_0^w \Delta_{z,w}(X,\Tilde{X})(w-z)^{-\beta-1}\int_{w-u}^0 \int_{-s}^{z-s} (u-v)^{-\beta-1}{d}v{d}s{d}zdw \\
    &\leq C\int_0^u \int_0^w (w-z)^{-\beta-1}\Delta_{z, w}(X,\Tilde{X})\int_{w-u}^0(u-z+s)^{-\beta} ds{d}z{d}w \\
    &\leq   C\int_0^u \int_0^w \Delta_{z, w}(X,\Tilde{X})(w-z)^{-\beta-1} {d}z{d}w \\
    &= C\int_0^u  \theta(w){d}w.
    \end{align*}
Also,    
   \begin{align*}
    \tilde{J}_{7-2} \leq& \int_0^{u}\int_{-v}^0(u-v)^{-\beta}  \int_{v+s}^{u+s}\int_{v+s}^w  |X(w)-\Tilde{X}^h(w)||w-z| (w-z)^{-\beta-1}{d}z{d}w ds{d}v\notag\\
    =&\int_0^u Z^h(w) \int_0^w (w-z)^{-\beta}\int_{w-u}^0 \int_{-s}^{z-s} (u-v)^{-\beta-1}{d}v{d}s{d}zdw \\
    \leq& C \int_0^u {Z}^h(w)dw.
         \end{align*}

Due to Theorem \ref{teoCotasapriori}, Corollary \ref{cor1} and \eqref{deltaX},
    \begin{align*}
    \tilde{J}_{7-3} \leq &  C(\omega)  \int_0^{u}\int_{-v}^0(u-v)^{-\beta}  \int_{v+s}^{u+s}\int_{v+s}^w    Z^h(w) (w-z)^{H-\rho-\beta-1}{d}z{d}w ds{d}v\\
    =&C(\omega)\int_0^u Z^h(w) \int_0^w (w-z)^{H-\rho-\beta-1}\int_{w-u}^0 \int_{-s}^{z-s} (u-v)^{-\beta-1}{d}v{d}s{d}zdw \\
    \leq& C(\omega) \int_0^u {Z}^h(w)dw.
      \end{align*}
      And similarly,  from Theorems \ref{texistenciax} and \ref{teoCotasapriori}, Corollary \ref{cor1} and \eqref{DeltaYtilde}
        \begin{align*}
    \tilde{J}_{7-5}& \leq  C(\omega)  \int_0^{u}\int_{-v}^0(u-v)^{-\beta}  \int_{v+s}^{u+s}\int_{v+s}^w    [Z^h(w)+h^{H-\rho}] (w-z)^{H-\rho-\beta-1}{d}z{d}w ds{d}v\\
    =&C(\omega)\int_0^u [Z^h(w)+h^{H-\rho}] \int_0^w (w-z)^{H-\rho-\beta-1}\int_{w-u}^0 \int_{-s}^{z-s} (u-v)^{-\beta-1}{d}v{d}s{d}zdw \\
    \leq& C(\omega) \int_0^u {Z}^h(w)dw + C(\omega) h^{H-\rho}.
      \end{align*}
Now we study $\tilde{J}_{7-4}$, similarly to $\tilde{J}_{7-1}$, by \eqref{DeltaY}
\begin{align*}
    \tilde{J}_{7-4}&\leq  \int_0^{u}\int_{-v}^0(u-v)^{-\beta-1}  \int_{v+s}^{u+s}\int_{v+s}^w  \Delta_{z, w}(Y,\Tilde{Y})(w-z)^{-\beta-1}dzdwdsdv, \\
      &\leq   C\int_0^u \int_0^w \Delta_{z, w}(Y,\Tilde{Y})(w-z)^{-\beta-1} {d}z{d}w\\
      &\leq   C\int_0^u  (u-w)^{-\beta}\int_0^w \Delta_{z, w}(Y,\Tilde{Y})(w-z)^{-\beta-1} {d}z{d}w.
    \end{align*}
The last bound of  $\tilde{J}_{7-4}$ is  equal to the bound \eqref{J74cota} obtained in Lemma \ref{lemaDelta} for $J_{7-4}$ . Hence,
\begin{align}
\label{Jb74cota}
    \tilde{J}_{7-4}\leq   C(\omega)h^{H-\rho-\beta} + \int_0^u Z^h(w)(u-w)^{-\beta}  dw +C(\omega)\int_0^u \kappa(w)(u-w)^{-\beta}  dw. 
\end{align}

Thus, from \eqref{Jb7}   
  \begin{align}
  \label{Jb7-1}
\tilde{J}_7 \leq&  C(\omega)\left(h^{H-\rho-\beta} + \int_0^u \theta(w) {d}w+ \int_0^u Z^h(w)(u-w)^{-\beta}{d}w + \int_0^u \kappa(w)(u-w)^{-\beta}dw \right).
\end{align}
Now, we study $\tilde{J}_8(u)$, similar to the last calculus, making a change of variables and integrating, we obtain the following result.
\begin{align}\label{J8tilde}
       \tilde{J}_8(u)&\leq  \int_0^{u}\int_{-v}^0(u-v)^{-\beta-1}  \int_{v+s}^{u+s}\int_{v+s}^w |\sigma(w,\Tilde{X}^h(w),\Tilde{Y}(w))-\sigma(C_h(w),\Tilde{X}^h(w),\Tilde{Y}(w)) \nonumber \\
    &\quad -(\sigma(z,\Tilde{X}^h(z),\Tilde{Y}(z))-\sigma(C_h(z),\Tilde{X}^h(z),\Tilde{Y}(z)))|(w-z)^{-\beta-1}{d}z{d}w ds{d}v \nonumber  \\
    &\leq C \int_0^{u}\int_{0}^w (w-z)^{-\beta-1} |\sigma(w,\Tilde{X}^h(w),\Tilde{Y}(w))-\sigma(C_h(w),\Tilde{X}^h(w),\Tilde{Y}(w)) \nonumber \\
    &\quad -(\sigma(z,\Tilde{X}^h(z),\Tilde{Y}(z))-\sigma(C_h(z),\Tilde{X}^h(z),\Tilde{Y}(z)))|\int_{w-u}^{0}\int_{-s}^{z-s} (u-v)^{-\beta-1}dvdsdzdw \nonumber  \\
     &\leq C \int_0^{u}\int_{0}^w (w-z)^{-\beta-1} |\sigma(w,\Tilde{X}^h(w),\Tilde{Y}(w))-\sigma(C_h(w),\Tilde{X}^h(w),\Tilde{Y}(w)) \nonumber \\
    &\quad -(\sigma(z,\Tilde{X}^h(z),\Tilde{Y}(z))-\sigma(C_h(z),\Tilde{X}^h(z),\Tilde{Y}(z)))|dzdw \nonumber  \\
     &\leq C [\int_0^{u}\int_{0}^{C_h(w)}\cdots dzdw +  \int_0^{u}\int_{C_h(w)}^w\cdots dzdw]\nonumber  \\
     &=  C[ \tilde{J}_{8-1}(u)+ \tilde{J}_{8-2}(u)]. 
\end{align}
From Hypothesis 1 and \eqref{u-chu},
\begin{align*}
       \tilde{J}_{8-1}(u)&= \int_0^{u}\int_0^{C_h(w)}(w-z)^{-\beta-1} |\sigma(w,\Tilde{X}^h(w),\Tilde{Y}(w))-\sigma(C_h(w),\Tilde{X}^h(w),\Tilde{Y}(w)) \nonumber \\
    &\quad -(\sigma(z,\Tilde{X}^h(z),\Tilde{Y}(z))-\sigma(C_h(z),\Tilde{X}^h(z),\Tilde{Y}(z)))|dzdw \nonumber  \\
    &\leq C \int_0^{u}\int_0^{C_h(w)}(w-z)^{-\beta-1}[ |w -C_h(w)|+|z -C_h(z)|]dzdw \nonumber  \\
    &\leq C(\omega)h\int_0^{u}\int_0^{C_h(w)}(w-z)^{-\beta-1}dzdw \nonumber  \\
        &\leq C(\omega)h^{1-\beta}.
\end{align*}
 From Hypothesis 1 , Theorem \ref{teoCotasapriori} and Corollary \ref{cor1},
 \begin{align*}
       \tilde{J}_{8-2}(u)=&\int_0^{u}\int_{C_h(w)}^w(w-z)^{-\beta-1} |\sigma(w,\Tilde{X}^h(w),\Tilde{Y}(w))-\sigma(C_h(w),\Tilde{X}^h(w),\Tilde{Y}(w)) \nonumber \\    
       &\quad -(\sigma(z,\Tilde{X}^h(z),\Tilde{Y}(z))-\sigma(C_h(z),\Tilde{X}^h(z),\Tilde{Y}(z)))|dzdw \nonumber  \\
   \leq &   \int_0^{u}\int_{C_h(w)}^w [|\sigma(w,\Tilde{X}^h(w),\Tilde{Y}(w))-\sigma(w,\Tilde{X}^h(z),\Tilde{Y}(z))| + |\sigma(w,\Tilde{X}^h(z),\Tilde{Y}(z))-\sigma(z,\Tilde{X}^h(z),\Tilde{Y}(z))|\nonumber \\
    &\quad +|\sigma(C_h(w),\Tilde{X}^h(z),\Tilde{Y}(z))-\sigma(C_h(w),\Tilde{X}^h(w),\Tilde{Y}(w))|](w-z)^{-\beta-1}dzdw \nonumber  \\
    \leq &C(\omega) \int_0^{u}\int_{C_h(w)}^w(w-z)^{H-\rho-\beta-1}dzdw \nonumber  \\
        \leq & C(\omega)h^{H-\rho-\beta}.
\end{align*}
Hence, 
\begin{equation}
    \label{Jb8}
    \tilde{J}_{8}(u) \leq  C(\omega)h^{H-\rho-\beta}. 
\end{equation}
Similarly, we study $\tilde{J}_{9}(u)$.
\begin{align*}
       \tilde{J}_9(u)&\leq  \int_0^{u}\int_{-v}^0(u-v)^{-\beta-1}  \int_{v+s}^{u+s}\int_{v+s}^w |\sigma(C_h(w),\Tilde{X}^h(w),\Tilde{Y}(w))\nonumber  \\
    & \quad-\sigma(C_h(w),\Tilde{X}^h(C_h(w)),\Tilde{Y}(C_h(w))) 
    -(\sigma(C_h(z),\Tilde{X}^h(z),\Tilde{Y}(z))\nonumber\\
    &\quad -\sigma(C_h(z),\Tilde{X}^h(C_h(z)),\Tilde{Y}(C_h(z))))|(w-z)^{-\beta-1}{d}z{d}w ds{d}v\nonumber\\
         &\leq C \int_0^{u}\int_{0}^w (w-z)^{-\beta-1} |\sigma(C_h(w),\Tilde{X}^h(w),\Tilde{Y}(w)) -\sigma(C_h(w),\Tilde{X}^h(C_h(w)),\Tilde{Y}(C_h(w))) \nonumber \\
    &\quad -(\sigma(C_h(z),\Tilde{X}^h(z),\Tilde{Y}(z)) -\sigma(C_h(z),\Tilde{X}^h(C_h(z)),\Tilde{Y}(C_h(z))))|{d}z{d}w\nonumber\\
     &\leq C [\int_0^{u}\int_{0}^{C_h(w)}\cdots dzdw +  \int_0^{u}\int_{C_h(w)}^w\cdots dzdw]\nonumber  \\
     &=  C[ \tilde{J}_{9-1}(u)+ \tilde{J}_{9-2}(u)],
\end{align*}
where
\begin{align*}
       \tilde{J}_{9-1}(u)&=\int_0^{u}\int_{0}^{C_h(w)} (w-z)^{-\beta-1} |\sigma(C_h(w),\Tilde{X}^h(w),\Tilde{Y}(w)) -\sigma(C_h(w),\Tilde{X}^h(C_h(w)),\Tilde{Y}(C_h(w))) \nonumber \\
    &\quad -(\sigma(C_h(z),\Tilde{X}^h(z),\Tilde{Y}(z)) -\sigma(C_h(z),\Tilde{X}^h(C_h(z)),\Tilde{Y}(C_h(z))))|{d}z{d}w\nonumber\\
     &\leq C \int_0^{u}\int_0^{C_h(w)}(w-z)^{-\beta-1}[ |w -C_h(w)|+|z -C_h(z)|]dzdw \nonumber  \\
        &\leq C h^{1-\beta}
\end{align*}
and
\begin{align*}
       \tilde{J}_{9-2}(u)&=\int_0^{u}\int_{C_h(w)}^w (w-z)^{-\beta-1} |\sigma(C_h(w),\Tilde{X}^h(w),\Tilde{Y}(w)) -\sigma(C_h(w),\Tilde{X}^h(C_h(w)),\Tilde{Y}(C_h(w))) \nonumber \\
    &\quad -(\sigma(C_h(z),\Tilde{X}^h(z),\Tilde{Y}(z)) -\sigma(C_h(z),\Tilde{X}^h(C_h(z)),\Tilde{Y}(C_h(z))))|{d}z{d}w\nonumber\\
    &=\int_0^{u}\int_{C_h(w)}^w (w-z)^{-\beta-1} |\sigma(C_h(w),\Tilde{X}^h(w),\Tilde{Y}(w)) -\sigma(C_h(w),\Tilde{X}^h(z),\Tilde{Y}(z)) |{d}z{d}w\nonumber\\
     &\leq C(\omega)\int_0^{u}\int_0^{C_h(w)}(w-z)^{H-\rho-\beta-1}dzdw \nonumber  \\
        &\leq C(\omega)h^{H-\rho-\beta}.
\end{align*}
Thus,
\begin{equation}
    \label{Jb9}
    \tilde{J}_{9}(u) \leq  C(\omega)h^{H-\rho-\beta}.
\end{equation}

Finally,  from \ref{kappad1} to \ref{Jb9} we obtain the result.
\end{proof}

\begin{proof}[Proof Theorem \ref{teoconv}]

From Lemmas \ref{lemZh}, \ref{lemcotaDY}, \ref{lemaDelta} and \ref{lemkappa}, 
\begin{align*}
&Z^h(t)+\theta(t)+\kappa(t)\\
&\leq C(\omega)\!\left(\!h^{H-\rho-\beta}\!\! +\! \int_0^t Z^h(u)\!u^{-\beta}{d}u \!+\! \int_0^t\! (Z^h(u)+\theta(u)+\kappa(u))(t-u)^{-\beta}{d}u\!+\! \int_0^t\! Z^h(u)(t-u)^{-2\beta}{d}u\right) \\
&\leq  C(\omega)\left(h^{H-\rho-\beta} + \int_0^t(Z^h(u)+\theta(u)+\kappa(u))u^{-\beta}{d}u + \int_0^t (t-u)^{-2\beta}(Z^h(u)+\theta(u)+\kappa(u)){d}u\right) \\
&\leq C(\omega)\left(h^{H-\rho-\beta}+\int_0^t(Z^h(u)+\theta(u)+\kappa(u))(u^{-\beta}+(t-u)^{-2\beta}){d}u  \right) \\
&\leq  C(\omega)\left(h^{H-\rho-\beta}+ t^{2\beta}\int_0^t (Z^h(u)+\theta(u)+\kappa(u))u^{-2\beta}(t-u)^{-2\beta}{d}u  \right).
\end{align*}
For $\omega \in \Omega_{\varepsilon,h_0,\rho}$ (see Theorem \ref{teoCotasapriori}) and recalling that  $1-H< \beta < 1/2$, $\rho>0$ is small enough to $1/2<H-\rho<H$. 
Due to  Gronwall's Lemma, we have that  
\[
Z^h(t)+\theta(t)+\kappa(t) \leq C(\omega)h^{2H-1}\exp\{Ct\}.
\]
Recalling the definition of $U_h$ in \eqref{eq-U}, the previous bounds on $Z^h$, $\theta$ and $\kappa$ imply the desired estimate.
\end{proof}



\section{Numerical Simulations}
\label{numerical}

In this section, we present numerical experiments to illustrate the strong convergence behavior of the proposed Euler-type method when applied to a stochastic functional differential equation driven by fractional Brownian motion. As a benchmark example, we consider the stochastic delay equation

\begin{equation}\label{edfe_ejemplo}
dX(t)=
\left[
\int_{t-1}^{t} X(s)\,ds
+
e^{-1} X(t)
\right] dt
+
\varepsilon X(t)\, dB^H(t),
\qquad t \in [0,2],
\end{equation}
with initial condition $
X(t) = e^{t},  t \in [-1,0],$
where $B^H(t)$ denotes a fractional Brownian motion with Hurst parameter $H>1/2$. This test equation was originally considered by Buckwar \cite{Buckwar2004} in the case of standard Brownian motion and is commonly used as a benchmark problem for evaluating the strong convergence properties of numerical methods for stochastic functional differential equations. Here, we extend this example to the fractional Brownian motion setting, which introduces long-range dependence and provides a more general framework.
The equation is discretized using the Euler-type scheme described in the previous section on a uniform time mesh with step size $h$. The memory integral appearing in \eqref{edfe_ejemplo} is approximated using Riemann sums. Since no explicit analytical solution is available in the stochastic case, a reference solution is computed using the same numerical method on a very fine mesh consisting of $2^{12}$ time steps. This reference solution is treated as the exact solution and is denoted by $X(T)$.

To quantify the accuracy of the numerical approximation, we compute the strong error, for each step size $h$, simulating $500$ independent trajectories and define the empirical strong error as

\begin{equation*}
\mathrm{Err}(h)
=
\frac{1}{500}
\sum_{j=1}^{500}
\max_{t_i\in\{t_0,\dots,t_N\}}\left|
\tilde{X}_N^{(j)} - X(T)
\right|,
\end{equation*}

where $\tilde{X}_N^{(j)}$ denotes the numerical approximation at time $T$ corresponding to the $j$-th trajectory. The convergence behavior is further analyzed using the ratio
\[
\mathrm{Ratio}
=
\frac{\mathrm{Err}(h/2)}{\mathrm{Err}(h)},
\]
which provides a measure of the reduction in error as the step size is refined.

The simulations are performed for step sizes $h=2^{-4}, 2^{-5}, 2^{-6}, 2^{-7}$ and for three different values of the noise intensity parameter $\varepsilon \in \{0.005, 0.1, 1\}$. The results are summarized in Tables~\ref{tab_err_H051}--\ref{tab_err_H09} and Figures~\ref{tab_fig_H0.51}--\ref{tab_fig_H0.9} for different values of the Hurst parameter.

\begin{table}[H]
\centering
\begin{tabular}{c c c c c c c}
\hline
$H=0.51$: & $\varepsilon=0.005$ & & $\varepsilon=0.1$ & & $\varepsilon=1$ & \\
\hline
$h$ & Error & Ratio & Error & Ratio & Error & Ratio \\
\hline
$2^{-4}$ & $1.681\times10^{-1}$ & -- & $1.582\times10^{-1}$ & -- & $1.388\times10^{-1}$ & --\\
$2^{-5}$ & $8.560\times10^{-2}$ & $0.509$ & $8.049\times10^{-2}$ & $0.509$ & $3.935\times10^{-2}$ & $0.284$ \\
$2^{-6}$ & $4.269\times10^{-2}$ & $0.499$ & $4.047\times10^{-2}$ & $0.503$ & $1.620\times10^{-2}$ & $0.412$ \\
$2^{-7}$ & $2.080\times10^{-2}$ & $0.487$ & $2.011\times10^{-2}$ & $0.497$ & $2.090\times10^{-2}$ & $1.290$ \\
\hline
\end{tabular}
\caption{Strong error and ratio for $H=0.51$ and different values of $\varepsilon$.}
\label{tab_err_H051}
\end{table}

\begin{table}[H]
\centering
\begin{tabular}{c c c c c c c}
\hline
$H=0.75$: & $\varepsilon=0.005$ & & $\varepsilon=0.1$ & & $\varepsilon=1$ & \\
\hline
$h$ & Error & Ratio & Error & Ratio & Error & Ratio \\
\hline
$2^{-4}$ & $1.669\times10^{-1}$ & -- & $1.331\times10^{-1}$ & -- & $7.626\times10^{-1}$ & --\\
$2^{-5}$ & $8.497\times10^{-2}$ & $0.509$ & $6.815\times10^{-2}$ & $0.512$ & $5.132\times10^{-1}$ & $0.673$ \\
$2^{-6}$ & $4.238\times10^{-2}$ & $0.499$ & $3.434\times10^{-2}$ & $0.504$ & $3.251\times10^{-1}$ & $0.634$ \\
$2^{-7}$ & $2.064\times10^{-2}$ & $0.487$ & $1.695\times10^{-2}$ & $0.493$ & $2.165\times10^{-1}$ & $0.666$ \\
\hline
\end{tabular}
\caption{Strong error and ratio for $H=0.75$ and different values of $\varepsilon$.}
\label{tab_err_H075}
\end{table}

\begin{table}[H]
\centering
\begin{tabular}{c c c c c c c}
\hline
$H=0.9$: & $\varepsilon=0.005$ & & $\varepsilon=0.1$ & & $\varepsilon=1$ & \\
\hline
$h$ & Error & Ratio & Error & Ratio & Error & Ratio \\
\hline
$2^{-4}$ & $1.663\times10^{-1}$ & -- & $1.778\times10^{-1}$ & -- & $4.953\times10^{-1}$ & --\\
$2^{-5}$ & $8.464\times10^{-2}$ & $0.509$ & $9.097\times10^{-2}$ & $0.512$ & $2.938\times10^{-1}$ & $0.593$ \\
$2^{-6}$ & $5.482\times10^{-2}$ & $0.500$ & $4.561\times10^{-2}$ & $0.501$ & $1.704\times10^{-1}$ & $0.580$ \\
$2^{-7}$ & $2.056\times10^{-2}$ & $0.487$ & $2.230\times10^{-2}$ & $0.489$ & $9.333\times10^{-2}$ & $0.548$ \\
\hline
\end{tabular}
\caption{Strong error and ratio for $H=0.9$ and different values of $\varepsilon$.}
\label{tab_err_H09}
\end{table}

The following pictures show Sample paths of the numerical solution obtained with the Euler-type scheme and the reference solution for the stochastic functional differential equation \eqref{eq1}--\eqref{eq2} with different Hurst parameter and different values of $\varepsilon$. The reference solution is computed using a sufficiently fine time discretization.

\begin{figure}[H]
\centering
\begin{subfigure}[b]{0.32\linewidth}
\includegraphics[width=\linewidth]{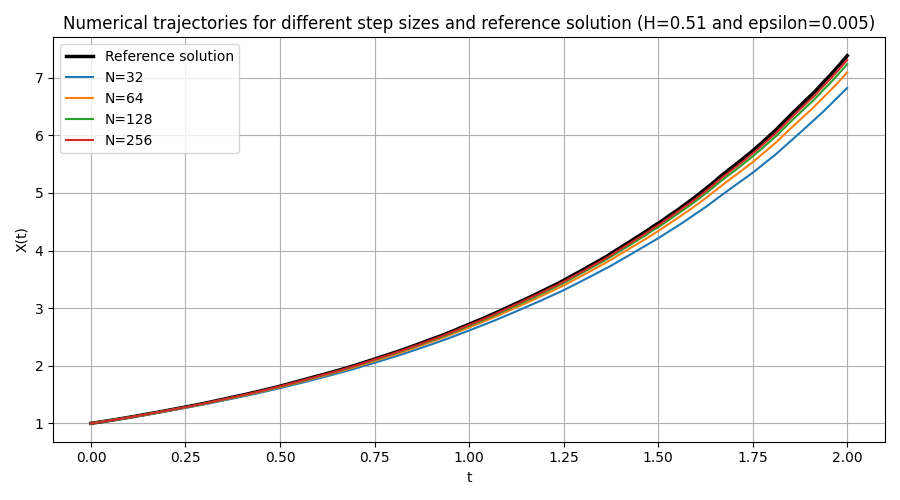}
\caption{}
\label{H_0.51,epsilon_0.005}
\end{subfigure}
\begin{subfigure}[b]{0.32\linewidth}
\includegraphics[width=\linewidth]{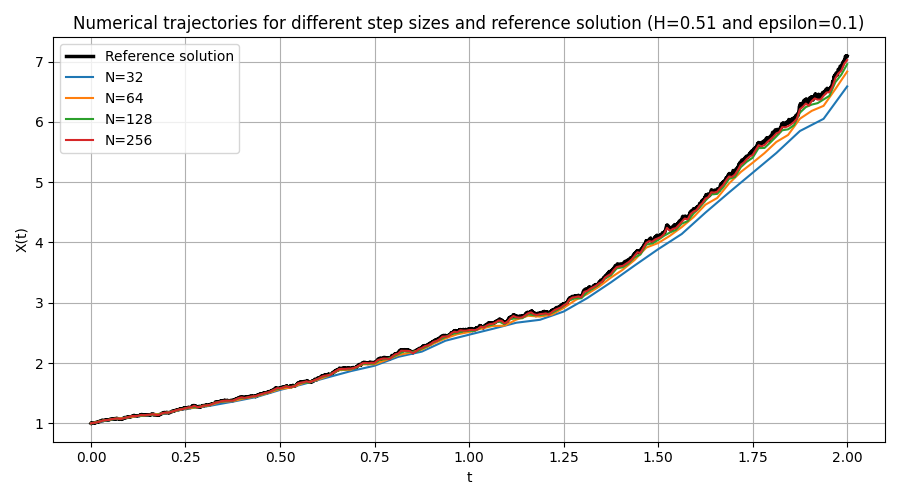}
\caption{}
\label{H_0.51,epsilon_0.1}
\end{subfigure}
\begin{subfigure}[b]{0.32\linewidth}
\includegraphics[width=\linewidth]{Trajectories_H051_epsilon_0005.png}
\caption{}
\label{H_0.51,epsilon_1}
\end{subfigure}
\caption{Sample path with Hurst parameter $H=0.51$ and different values of $\varepsilon$.}
\label{tab_fig_H0.51}
\end{figure}

\begin{figure}[H]
\centering
\begin{subfigure}[b]{0.32\linewidth}
\includegraphics[width=\linewidth]{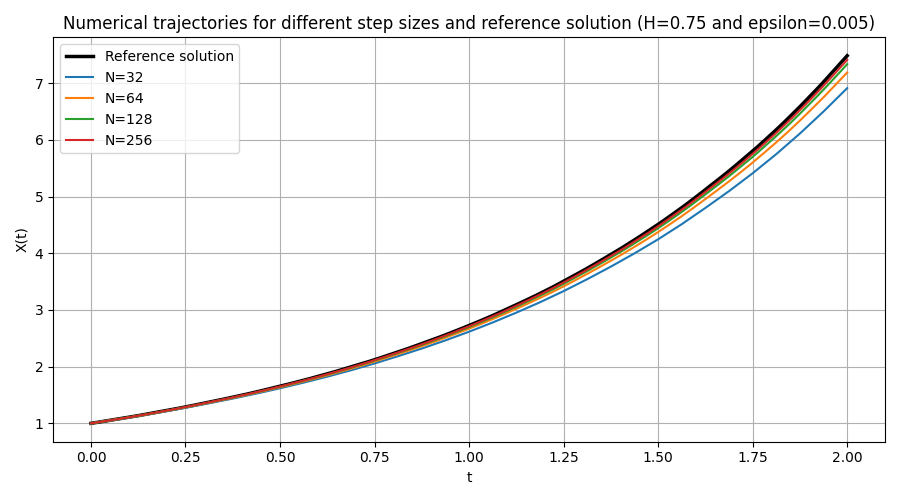}
\caption{}
\label{H_0.75,epsilon_0.005}
\end{subfigure}
\begin{subfigure}[b]{0.32\linewidth}
\includegraphics[width=\linewidth]{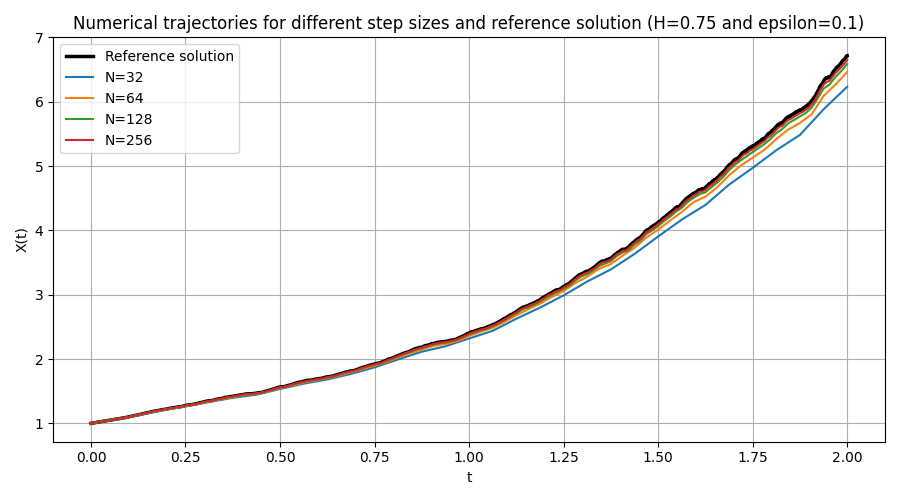}
\caption{}
\label{H_0.75,epsilon_0.1}
\end{subfigure}
\begin{subfigure}[b]{0.32\linewidth}
\includegraphics[width=\linewidth]{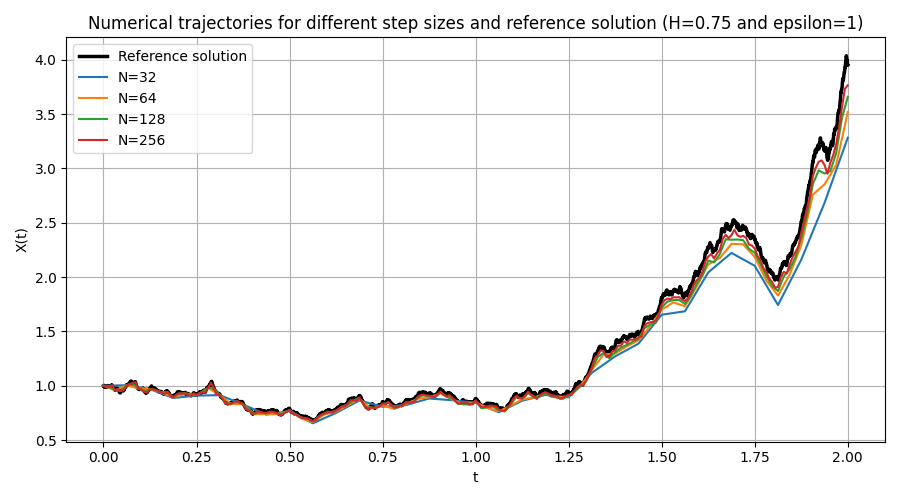}
\caption{}
\label{H_0.75,epsilon_1}
\end{subfigure}
\caption{Sample path with Hurst parameter $H=0.75$ and different values of $\varepsilon$.}
\label{tab_fig_H0.75}
\end{figure}

\begin{figure}[H]
\centering
\begin{subfigure}[b]{0.32\linewidth}
\includegraphics[width=\linewidth]{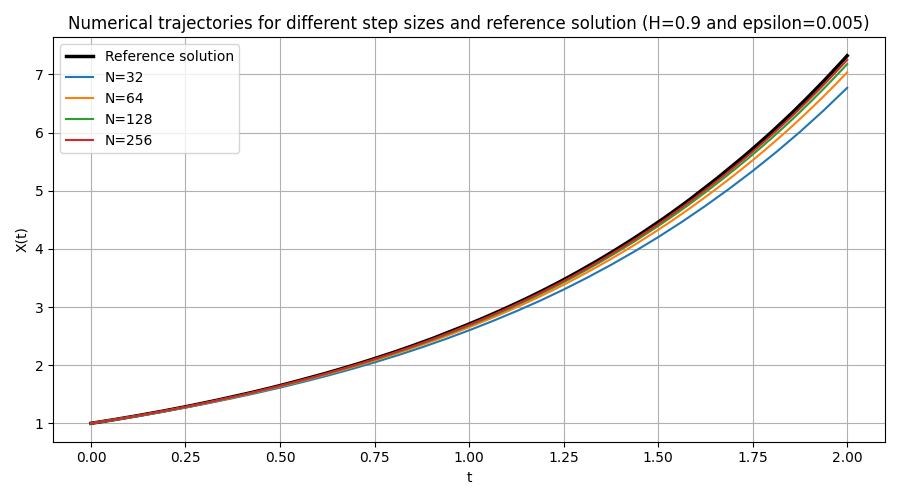}
\caption{}
\label{H_0.9,epsilon_0.005}
\end{subfigure}
\begin{subfigure}[b]{0.32\linewidth}
\includegraphics[width=\linewidth]{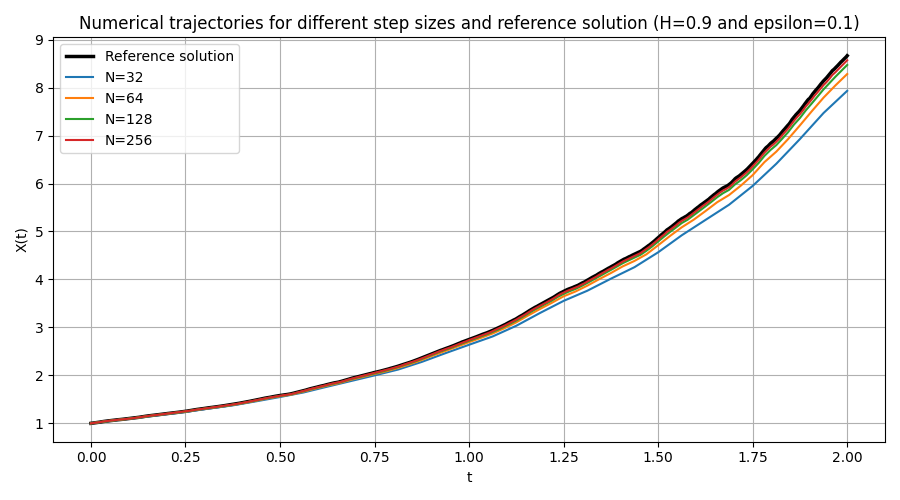}
\caption{}
\label{H_0.9,epsilon_0.1}
\end{subfigure}
\begin{subfigure}[b]{0.32\linewidth}
\includegraphics[width=\linewidth]{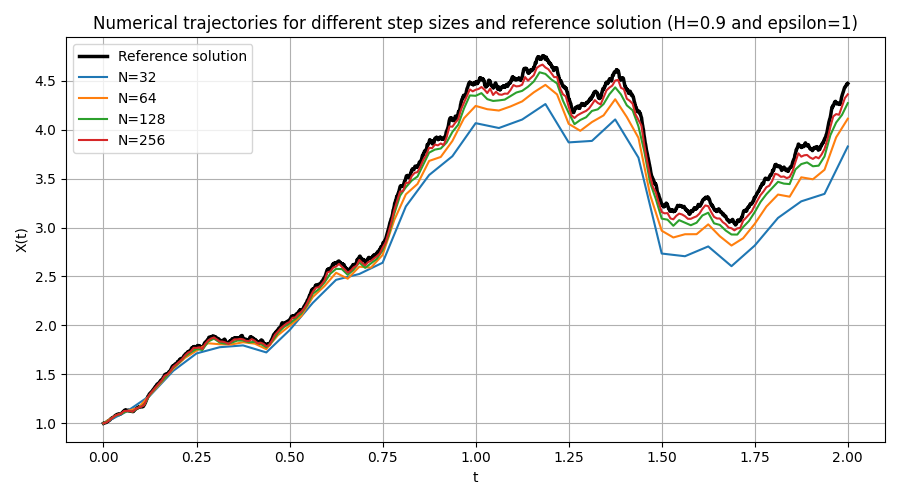}
\caption{}
\label{H_0.9,epsilon_1}
\end{subfigure}
\caption{Sample path with Hurst parameter $H=0.9$ and different values of $\varepsilon$.}
\label{tab_fig_H0.9}
\end{figure}

The results clearly demonstrate the convergence of the numerical method as the step size decreases. Moreover, the observed ratios indicate a consistent reduction of the error when refining the time discretization, confirming the stability and reliability of the proposed Euler-type scheme for stochastic functional differential equations driven by fractional Brownian motion.
\appendix

\section{Auxiliary Technical Lemmas}\label{Aux-lemmas}
{\bf Proof of Lemma \ref{lemmaY}}
\begin{proof}
From \eqref{eq2}, and by adding and subtracting
$
\int_{t-r}^t K(s,u-t,X(u))\,du,
$
\begin{align*}
Y(t)-Y(s) &= \int_{t-r}^t K(t,u-t,X(u))\,du -
\int_{s-r}^s K(s,u-s,X(u))\,du \\
&=
\int_{t-r}^t
\Big( K(t,u-t,X(u)) - K(s,u-t,X(u)) \Big) du \\
&\quad
+
\int_{t-r}^t
\Big( K(s,u-t,X(u)) - K(s,u-s,X(u)) \Big) du 
+
\int_{s-r}^{t-r}
K(s,u-s,X(u))\,du.
\end{align*}
Thus, using Hypothesis 3,
\begin{align*}
|Y(t)-Y(s)| &\leq \int_{t-r}^t K_1 |t-s|\,du+\int_{t-r}^t K_2 |(u-t)-(u-s)|\,du+ \int_{s-r}^{t-r} K_3 (1+|X(u)|)\,du \\
&\le K_1 r |t-s|+K_2 r |t-s|+K_3 |t-s|
+ K_3 \int_{s-r}^{t-r} |X(u)|\,du.
\end{align*}
Hence, we obtain the result.
\end{proof}

{\bf Proof of Lemma \ref{cotadeYtilde-Xtilde}}
\begin{proof} 
For $\Tilde{Y}$ defined by \eqref{AproxY}, from the linear growth property of $K$ \eqref{HK1}, for each $0\leq u\leq T$,
\begin{align*}
|\Tilde{Y}(C_h(u))|&\leq \sum_{i=-N}^{-1} h\cdot |K(C_h(u),t_i,\Tilde{X}^h(C_h(u)+t_i))|  \leq \sum_{i=-N}^{-1} h\cdot K_3(1+|\Tilde{X}^h(C_h(u)+t_i)|)  \\
&\leq \sum_{i=-N}^{-1} \dfrac{r}{N}\cdot K_3(1+\sup_{x\in[-r,u-t_1]}|\Tilde{X}^h(x)|) \leq rK_3(1+\sup_{x\in[-r,u-h]}|\Tilde{X}^h(x)|) \leq rK_3(1+\Tilde{F}(u-h)).
\end{align*} 
\end{proof}
{\bf Proof of Lemma \ref{cotadeDYtilde-DXtilde}}
\begin{proof}
From Hypothesis 3, 
\begin{align*}
    |\Tilde{Y}(u)-\Tilde{Y}(v)|
   \leq & \sum_{i=-N}^{-1} h|K(u,t_i,\Tilde{X}^h(u+t_i))-K(v,t_i,\Tilde{X}^h(v+t_i))| \nonumber \\
        \leq  &\sum_{i=-N}^{-1} h(K_1|u-v| + K_2|\Tilde{X}^h(u+t_i)-\Tilde{X}^h(v+t_i)|) \nonumber \\
        =& hK_1 \sum_{i=-N}^{-1} |u-v| + K_2 \sum_{i=-N}^{-1}  \int_{t_i}^{t_{i+1}} |\Tilde{X}^h(u+C_h(s))-\Tilde{X}^h(v+C_h(s))|{d}s \nonumber \\
        = &rK_1|u-v| + K_2 \int_{-r}^{0} |\Tilde{X}^h(u+C_h(s))-\Tilde{X}^h(v+C_h(s))|{d}s.
        \end{align*}
\end{proof}

{\bf Proof of Lemma \ref{cota-Xtilde}}
\begin{proof}
From \eqref{EulerInterpolado} and  \eqref{eq5}, we get
\begin{align*}
&|\Tilde{X}^h(v)-\Tilde{X}^h(C_h(v))|\\
\le& \left|\int_{C_h(v)}^v b(C_h(u),\Tilde{X}^h(C_h(u)),\Tilde{Y}(C_h(u)))\,du\right| + \left|\int_{C_h(v)}^v \sigma(C_h(u),\Tilde{X}^h(C_h(u)),\Tilde{Y}(C_h(u)))\,dB^H(u)\right| \\
=& |b(C_h(v),\Tilde{X}^h(C_h(v)),\Tilde{Y}(C_h(v)))|(v-C_h(v))  + |\sigma(C_h(v),\Tilde{X}^h(C_h(v)),\Tilde{Y}(C_h(v)))|
|B^H(v)-B^H(C_h(v))| \\
\le & L_3(1+|\Tilde{X}^h(C_h(v))|+|\Tilde{Y}(C_h(v))|) \times \Big((v-C_h(v))
+ |B^H(v)-B^H(C_h(v))|\Big).
\end{align*}

Using \eqref{eqdefc0},
$
|B^H(v)-B^H(C_h(v))|
\le C(\omega)(v-C_h(v))^{H-\rho},
$
since $v-C_h(v)\le T^{1-H+\rho}(v-C_h(v))^{H-\rho}$ we obtain
\begin{align*}
|\Tilde{X}^h(v)-\Tilde{X}^h(C_h(v))|
&\le L_3\big(T^{1-H+\rho}+ C(\omega)\big)
\big(1+|\Tilde{X}^h(C_h(v))|+|\Tilde{Y}(C_h(v))|\big)
(v-C_h(v))^{H-\rho}.
\end{align*}
Using Lemma \ref{cotadeYtilde-Xtilde},
$
|\Tilde{Y}(C_h(v))|
\le rK_3(1+\Tilde{F}(v-h)),
$ and defining
\(
C(\omega)=L_3\big(T^{1-H+\rho}+ C(\omega)\big),
\)
we obtain the result.
\end{proof}

{\bf Proof of Lemma \ref{X-CHU-v}}
\begin{proof} 
 From \eqref{EulerInterpolado}, the  bounds stated in Hypothesis 1, \eqref{eqcotaintegral} and the fact that $C_h(w)=C_h(z)$ for $C_h(w)\leq z < w$, together with the observation that 
 $C_h(w)-C_h(z)\leq w-z+h$, we have 
\begin{align*}
      &  |\Tilde{X}^h(C_h(u))-\Tilde{X}^h(v)|
      \leq C(\omega)\left(\int_{v}^{C_h(u)} (1+|\Tilde{X}^h(C_h(w))|+|\Tilde{Y}(C_h(w))|){d}w \right.  \nonumber\\
        &+ \int_v^{C_h(u)} \dfrac{(1+|\Tilde{X}^h(C_h(w))|+|\Tilde{Y}(C_h(w))|)}{(w-v)^{\beta}}{d}w  \nonumber \\
        &\left.+ \int_v^{C_h(u)} \int_v^w \dfrac{|C_h(w)-C_h(z)|+|\Tilde{X}^h(C_h(w))-\Tilde{X}^h(C_h(z))|+|\Tilde{Y}(C_h(w))-\Tilde{Y}(C_h(z))|}{(w-z)^{\beta+1}}{d}z{d}w\right).
        \end{align*}
Using Remark \ref{u-cu-beta}, see \eqref{u-chu}, we derive the inequality
\begin{align*}
&  |\Tilde{X}^h(C_h(u))-\Tilde{X}^h(v)|
         \leq C(\omega)\left((C_h(u)-v)^{1-\beta} + \int_v^{C_h(u)} \dfrac{|\Tilde{X}^h(C_h(w))|}{(w-v)^{\beta}}{d}w + \int_v^{C_h(u)} \dfrac{|\Tilde{Y}(C_h(w))|}{(w-v)^{\beta}}{d}w \right.\nonumber \\
        & + h\int_v^{C_h(u)} (w-C_h(w))^{-\beta}{d}w + \int_v^{C_h(u)} \int_v^{C_h(w)} \dfrac{|\Tilde{X}^h(C_h(w))-\Tilde{X}^h(z)|}{(w-z)^{\beta+1}}{d}z{d}w \nonumber\\
        & + \int_v^{C_h(u)} \int_v^{C_h(w)} \dfrac{|\Tilde{X}^h(z)-\Tilde{X}^h(C_h(z))|}{(w-z)^{\beta+1}}{d}z{d}w  \left. + \int_v^{C_h(u)} \int_v^{C_h(w)} \dfrac{|\Tilde{Y}(C_h(w))-\Tilde{Y}(C_h(z))|}{(w-z)^{\beta+1}}{d}z{d}w\right).
    \end{align*}

Multiplying by $(u-v)^{-\beta-1}$ and integrating over $[0,C_h(u)]$, we obtain  that   
\begin{align}
\label{I4-1}
&\int_{0}^{C_h(u)}|\Tilde{X}^h(C_h(u))-\Tilde{X}^h(v)|(u-v)^{-\beta-1}{d}v \leq  C(\omega) \sum_{i=1}^7 Q^i(u), 
\end{align}
where
\begin{align*}
    Q^1(u) &= \int_{0}^{C_h(u)}(C_h(u)-v)^{1-\beta}(u-v)^{-\beta-1}dv, \nonumber\\
    Q^2(u) &= \int_{0}^{C_h(u)}(u-v)^{-\beta-1}\int_v^{C_h(u)} \dfrac{|\Tilde{X}^h(C_h(w))|}{(w-v)^{\beta}}{d}w{d}v, \nonumber\\
    Q^3(u) &= \int_0^{C_h(u)} (u-v)^{-\beta-1} \int_v^{C_h(u)} \dfrac{|\Tilde{Y}(C_h(w))|}{(w-v)^{\beta}}{d}w{d}v \nonumber\\
    Q^4(u) &= h\int_{0}^{C_h(u)}(u-v)^{-\beta-1}\int_v^{C_h(u)} (w-C_h(w))^{-\beta}{d}w{d}v, \nonumber\\
    Q^5(u) &= \int_{0}^{C_h(u)}(u-v)^{-\beta-1}\int_v^{C_h(u)} \int_v^{C_h(w)} \dfrac{|\Tilde{X}^h(C_h(w))-\Tilde{X}^h(z)|}{(w-z)^{\beta+1}}{d}z{d}w{d}v,\nonumber \\
    Q^6(u) &= \int_{0}^{C_h(u)}(u-v)^{-\beta-1}\int_v^{C_h(u)} \int_v^{C_h(w)} \dfrac{|\Tilde{X}^h(z)-\Tilde{X}^h(C_h(z))|}{(w-z)^{\beta+1}}{d}z{d}w{d}v \nonumber\\
    Q^7(u) &= \int_0^{C_h(u)} (u-v)^{-\beta-1} \int_v^{C_h(u)} \int_v^{C_h(w)} \dfrac{|\Tilde{Y}(C_h(w)) - \Tilde{Y}(C_h(z))|}{(w-z)^{\beta+1}}{d}z{d}w{d}v.
\end{align*}
Noting that $0 < \beta < 1/2$, $C_h(u)-v\leq u-v,$ it follows that
\begin{align}\label{Q1}
    Q^1(u) \leq \int_0^{C_h(u)} (u-v)^{-2\beta}{d}v \leq \dfrac{u^{1-2\beta}}{1-2\beta}. 
\end{align}
Exchanging the integration limits in $Q^2(u)$ and $Q^3(u)$, making a  change of variable as in \cite{Mishura2008}, and taking  $C=\int_0^{\infty} (1+y)^{-\beta-1}y^{-\beta}{d}y$, we obtain
\begin{align}\label{Q2}
    Q^2(u) &= \int_{0}^{C_h(u)} |\Tilde{X}^h(C_h(w))|\int_0^{w} (u-v)^{-\beta-1}(w-v)^{-\beta}{d}v{d}w \leq C\int_{0}^{C_h(u)} \frac{|\Tilde{X}^h(C_h(w))|}{(u-w)^{2\beta}}{d}w.
\end{align}
Since $\widetilde F$ is nondecreasing, we have 
$\widetilde F(w-h)\le \widetilde F(w)$, and
\begin{align}\label{Q3}
    Q^3(u) &= \int_0^{C_h(u)} |\Tilde{Y}(C_h(w))| \int_0^w (u-v)^{-\beta-1}(w-v)^{-\beta}{d}v{d}w \nonumber \\
    &\leq C\int_0^{C_h(u)} rK_3(1+\Tilde{F}(w-h))(u-w)^{-2\beta}{d}w \nonumber \\
    &\leq C \left(u^{1-2\beta} + \int_0^{C_h(u)} \frac{\Tilde{F}(w)}{(u-w)^{2\beta}}{d}w\right).
    \end{align}
From Lemma 1 in \cite{Mishura2008}, it follows directly that $ Q^4(u) \leq C h^{1-\beta}$.

Exchanging the integration limits in $Q^5(u)$, we obtain
\begin{align}\label{Q5}
    Q^5(u) &\leq \int_{0}^{C_h(u)} \int_0^{C_h(w)} \int_0^{w} (u-v)^{-\beta-1}\dfrac{|\Tilde{X}^h(C_h(w))-\Tilde{X}^h(z)|}{(w-z)^{\beta+1}}{d}v{d}z{d}w \nonumber \\
    &\leq \dfrac{1}{\beta}\int_{0}^{C_h(u)} (u-w)^{-\beta} \int_0^{C_h(w)}\dfrac{|\Tilde{X}^h(C_h(w))-\Tilde{X}^h(z)|}{(w-z)^{\beta+1}}   {d}z{d}w.
\end{align}
For $Q^6(u)$, we use Lemma \ref{cota-Xtilde}, \eqref{Ftilde} and  
\eqref{u-chu} to the interval $[v,C_h(u)]$ to deduce
\begin{align}\label{Q6}
    Q^6(u) &\leq C(\omega)  h^{H-\rho-\beta}\left(1 + \Tilde{F}(C_h(u)) \right). 
\end{align}
Finally, we estimate $Q^7(u)$. Using the discrete structural estimate of the delay functional (Lemma \ref{cotadeDYtilde-DXtilde}), which mirrors 
Lemma \ref{lemmaY},
\begin{align}\label{Q7}
&Q^7(u)  \leq \int_0^{C_h(u)} (u-v)^{-\beta-1}\int_v^{C_h(u)} \int_v^{C_h(w)} \dfrac{rK_1(C_h(w)-C_h(z))}{(w-z)^{\beta+1}}{d}z{d}w{d}v \nonumber  \\
&+ \int_0^{C_h(u)}\!\!\! (u-v)^{-\beta-1}\!\!\!\int_v^{C_h(u)}\!\! \int_v^{C_h(w)}\!\! \int_{-r}^0 \!\dfrac{K_2|\Tilde{X}^h(C_h(w)+C_h(s))-\Tilde{X}^h(C_h(z)+C_h(s))|}{(w-z)^{\beta+1}} {d}s {d}z{d}w{d}v. 
\end{align}
Using $C_h(w)-C_h(z)\le (w-z)+h$ and arguing as in the estimate of $Q^4(u)$,
we obtain the bound
\begin{align}
\label{Q7-3}
Q^7(u)&\le C
+ rK_1 h \int_0^{C_h(u)} (u-v)^{-\beta-1}
\int_v^{C_h(u)} (w-C_h(w))^{-\beta}\,dw\,dv \nonumber\\
&\quad
+ \frac{K_2}{\beta}
\int_0^{C_h(u)} (u-w)^{-\beta}
\int_0^{C_h(w)} \int_{-r}^0
\dfrac{|\Tilde X^h(C_h(w)+C_h(s))
-\Tilde X^h(C_h(z)+C_h(s))|}
{(w-z)^{\beta+1}}
\,ds\,dz\,dw \nonumber\\
&\leq\! C\!+\! \dfrac{K_2}{\beta}\!\int_0^{C_h(u)}\!\! (u-w)^{-\beta}\int_0^{C_h(w)}\!\! \int_{-r}^0  \dfrac{ |\Tilde{X}^h(C_h(w)+C_h(s))-\Tilde{X}^h(C_h(z)+C_h(s))|}{(w-z)^{\beta+1}} ds {d}z{d}w. 
\end{align}

By \eqref{I4-1} and the estimates given in \eqref{Q1} - \eqref{Q7-3},
\begin{align*}
\hspace{-2cm}&\int_0^{C_h(u)} |\Tilde{X}^h(C_h(u))-\Tilde{X}^h(v)|(u-v)^{-\beta-1}{d}v \\
  \leq & C(\omega)\left(1 + \int_0^{C_h(u)} |\Tilde{X}^h(w)|(u-w)^{-2\beta}{d}w + \int_0^{C_h(u)} \Tilde{F}(w)(u-w)^{-2\beta}{d}w \right. \notag\\
 &+ \int_0^{C_h(u)} (u-w)^{-\beta}\int_0^{C_h(w)} \dfrac{|\Tilde{X}^h(C_h(w))-\Tilde{X}^h(z)|}{(w-z)^{\beta+1}}{d}z{d}w + \Tilde{F}(C_h(u))h^{H-\rho-\beta} \notag\\
 & \left. + \int_0^{C_h(u)} (u-w)^{-\beta}\int_0^{C_h(w)} \int_{-r}^0 \dfrac{ |\Tilde{X}^h(C_h(w)+C_h(s))-\Tilde{X}^h(C_h(z)+C_h(s))|{d}s}{(w-z)^{\beta+1}}{d}z  {d}w\right).\notag
\end{align*}
\end{proof}

{\bf Proof of Lemma \ref{Cota- I(w)}}
\begin{proof} 
For  $-r<a<b<0$, let 
$J(a,b):= \displaystyle \int_{a}^{b}
\big|\widetilde X^h(C_h(w)+C_h(s))-\widetilde X^h(C_h(z)+C_h(s))\big|ds$.
We partition the integration interval with respect to $s \in [-r,0]$ into  four subintervals  as

\[
[-r,-C_h(w)],\quad [-C_h(w), -z ] ,\quad [-z,-C_h(z)],\quad [-C_h(z),0].
\]

\textbf{Case 1} (First subinterval $[-r,-C_h(w)]$). Within this subinterval, the arguments of $\widetilde{X}^h$ inside the integral $J(-r,-C_h(w))$ are negative or zero (lying within the region where the initial condition $\phi$ is applied). By invoking the Hölder continuity and local regularity assumptions \eqref{eqphi}, we obtain the following point-wise control:

\[
\big|\widetilde X^h(C_h(w)+C_h(s))-\widetilde X^h(C_h(z)+C_h(s))\big|
\le C_1(\omega)\,|\,C_h(w)-C_h(z)\,|^{H-\rho}.
\]
Then
\[
J(-r, -C_h(w)) \le r\,C_1(\omega)\,|C_h(w)-C_h(z)|^{H-\rho} 
\]
Using $C_h(w)-C_h(z)\le (w-z)+h$ and arguing as in Remark \ref{u-cu-beta},
we obtain
\begin{align}\label{Caso1}
\int_{0}^{C_h(w)}(w-z)^{-\beta-1} J(-r, -C_h(w)) dz
&\leq rC_1(\omega)  \int_{0}^{C_h(w)}(w-z)^{-\beta-1}
(C_h(w)- C_h(z))^{H-\rho }dz \nonumber \\
&\leq rC_1(\omega) \left(  w^{-\beta + H - \rho} + h ^{H - \rho}(w- C_h(w))^{-\beta} ) \right). 
\end{align}

\textbf{Case 2} (Second subinterval $[-C_h(w), -z ]$). In this case, $-C_h(w) \leq s \leq -z \leq -C_h(z)$,
which implies that $C_h(w) + C_h(s) > 0$ and $C_h(s) + C_h(z) < 0$.
By adding and subtracting $\widetilde{X}^h(0)$ and applying the triangle inequality, we obtain:
\begin{align*}
&\big|\widetilde X^h(C_h(w)+C_h(s))-\widetilde X^h(C_h(z)+C_h(s))\big|
\\
 \leq & C_1(\omega)\left( \big|\widetilde X^h(C_h(w)+C_h(s))-\widetilde X^h(0)\big| + |C_h(w)-C_h(z)|^{H-\rho} \right)  \nonumber \\
 \leq & C_1(\omega) ( 1 + \widetilde F(C_h(w)) + |C_h(w)-C_h(z)|^{H-\rho} ).
\end{align*}
Recall that the length of this interval is $C_h(w)-z$.
Using $C_h(w)-C_h(z)\le (w-z)+h$ and the monotonicity of $\widetilde F$, we obtain
\begin{align}\label{Caso2-1}
&\int_{0}^{C_h(w)}(w-z)^{-\beta-1}
\int_{-C_h(w)}^{-z}
\big|\widetilde X^h(C_h(w)+C_h(s))-\widetilde X^h(C_h(z)+C_h(s))\big|dsdz \nonumber \\
\leq &C_1(\omega) \int_{0}^{C_h(w)} (C_h(w)-z) (w-z)^{-\beta-1}
(C_h(w)- C_h(z))^{H-\rho } dz \nonumber \\
\leq &C_1(\omega)\big( 1 + \widetilde F(C_h(w))  \big).
\end{align}

\textbf{Case 3} (Third subinterval $[-z,-C_h(z)]$).
In this case, $-C_h(w) \leq -z \leq s \leq -C_h(z)$.
Hence, we have $C_h(w)+C_h(s)>0$ and $C_h(s)+C_h(z)<0$. Proceeding as in Case 2 and using the triangle inequality, we obtain
\begin{align*}
&\big|\widetilde X^h(C_h(w)+C_h(s))-\widetilde X^h(C_h(z)+C_h(s))\big|\\
\le &C_1(\omega)\left( \big|\widetilde X^h(C_h(w)+C_h(s))-\widetilde X^h(0)\big| + |C_h(w)-C_h(z)|^{H-\rho} \right)  \nonumber \\
 \leq & C_1(\omega) ( 2\widetilde F(C_h(w)) + |C_h(w)-C_h(z)|^{H-\rho} ).
\end{align*}

Since we work in the interval $[-z,-C_h(z)]$, whose length is smaller than $h$, we get 
\begin{align}\label{Caso3-1}
&\int_{0}^{C_h(w)}(w-z)^{-\beta-1}
\int_{-z}^{-C_h(z)}
\big|\widetilde X^h(C_h(w)+C_h(s))-\widetilde X^h(C_h(z)+C_h(s))\big|dsdz \nonumber \\
\leq &hC_1(\omega) \int_{0}^{C_h(w)}(w-z)^{-\beta-1} 
(C_h(w)- C_h(z))^{H-\rho } dz \nonumber \\
\leq & h C_1(\omega) \left( 
w^{H-\rho-\beta} 
+ (w-C_h(w))^{-\beta}  (1+\widetilde F(C_h(w)))\right).
\end{align}

\textbf{Case 4} (Final subinterval $[-C_h(z),0]$). On this subinterval it holds that
$-C_h(z) \leq s \leq 0$,
 $-C_h(z) \le C_h(s)$
 and
$-C_h(w) \le C_h(s) $.
Then  $C_h(w) + C_h(s) >0$ and $C_h(s) + C_h(z) >0$. 
We leave the integral in its original form, 
\begin{align}\label{caso4} 
& 
\int_{-C_h(z)}^{0}
\big|\widetilde X^h(C_h(w)+C_h(s))-\widetilde X^h(C_h(z)+C_h(s))\big|ds.
\end{align}
\\
In summary, from \eqref{Caso1}, \eqref{Caso2-1}, \eqref{Caso3-1}, and \eqref{caso4}, we obtain
\begin{align*}
    I(w) \leq & C_1(\omega) \left( 1+ \widetilde F(C_h(w))  +  h ^{H - \rho}(w- C_h(w))^{-\beta}  +  h  (w-C_h(w))^{-\beta}  \widetilde F(C_h(w)) \right. \nonumber \\
& + \left.
\int_0^{C_h(w)} (w-z)^{-\beta-1} \int_{-C_h(z)}^{0}
\big|\widetilde X^h(C_h(w)+C_h(s))-\widetilde X^h(C_h(z)+C_h(s))\big|\,dsdz
\right).
\end{align*}
\end{proof}

{\bf Proof of Lemma \ref{dif-dif}}
\begin{proof}
We observe that 
\begin{align*} 
|\tilde{X}^h(C_h(w)+C_h(s))-\tilde{X}^h(C_h(z)+C_h(s))| \leq & |\tilde{X}^h(C_h(w)+C_h(s))-\tilde{X}^h(z+C_h(s))| \\
&+ |\tilde{X}^h(z+C_h(s))-\tilde{X}^h(C_h(z)+C_h(s))|.
\end{align*}

Noting that $C_h(z)+C_h(s)=C_h(z+C_h(s)),$ then from Lemma \ref{cota-Xtilde},
\begin{align*}
   & |\tilde{X}^h(z+C_h(s))-\tilde{X}^h(C_h(z)+C_h(s))| \\
    \leq& C(\omega)(1 + \tilde{F}(C_h(z)) + rK_3(1+\tilde{F}(C_h(z))))h^{H-\rho} \\
    \leq& C(\omega)(1 + \tilde{F}(C_h(z)))h^{H-\rho}.
\end{align*}

Integrating with respect to  $s$ over $[-C_h(z),0]$,   multiplying by $(w-z)^{-\beta-1}$ and integrating with respect to $z$ over $[0,C_h(w)]$ we  obtain 
\begin{align*}
    &\int_0^{C_h(w)}(w-z)^{-\beta-1}\int_{-C_h(z)}^0 |\tilde{X}^h(C_h(w)+C_h(s))-\tilde{X}^h(z+C_h(s))|\,dsdz \\
    &\leq C(\omega)h^{H-\rho}\int_0^{C_h(w)}(w-z)^{-\beta-1}(1 + \tilde{F}(C_h(z)))\int_{-C_h(z)}^0 \,dsdz \\
    &\leq C(\omega)(1+\tilde{F}(C_h(w)))h^{H-\rho} (w-C_h(w))^{-\beta}.
\end{align*}

Let us analyze now $|\tilde{X}^h(C_h(w)+C_h(s))
-\tilde{X}^h(z+C_h(s))|$.
Applying inequality \ref{eqcotaintegral} on the interval 
$[z+C_h(s),\, C_h(w)+C_h(s)]$, we obtain
\begin{align}\label{CHS-new}
&|\tilde{X}^h(C_h(w)+C_h(s))
-\tilde{X}^h(z+C_h(s))|  \nonumber \\
\le& C(\omega)\Bigg[
\int_{z+C_h(s)}^{C_h(w)+C_h(s)}
\frac{1+|\tilde X^h(C_h(v))|+|\tilde Y(C_h(v))|}
{(v-(z+C_h(s)))^\beta} dv +
\int_{z+C_h(s)}^{C_h(w)+C_h(s)}
\int_{z+C_h(s)}^{C_h(u)} \nonumber \\
& \frac{|C_h(u)-C_h(v)|
+|\tilde X^h(C_h(u))-\tilde X^h(C_h(v))|
+|\tilde Y(C_h(u))-\tilde Y(C_h(v))|}
{(u-v)^{\beta+1}}
dvdu \Bigg].
\end{align}

Let us denote by $I_c$ the double integral in \eqref{CHS-new}. We can decompose $I_c$ as follows:
\begin{align*}
I_c \leq& \int_{z+C_h(s)}^{C_h(w)+C_h(s)} \int_{z+C_h(s)}^{C_h(u)} \dfrac{|C_h(u)-C_h(v)|}{(u-v)^{\beta+1}}{d}v{d}u + \int_{z+C_h(s)}^{C_h(w)+C_h(s)} \int_{z+C_h(s)}^{C_h(u)} \dfrac{|\tilde{X}^h(C_h(u))-\tilde{X}^h(C_h(v))|}{(u-v)^{\beta+1}}{d}v{d}u \nonumber \\
&+ \int_{z+C_h(s)}^{C_h(w)+C_h(s)} \int_{z+C_h(s)}^{C_h(u)} \dfrac{|\tilde{Y}(C_h(u))-\tilde{Y}(C_h(v))|}{(u-v)^{\beta+1}}{d}v{d}u \nonumber \\
=:& I_{c_1} + I_{c_2} + I_{c_3}.
\end{align*}
Next, for \(I_{c_1}\), and using the inequality \(C_h(u)-C_h(v)\le u-v+h\), we deduce
\begin{align}\label{IC-1}
I_{c_1}&=\int_{z+C_h(s)}^{C_h(w)+C_h(s)} \int_{z+C_h(s)}^{C_h(u)} \dfrac{|C_h(u)-C_h(v)|}{(u-v)^{\beta+1}}{d}v{d}u \nonumber  \\
&\leq \int_{z+C_h(s)}^{C_h(w)+C_h(s)} \int_{z+C_h(s)}^{C_h(u)} (u-v)^{-\beta}{d}v{d}u + \int_{z+C_h(s)}^{C_h(w)+C_h(s)} \int_{z+C_h(s)}^{C_h(u)} \dfrac{h}{(u-v)^{\beta+1}}{d}v{d}u \nonumber \\
&\leq \dfrac{1}{1-\beta} \int_{z+C_h(s)}^{C_h(w)+C_h(s)} (u-z-C_h(s))^{1-\beta}{d}u + \dfrac{h}{\beta}\int_{z+C_h(s)}^{C_h(w)+C_h(s)} (u-C_h(u))^{-\beta}{d}u \nonumber \\
&\leq C (C_h(w)-z)^{2-\beta} + \dfrac{h}{\beta}\int_{z+C_h(s)}^{0} (u-C_h(u))^{-\beta}{d}u + \dfrac{h}{\beta}\int_{0}^{C_h(w)+C_h(s)} (u-C_h(u))^{-\beta}{d}u.
\nonumber \\
&\leq C \left((C_h(w)-z)^{2-\beta} + (C_h(w)-z)h^{1-\beta} \right),
\end{align}
where the last inequality follows from the expansion in \eqref{u-chu} and from the fact that this inequality remains valid for negative values, since the interval $[-r,0]$ is partitioned into the subintervals $t_i$ together with \eqref{u-chu}.  
For the second integral $I_{c_2}$, we further decompose it into two integrals:
\begin{align}\label{IC-2}
I_{c_2}=& \int_{z+C_h(s)}^{C_h(w)+C_h(s)} \int_{z+C_h(s)}^{C_h(u)} \dfrac{|\tilde{X}^h(C_h(u))-\tilde{X}^h(C_h(v))|}{(u-v)^{\beta+1}}{d}v{d}u \nonumber \\
\leq & \int_{z+C_h(s)}^{C_h(w)+C_h(s)} \int_{z+C_h(s)}^{C_h(u)} \dfrac{|\tilde{X}^h(C_h(u))-\tilde{X}^h(v)|}{(u-v)^{\beta+1}}{d}v{d}u 
+ \int_{z+C_h(s)}^{C_h(w)+C_h(s)} \int_{z+C_h(s)}^{C_h(u)} \dfrac{|\tilde{X}^h(v)-\tilde{X}^h(C_h(v))|}{(u-v)^{\beta+1}}{d}v{d}u.
\end{align}

From Lemma \eqref{cota-Xtilde}  
\begin{equation*}
     |\Tilde{X}^h(v)-\Tilde{X}^h(C_h(v))| \leq C(\omega) \left(1+|\Tilde{X}^h(C_h(v))|+rK_3(1+\Tilde{F}(v-h))\right)(v-C_h(v))^{H-\rho},
\end{equation*}
then
\begin{align}\label{IC2-1}
&\int_{z+C_h(s)}^{C_h(w)+C_h(s)} \int_{z+C_h(s)}^{C_h(u)} \dfrac{|\tilde{X}^h(v)-\tilde{X}^h(C_h(v))|}{(u-v)^{\beta+1}}{d}v{d}u \nonumber  \\ 
\leq& \int_{z+C_h(s)}^{C_h(w)+C_h(s)} \int_{z+C_h(s)}^{C_h(u)} \dfrac{C(\omega) \left(1+|\tilde{X}^h(C_h(v))|+rK_3(1+\tilde{F}(v-h))\right)h^{H-\rho}}{(u-v)^{\beta+1}}{d}v{d}u \nonumber \\
\leq& \int_{z+C_h(s)}^{C_h(w)+C_h(s)} \int_{z+C_h(s)}^{C_h(u)} \dfrac{C(\omega) \left(1+\tilde{F}^h(C_h(u))+rK_3(1+\tilde{F}(C_h(u)))\right)h^{H-\rho}}{(u-v)^{\beta+1}}{d}v{d}u \nonumber \\
\leq& C(\omega) (1+rK_3)(1+\tilde{F}(C_h(w)))h^{H-\rho}\int_{z+C_h(s)}^{C_h(w)+C_h(s)} \int_{z+C_h(s)}^{C_h(u)} (u-v)^{-\beta-1}{d}v{d}u \nonumber \\
\leq& C(\omega) h^{H-\rho-\beta}(1+\tilde{F}(C_h(w))) (C_h(w)-z).
\end{align}
Therefore, combining \eqref{IC-1}, \eqref{IC-2} and \eqref{IC2-1}, the equation \eqref{CHS-new} takes the form
\begin{eqnarray*}
|\tilde{X}^h(C_h(w)+C_h(s))-\tilde{X}^h(z+C_h(s))| \leq C(\omega)\left((C_h(w)-z)^{1-\beta} + \int_{z}^{C_h(w)} \dfrac{\tilde{F}(C_h(v))}{(v-z)^{\beta}}{d}v \right.  + (C_h(w)-z)^{2-\beta}\nonumber \\
 + (C_h(w)-z)h^{1-\beta} 
+ \int_{z+C_h(s)}^{C_h(w)+C_h(s)} \int_{z+C_h(s)}^{C_h(u)} \dfrac{|\tilde{X}^h(C_h(u))-\tilde{X}^h(v)|}{(u-v)^{\beta+1}}{d}v{d}u  \nonumber \\
+ h^{H-\rho-\beta}(1+\tilde{F}(C_h(w)) (C_h(w)-z) + \left. \int_{z+C_h(s)}^{C_h(w)+C_h(s)} \int_{z+C_h(s)}^{C_h(u)} \dfrac{|\tilde{Y}(C_h(u))-\tilde{Y}(C_h(v))|}{(u-v)^{\beta+1}}{d}v{d}u \right).
\end{eqnarray*}
From Lemma \eqref{Cota- I(w)}, integrating with respect to $s$ over $[-C_h(z),0]$, and then multiplying by $(w-z)^{-\beta-1}$ and integrating with respect to $z$ over $[0,C_h(w)]$, we obtain
\begin{align*}
&    \int_0^{C_h(w)}\int_{-C_h(z)}^0 \dfrac{ |\tilde{X}^h(C_h(w)+C_h(s))-\tilde{X}^h(C_h(z)+C_h(s))|}{(w-z)^{\beta+1}} {d}s{d}z \leq \sum_{i=1}^6 R^i(w),
\end{align*}
where
\begin{align*}
R^1(w) &= \int_{0}^{C_h(w)} (w-z)^{-\beta-1}\int_{-C_h(z)}^0 (C_h(w)-z)^{1-\beta}{d}s{d}z,\\
R^2(w) &= \int_{0}^{C_h(w)}(w-z)^{-\beta-1}\int_{-C_h(z)}^0 \int_{z}^{C_h(w)} \dfrac{\tilde{F}(C_h(v))}{(v-z)^{\beta}}{d}v{d}s{d}z, \\
R^3(w) &= h^{1-\beta}\int_0^{C_h(w)} (w-z)^{-\beta-1} (C_h(w)-z) C_h(z) dz, \\
{R^4(w)} &= \int_{0}^{C_h(w)}(w-z)^{-\beta-1} \int_{-C_h(z)}^0 \int_{z+C_h(s)}^{C_h(w)+C_h(s)} \int_{z+C_h(s)}^{C_h(u)} \dfrac{|\tilde{X}^h(C_h(u))-\tilde{X}^h(v)|}{(u-v)^{\beta+1}}{d}v{d}u{d}s{d}z, \\
R^5(w) &= \int_{0}^{C_h(w)}(w-z)^{-\beta-1} \int_{-C_h(z)}^0 C(1+rK_3)(1+\tilde{F}(C_h(w)))h^{H-\rho-\beta} (C_h(w)-z) {d}s{d}z, \\
R^6(w) &= \int_0^{C_h(w)} (w-z)^{-\beta-1} \int_{-C_h(z)}^0 \int_{z+C_h(s)}^{C_h(w)+C_h(s)} \int_{z+C_h(s)}^{C_h(u)} \dfrac{|\tilde{Y}(C_h(u)) - \tilde{Y}(C_h(v))|}{(u-v)^{\beta+1}}{d}v{d}u{d}s{d}z.
\end{align*}

For the term $R^1$
 we obtain the following.
\begin{align}\label{R1}
    R^1(w) &=  \int_{0}^{C_h(w)} (w-z)^{-\beta-1}\int_{-C_h(z)}^0 (C_h(w)-z)^{1-\beta}{d}s{d}z\leq \int_0^{C_h(w)} z(w-z)^{1-\beta}{d}z \leq C w^{3-\beta}.
\end{align}
A direct computation of   $R^2$ yields (see equation (17) in \cite{Mishura2008}
\begin{align}\label{R2}
    R^2(w) &= \int_{0}^{C_h(w)} C_h(z)|\Tilde{F}^h(C_h(z))|\int_z^{C_h(w)} (w-z)^{-\beta-1}(v-z)^{-\beta}{d}v{d}z \leq C\int_{0}^{C_h(w)} \frac{|\Tilde{F}^h(C_h(z))|}{(w-z)^{2\beta}}{d}z.
\end{align}
For $R^3,$ 
\begin{align}\label{R3}
    R^3(w) &= h^{1-\beta}\int_0^{C_h(w)} (w-z)^{-\beta-1} (C_h(w)-z) C_h(z)\, dz \leq Ch^{1-\beta}
\end{align}

For the term $R^4$, instead of computing explicitly the stepwise
decomposition in $s$ and the discrete summation, we observe that after
rearranging the order of integration its structure is
\begin{align}\label{R4}
R^4(w) & = 
\int_{0}^{C_h(w)} (w-u)^{-\beta}
\int_{0}^{C_h(u)}
\frac{|\tilde{X}^h(C_h(u))-\tilde{X}^h(C_h(v))|}
{(u-v)^{\beta+1}}
\, dv\, du.
\end{align}
This coincides with the fractional convolution structure appearing
in inequality \ref{eqcotaintegral}, and therefore we keep this operator form
without expanding the discrete summation further.

For $R^5$ we have 
\begin{align}\label{R5}
R^5(w) &= \int_{0}^{C_h(w)}(w-z)^{-\beta-1} \int_{-C_h(z)}^0 C(1+rK_3)(1+\tilde{F}(C_h(w)))h^{H-\rho-\beta} (C_h(w)-z) {d}s{d}z \nonumber \\
&\leq C(1+rK_3) h^{H-\rho -\beta} w (1+\tilde{F}(w)).
\end{align}


Finally, for $R^6$ we use Lemma \ref{cotadeDYtilde-DXtilde}, which yields
\[ |\Tilde{Y}(C_h(u))-\Tilde{Y}(C_h(v))|\le {rK_1 |u-v| +h} + K_2\int_{-r}^{0}|\Tilde{X}^h(C_h(u)+C_h(s))-\Tilde{X}^h(C_h(v)+C_h(s))|\,ds.\]

Substituting into $R^6(w)$ we obtain
\begin{align*}
R^6(w)
\le&
C \int_{0}^{C_h(w)} (w-u)^{-\beta}
\int_{0}^{C_h(u)}
\frac{|u-v| +h}{(u-v)^{\beta+1}}
\,dv\,du \\
&
+
C \int_{0}^{C_h(w)} (w-u)^{-\beta}
\int_{0}^{C_h(u)}
\int_{-r}^{0}
\frac{
|\Tilde{X}^h(C_h(u)+C_h(s))
-\Tilde{X}^h(C_h(v)+C_h(s))|
}{(u-v)^{\beta+1}}
\,ds\,dv\,du.
\end{align*}
The first part of $R^6(w)$
\begin{equation}\label{R6-0}
\int_0^{C_h(w)} (w-u)^{-\beta}\int_{0}^{C_h(u)} (u-v)^{-\beta} 
\,dv  \,du + h \int_{0}^{C_h(w)} (w-u) ^{-\beta} \int_0^{C_h(u)}(u-v)^{-1-\beta} \,dv  \,du 
\le
C(1 + h^{1-\beta}).
\end{equation}

Indeed, a direct computation in the first integral leads to the inequality $C w^{2 - 2 \beta}$. For the  second integral, using Remark \ref{u-cu-beta}, then we split the first integral into the intervals $[t_k,t_{k+1})$, where $C_h(u)$ is constant and equal to $t_k$.
Then, after the change of variable $u= t_k+x$, we estimate the singular term $x^{-\beta}$
 explicitly and sum the resulting bounds over all subintervals the integral reduces to the inequality $C h^{1-\beta}$.

For $u \in [0,T]$ define
\[
I(u)
=
\int_{0}^{C_h(u)} (u-v)^{-\beta-1}
\int_{-r}^{0}
|\widetilde X^h(C_h(u)+C_h(s))
-\widetilde X^h(C_h(v)+C_h(s))|
\,ds\,dv.
\]
Then observe that this coincides with the operator studied in Lemma \ref{Cota- I(w)} (with $w$ replaced by $u$). Therefore
\[
R^6(w)
\le
C(\omega)\left(
1 + h^{1-\beta}
+ \int_{0}^{C_h(w)} (w-u)^{-\beta} I(u)\,du
\right).
\]
Applying the bound of Lemma \ref{Cota- I(w)} and estimating each convolution term separately yields
\begin{align}\label{R6}
&R^6(w)\leq
C(\omega)\Bigg[
1
+ h^{1-\beta}
+ h^{H-\rho-\beta}
+ h^{1-\beta}\widetilde F(C_h(w))
+ \widetilde F(C_h(w)) w^{1-\beta}
\Big) \nonumber \\
&+\! 
\int_{0}^{C_h(w)}\!\!\! (w-u)^{-\beta}\!\!
\int_0^{C_h(u)}\!\!\! (u-v)^{-\beta-1}\!\!
\int_{-C_h(v)}^{0}\!\!
\big|\widetilde X^h(C_h(u)+C_h(s))
-\widetilde X^h(C_h(v)+C_h(s))\big|
\,ds\,dv\,du\Bigg].
\end{align}

Finally, from \eqref{R1}, \eqref{R2}, \eqref{R3}, \eqref{R4}, \eqref{R5}, and  \eqref{R6}
we obtain the result.

\end{proof}

{\bf Proof of Corollary \ref{Cor-cota-nualart}}

Assume Hypothesis \ref{HM2} on $\sigma$ and suppose that
$|x_i| \le C$, $|y_i|\le C$, $i=1,\dots,4$. Then
\begin{align*}
&|\sigma(t_1,x_1,y_1)-\sigma(t_1,x_2,y_2)
-\sigma(t_2,x_3,y_3)
+\sigma(t_2,x_4,y_4)|
\\
\le&
L_2 |x_1-x_2-x_3+x_4|
+ L_2 |y_1-y_2-y_3+y_4|+ L_1 |x_1-x_2|\,|t_1-t_2|+ L_1 |y_1-y_2|\,|t_1-t_2|
\\
&
+ L_2 |x_1-x_2|
\big(
|x_2-x_4|
+ |x_1-x_3|
+ |y_2-y_4|
+ |y_1-y_3|
\big)
\\
&+ L_2 |y_1-y_2|
\big(
|x_2-x_4|
+ |x_1-x_3|
+ |y_2-y_4|
+ |y_1-y_3|
\big).
\end{align*}

\begin{proof}
Using the mean value theorem in several variables,
\[
\sigma(t,x_1,y_1)-\sigma(t,x_2,y_2)
=
\int_0^1
(x_1-x_2,y_1-y_2)
\cdot
\nabla \sigma
\big(
t,
\xi(\lambda),
\eta(\lambda)
\big)
\, d\lambda,
\]
where $
\xi(\lambda)=(1-\lambda)x_2+\lambda x_1
$ and $\eta(\lambda)=(1-\lambda)y_2+\lambda y_1.
$
Applying this identity at times $t_1$ and $t_2$ and subtracting,
we obtain
\begin{align*}
&\sigma(t_1,x_1,y_1)-\sigma(t_1,x_2,y_2) -\sigma(t_2,x_3,y_3)
+\sigma(t_2,x_4,y_4)\\
=&
\int_0^1
(x_1-x_2,y_1-y_2)
\cdot
\nabla \sigma(t_1,\xi_1,\eta_1)
\, d\lambda
-
\int_0^1
(x_3-x_4,y_3-y_4)
\cdot
\nabla \sigma(t_2,\xi_2,\eta_2)
\, d\lambda.
\end{align*}

Adding and subtracting suitable terms and using the Lipschitz
assumptions on $\nabla \sigma$, we obtain
\begin{align*}
&|\sigma(t_1,x_1,y_1)-\sigma(t_1,x_2,y_2) -\sigma(t_2,x_3,y_3)
+\sigma(t_2,x_4,y_4)|\\
\le&
C
\Big(
|x_1-x_2-x_3+x_4|
+
|y_1-y_2-y_3+y_4|
\Big) +
C |t_1-t_2|
\big(
|x_1-x_2|
+
|y_1-y_2|
\big)   \\
&+
C (|x_1-x_2|+|y_1-y_2|)
\big(
|x_2-x_4|
+
|x_1-x_3|
+
|y_2-y_4|
+
|y_1-y_3|
\big),
\end{align*}
which proves the result.
\end{proof}

\section{Technical Estimates for the Solution}\label{Tec-Est-Sol}                                                                                                                                                             
In this section, we prove the Proposition \ref{cotasapriori}. First, we establish some technical lemmas, which provide a priori estimates for the process $X$. Their proofs are deferred to the end of this section.

Let us introduce the following notation. For $0 < u <v \leq T $ and $s \in [-r,0]$,
\begin{align}
\label{notation}
    |X^*(u)| &:= \displaystyle \sup_{-r \leq z \leq  u} |X(z)|. \nonumber \\
        \varphi_{u,v}^X  &:=  \frac{ |X(u)-X(v)|}{(u-v)^{\beta+1}}. \nonumber\\ 
 \Gamma_{u,v}(s)   &:=    \frac{|X(u+s) - X(v+s)|}{(u-v)^{\beta+1}} 
\end{align}

\begin{lemma}\label{DifX-Casos}
Let $0 < u <v \leq T $ and $s \in [-r,0]$.

Then, 
\begin{align*}
    \int_0^u \int_{-r}^0  \frac{|X(u+s) - X(v+s)|}{(u-v)^{\beta+1}} ds{d}v  
   \leq & 
    C\left( 1  + \int_0^u \frac{|X^*(v)|}{v^{-\beta}} dv+\int_0^u \int_{-v}^{0}  \frac{|X(u+s) - X(v+s)|}{(u-v)^{\beta+1}} ds dv \right). 
\end{align*}

\end{lemma}

\begin{lemma}\label{CA-XU-XV}
Let $0\leq v <u \leq T$ and $\varphi_{u,v}^X$ be defined by \eqref{notation}, then

\begin{align*} 
\int_0^u \varphi_{u,v}^X dv  
\leq& C\left( 1 + \int_0^u   |X^*(a)| (u-a)^{-2\beta}  da  +
 \int_{0}^{u} (u-a)^{-\beta}\int_{0}^{a}
\varphi_{a,z}^X dz da \right. \nonumber \\ 
& \quad +
\left.\int_0^u (u-a)^{-\beta}  \int_0^a \int_{-a}^0 \Gamma_{a,z}(s) ds dz da \right).
\end{align*}
\end{lemma}

\begin{lemma}\label{lema6-3}
Let $\Gamma_{u,v}(s)$ defined by \eqref{notation}, then

\begin{align*}
      \int_0^u \int_{-r}^0 \Gamma_{u,v}(s) ds{d}v  
&\leq C\left( 1 
+ \int_0^u |X^*(v)|v^{-\beta} dv + 
\int_{0}^{u} |X^*(a)| (u-a)^{-2\beta} da.
 \right.\nonumber  \\
&+
\int_0^u (u-a)^{-\beta} \int_0^a
\varphi^X_{a,z} dz da +
\left. \int_0^u (u-a)^{-\beta}  \int_0^a \int_{-a}^0 \Gamma_{a,z}(s) ds dz da \right).
\end{align*}
\end{lemma}

{\bf Proof of Proposition \ref{cotasapriori}}
\begin{proof}
For $0< t\leq T,$ we are interested in estimating $|X(t)|$. From \eqref{eq1},
\begin{align*}
|X(t)|&=\left|X(0) + \int_0^t b(u,X(u),Y(u)){d}u + \int_0^t \sigma(u,X(u),Y(u)){d}B^H(u)\right| \\
&\leq |X(0)| + \int_0^t |b(u,X(u),Y(u))|{d}u + \nonumber \\
&\quad C(\omega) \left(\int_0^t \dfrac{|\sigma(u,X(u),Y(u))|}{u^{\beta}}{d}u + \int_0^t\int_0^u \dfrac{|\sigma(u,X(u),Y(u))-\sigma(v,X(v),Y(v))|}{(u-v)^{\beta+1}}{d}v{d}u\right).
\end{align*}
Let us denote
\begin{align}\label{CA-Is}
    I_1&=\int_0^t |b(u,X(u),Y(u))|{d}u,\\
    I_2&=\int_0^t \dfrac{|\sigma(u,X(u),Y(u))|}{u^{\beta}}{d}u, \nonumber \\
    I_3&=\int_0^t\int_0^u \dfrac{|\sigma(u,X(u),Y(u))-\sigma(v,X(v),Y(v))|}{(u-v)^{\beta+1}}{d}v{d}u.\nonumber 
\end{align}
For $I_1$ and $I_2$ we use the linear growth condition on the drift $b$ and  the diffusion $\sigma$ \eqref{eq5} to obtain 
\begin{align}\label{CA-I1}
    I_1 &\leq L_3 \left( t + \int_0^t |X(u)| du + \int_0^t |Y(u)|du \right) \leq L_3 
    \left( t + T^\beta \int_0^t \frac{|X(u)|}{u^\beta} du + \int_0^t |Y(u)|du \right)    
\end{align}
and 
\begin{align}\label{CA-I2}
    I_2 &\leq \int_0^t \dfrac{L_3(1+|X(u)|+|Y(u)|)}{u^{\beta}}{d}u.
\end{align}

On the other hand, using the linear growth condition on $K$ \eqref{HK1} we obtain 
\begin{align*}
    |Y(u)|
    &\leq \int_{-r}^0 |K(u,s,X(u+s))|{d}s \leq \int_{-r}^0 K_3(1+|X(u+s)|){d}s\\
    &\leq \int_{-r}^{-u} K_3(1+|X(u+s)|){d}s + \int_{-u}^{0} K_3(1+|X(u+s)|){d}s.
    \nonumber 
\end{align*}
In the interval $[-r,-u]$;  $-r+u < s+u <0$. Similarly, in the interval $[-u,0]$; 
$0 < u+s <u$. Then 
\begin{align*}
    |Y(u)|
    &\leq K_3r + \int_{-r}^{-u} K_3|\phi(u+s)|{d}s + K_3\int_{-u}^{0} |X(u+s)|{d}s.
\end{align*}
This leads to
\begin{align}
\label{CA-Y-3}
    |Y(u)|
    &\leq K_3r + K_3(u-r) \sup_{s\in [-r,-u]}|\phi(u+s)| + K_3\int_{-u}^{0} |X(u+s)|{d}s.
\end{align}

Thus, from \eqref{CA-I1}, \eqref{CA-I2}, and \eqref{CA-Y-3}, and considering $C(\omega)$ a generic random variable and may change from line to line, we obtain that
\begin{align}
\label{I1-2}
    I_1 + I_2  
     &\leq C(\omega)\!
    \left(\! t \!+\! T^\beta\! \int_0^t \frac{|X(u)|}{u^\beta} du + \int_0^t (u-r)\! \sup_{s\in [-r,-u]}|\phi(u+s)| + K_3\int_0^t \int_{-u}^{0} \frac{|X(u+s)|}{u^\beta}{d}s du \right) \nonumber\\  
 &\leq C(\omega)
    \left( (t-r)(1 +\sup_{s\in [-r,0]}|\phi(u+s)| ) + \int_0^t \frac{|X(u)|}{u^\beta} du  + \int_0^t  |X^*(u)| u^{1-\beta} du \right) \nonumber\\ 
     &\leq C(\omega)
    \left(1 + \int_0^t  \frac{|X^*(u)|}{u^\beta} du \right). 
\end{align}
where $|X^*(u)|$ is given by \eqref{notation}.

Let us now analyze $I_3$ in \eqref{CA-Is}
\begin{align}\label{CA-I3}
    I_3  
     &= \int_0^t\int_0^u \dfrac{|\sigma(u,X(u),Y(u))-\sigma(v,X(v),Y(v))|}{(u-v)^{\beta+1}}{d}v{d}u \nonumber \\
    &\leq \int_0^t\int_0^u \dfrac{L_1 |u-v| + L_2 |X(u)-X(v)| + L_2 |Y(u)-Y(v)| }{(u-v)^{\beta+1}}{d}v{d}u. 
\end{align}

Now, using the Lipschitz condition on 
$K$ \eqref{HK1}, 
\begin{align}
\label{CA-Y-2-}
    |Y(u)-Y(v)| 
    &\leq \int_{-r}^0 |K(u,s,X(u+s))-K(v,s,X(v+s))|{d}s \nonumber \\ 
    &\leq \int_{-r}^0 K_1|u-v| +K_2 |X(u+s)-X(v+s)|{d}s \nonumber \\
    &\leq C |u-v| + C \int_{-r}^{0} |X(u+s)-X(v+s)|{d}s.
\end{align}

Therefore  \eqref{CA-I3} becomes
\begin{align}\label{CA-I3-1}
    I_3  
     &\leq C T^{1-\beta} + C\int_0^t\int_0^u \frac{ |X(u)-X(v)|}{(u-v)^{\beta+1}} dvdu + 
     C \int_0^t \int_0^u \int_{-r}^0  \frac{|X(u+s) - X(v+s)|}{(u-v)^{\beta+1}} ds{d}v{d}u.
\end{align}

Combining \eqref{I1-2} and \eqref{CA-I3-1} together with the notation in \eqref{notation}, we obtain
\begin{align*}
&|X(t)| \leq C(\omega)   \left(1 + \int_0^t  \frac{|X^*(u)|}{u^\beta} du      + \int_0^t  \int_0^u \varphi_{u,v}^X dv du     + \int_0^t  \int_0^u \int_{-r}^0 \Gamma_{u,v}(s) dsdv du     \right).\nonumber
\end{align*}

The final part of the proof of Proposition \ref{cotasapriori} follows the same argument used in the proof of Theorem \ref{teoCotasapriori} using Lemmas \ref{DifX-Casos}, \ref{CA-XU-XV} and \ref{lema6-3} given above. In particular, once the previous estimates are obtained, the conclusion is derived by applying the Gronwall inequality in the same way as before. Since the same bounds and conditions hold, the argument can be repeated without modification, and therefore we omit the details.
\end{proof}

{\bf Proof of Lemma \ref{DifX-Casos}}

\begin{proof}
We assume that $-r < -u < -v < 0$. To facilitate the analysis, we divide the integral into the interval $[-r,0]$ according to the relative positions of the points $-r$, $-u$, and $-v$. This decomposition allows us to treat each contribution separately.

Accordingly, we partition the interval $[-r,0]$ into the three subintervals
\[
[-r,-u], \quad [-u,-v], \quad [-v,0].
\]
We analyze each case separately.
\medskip

\noindent\textbf{Case I:} $s \in [-r,-u]$.
In this case, $-r \leq s \leq -u$. Adding $u$ yields
\[
-r+u \leq s+u \leq 0.
\]
Similarly, adding $v$ gives
\(
-r+v \leq s+v \leq v-u < 0.
\)
Therefore, we obtain $X(u+s) = \phi(u+s)$,  $X(v+s) = \phi(v+s)$  and
\begin{equation*}
    |X(u+s) - X(v+s)|
    = |\phi(u+s) -\phi(v+s)| \leq |u-v|^{H-\rho }.
\end{equation*}

In this way 
\begin{align*}
\int_{-r}^{-u} \!\! \frac{|X(u+s) - X(v+s)|}{(u-v)^{\beta+1}} ds& \leq\!\!
\int_{-r}^{-u} \!\! \frac{|u-v|^{H-\rho}}{(u-v)^{\beta+1}} ds \leq |u-v|^{H-\rho -\beta -1}(r-u). 
\end{align*}
Then
\begin{align*}
\int_0^u \int_{-r}^{-u}  \frac{|X(u+s) - X(v+s)|}{(u-v)^{\beta+1}} ds dv & \leq (r-u)\int_0^u |u-v|^{H-\rho -\beta -1} dv \leq (r-u) u^{H-\rho-\beta}.
\end{align*}

{\bf{Case II:}}
$s \in [-u,-v]$ \\
In this case, $-u < s < -v$. Thus, by adding $u$, we obtain $0 < s+u < u-v$.
Similarly, by adding $v$, we have $-u+v < s+v < 0$.
That is, one quantity is positive and the other is negative. Therefore,
$X(v+s) = \phi(v+s)$, and
\begin{align}\label{CA-C2}
    |X(u+s) - X(0) + X(0) - X(v+s)|
    &= |X(u+s) -X(0)| + |X(0) - \phi(v+s)| \nonumber \\ &\leq C(\omega) \left(|X(u+s) -X(0)| + |v+s|^{H-\rho } \right).
\end{align}
In this way, since $-u < s$, it follows that $-s < u$. Therefore, $0 <-v - s < u - v$. Thus, 
\begin{align*}
 \int_{-u}^{-v}  \frac{|X(u+s) - X(v+s)|}{(u-v)^{\beta+1}} ds  &\leq
\int_{-u}^{-v}  \frac{|X(u+s) - X(0)|}{(u-v)^{\beta+1}} ds + \int_{-u}^{-v}  \frac{|v+s|^{H-\rho}}{(u-v)^{\beta+1}} ds \nonumber \\ 
&\leq
 \int_{-u}^{-v}  \frac{|X(u+s) - X(0)|}{(u-v)^{\beta+1}} ds + \int_{-u}^{-v}  (u-v)^{H-\rho-\beta-1} ds \nonumber \\ 
 &\leq
 \int_{-u}^{-v}  2\frac{|X^*(u-v)|}{(u-v)^{\beta+1}} ds + (u-v)^{H-\rho - \beta}  \nonumber \leq
 2\frac{|X^*(u-v)|}{(u-v)^{\beta}} + (u-v)^{H-\rho - \beta}.
\end{align*}
Then
\begin{align*}
\int_0^u \int_{-u}^{-v}  \frac{|X(u+s) - X(v+s)|}{(u-v)^{\beta+1}} ds dv & \leq  C \left( \int_0^u
\frac{|X^*(u-v)|}{(u-v)^{\beta}} dv \right)  + \int_0^u (u-v)^{H-\rho - \beta} dv \\
&\leq 
C\left( \int_0^u |X^*(v)|v^{-\beta} dv + u^{H-\rho - \beta +1} \right).
\end{align*}

{\bf{Caso III:}}
$s \in [-v,0]$ \\
In this case, $-v < s < 0$. Adding $u$ to the inequality, we obtain $u - v < s + u < u$. Similarly, by adding $v$, we get $0 < s + v < v$. Hence, both quantities are positive, and the integrals remain unchanged, i.e., no modification of the integration domains is required,
 \begin{align}\label{CA-C3} 
 \int_0^u \int_{-v}^{0}  \frac{|X(u+s) - X(v+s)|}{(u-v)^{\beta+1}} ds dv.
\end{align}
  
Combining the three cases, we obtain the result.
\end{proof}

{\bf Proof of Lemma \ref{CA-XU-XV}}
\begin{proof}
We will first study $|X(u)-X(v)|$, for $v<u$
\begin{align*}
|X(u)-X(v)| \leq &  C
\int_v^u  (1 + |X(a)| + |Y(a)|) da +  C
\int_v^u  \frac{1 + |X(a)| + |Y(a)|}{(a-v)^\beta } da  \nonumber \\
&+
\int_v^u  \int_v^a \frac{|\sigma(a,X(a), Y(a)) -\sigma(z,X(z), Y(z))|}{(a-z)^{\beta +1} } dz da. 
\end{align*}

From \eqref{CA-Y-3} we  have that 
\begin{align*}
    |Y(a)|
    &\leq K_3r + K_3(a-r) \sup_{s\in [-r,-a]}|\phi(a+s)| + K_3\int_{-a}^{0} |X(a+s)|{d}s.
\end{align*}
Thus
\begin{align}\label{CA-XU-XV-2} 
&|X(u)-X(v)| \nonumber \\
\leq &  C
\int_v^u  \left(1 + |X(a)| + K_3r + K_3(a-r) \sup_{s\in [-r,-a]}|\phi(a+s)| + K_3\int_{-a}^{0} |X(a+s)|{d}s \right) da \nonumber \\
&+ C
\int_v^u  \frac{1 + |X(a)| + K_3r + K_3(a-r) \sup_{s\in [-r,-a]}|\phi(a+s)| + K_3\int_{-a}^{0} |X(a+s)|{d}s}{(a-v)^\beta } da  \nonumber \\
&+
\int_v^u  \int_v^a \frac{|\sigma(a,X(a), Y(a)) -\sigma(z,X(z), Y(z))|}{(a-z)^{\beta +1} } dz da \nonumber\\
 \leq  &C \left(1 + \int_v^u |X^*(a)| da +  \int_v^u \frac{|X^*(a)|}{(a-v)^\beta} da \right. +
\left.\int_v^u  \int_v^a \frac{|\sigma(a,X(a), Y(a)) -\sigma(z,X(z), Y(z))|}{(a-z)^{\beta +1} } dz da \right)\nonumber \\ 
\leq & C \left(1  +  \int_v^u \frac{|X^*(a)|}{(a-v)^\beta} da \right. +
\left.\int_v^u  \int_v^a \frac{|\sigma(a,X(a), Y(a)) -\sigma(z,X(z), Y(z))|}{(a-z)^{\beta +1} } dz da \right).
\end{align}
Hence, 
\begin{align*}
\int_0^u \frac{ |X(u)-X(v)|}{(u-v)^{\beta+1}} dv  
 \leq & C  + C \left( \int_0^u   \int_v^u \frac{|X^*(a)|}{(a-v)^\beta   (u-v)^{\beta +1} } da \right. dv\nonumber \\
&+
\left. \int_0^u \int_v^u  \int_v^a \frac{|\sigma(a,X(a), Y(a)) -\sigma(z,X(z), Y(z))|}{(a-z)^{\beta +1} (u-v)^{\beta +1}} dz da dv \right).
\end{align*}
Using inequality (16) in \cite{Mishura2008} and recalling that $\varphi_{u,v}$ is defined in \eqref{notation},   we obtain 
\begin{align*}
\int_0^u  \varphi_{u,v}^X dv  \leq & C  + C \left( \int_0^u |X^*(a)| (u-a)^{-2\beta}  da  \right. \nonumber \\
&+
\left. \int_0^u \int_v^u  \int_v^a \frac{|\sigma(a,X(a), Y(a)) -\sigma(z,X(z), Y(z))|}{(a-z)^{\beta +1} (u-v)^{\beta +1}} dz da dv  \right)\\
 \leq & C  + C \left( \int_0^u   |X^*(a)| (u-a)^{-2\beta}  da  \right. +
\int_0^u \int_v^u  \int_v^a \frac{|X(a) -X(z)|}{(a-z)^{\beta +1} (u-v)^{\beta +1}} dz da dv \nonumber \\
& \qquad \quad + \left.
\int_0^u \int_v^u  \int_v^a \frac{|Y(a) -Y(z)|}{(a-z)^{\beta +1} (u-v)^{\beta +1}} dz da dv\right). 
\end{align*}

By making a change of integration variables as in \cite{Mishura2008} equation (18), we obtain
\begin{align} 
\int_0^u \varphi_{u,v}^X dv   \leq & C  + C \left( \int_0^u   |X^*(a)| (u-a)^{-2\beta}  da +
 \int_{0}^{u} (u-a)^{-\beta}\int_{0}^{a}
\frac{|X(a)-X(z)|}{(a-z)^{\beta+1}} dz da   \right. \nonumber  \\
& \qquad \qquad
\left. +
\int_0^u (u-a)^{-\beta}  \int_0^a \frac{|Y(a) -Y(z)|}{(a-z)^{\beta +1}}  dz da \right)\nonumber\\
\label{CA-XU-XV-9}
= &C  + C \left( \int_0^u   |X^*(a)| (u-a)^{-2\beta}  da +
 \int_{0}^{u} (u-a)^{-\beta}\int_{0}^{a}
\varphi_{a,z}^X dz da   \right. \nonumber  \\
& \qquad \qquad\left.  
 +
\int_0^u (u-a)^{-\beta}  \int_0^a \frac{|Y(a) -Y(z)|}{(a-z)^{\beta +1}} dz da \right) \nonumber \\
\leq & C\left( 1 + \int_0^u   |X^*(a)| (u-a)^{-2\beta}  da  +
 \int_{0}^{u} (u-a)^{-\beta}\int_{0}^{a}
\varphi_{a,z}^X dz da \right. \nonumber \\ 
&\qquad \qquad +
\left.\int_0^u (u-a)^{-\beta}  \int_0^a \int_{-a}^0 \Gamma_{a,z}(s) ds dz da \right).
\end{align}

The last inequality follows from \eqref{CA-Y-2-} and Lemma \ref{DifX-Casos}.
\end{proof}

{\bf Proof of Lemma \ref{lema6-3}}
\begin{proof}
From Lemma \eqref{DifX-Casos}, we obtain the following
\begin{align}\label{lema6}
     & \int_0^u \int_{-r}^0 \Gamma_{u,v}(s) ds{d}v 
\leq  C\left( 1 
+ \int_0^u |X^*(v)|v^{-\beta} dv + \int_0^u \int_{-v}^{0}  \frac{|X(u+s) - X(v+s)|}{(u-v)^{\beta+1}} ds dv \right).
\end{align}
Following the same analysis as in \eqref{CA-XU-XV-2}, we obtain
\begin{align*}
    &  |X(u+s) - X(v+s)|  \nonumber \\ 
&\leq C \left(1+ \int_{v+s}^{u+s} \frac{|X^*(a)|}{ (a-(v+s))^{\beta}} da  + \int_{v+s}^{u+s} \int_{v+s}^a \frac{|\sigma(a,X(a), Y(a)) - \sigma(z,X(z),Y(z))|}{(a-z)^{\beta+1}} dzda \right. \nonumber \\
&\leq C \left(\!1+\! \int_{v+s}^{u+s}\!\! \frac{|X^*(a)|}{ (a-(v+s))^{\beta}} da \!+\!
 \int_{v+s}^{u+s}\! \int_{v+s}^a\!\! \frac{|X(a) -X(z)|}{(a-z)^{\beta+1}} dzda\!
 + \!
 \int_{v+s}^{u+s}\! \int_{v+s}^a \!\!\frac{|Y(a) -Y(z)|}{(a-z)^{\beta+1}} dzda \!\! \right).
\end{align*}
Writing this in \eqref{lema6}, we obtain the following
\begin{align}\label{lema6-1}
     &   \int_0^u \int_{-r}^0 \Gamma_{u,v}(s) ds{d}v \leq   C\left( 1 
+ \int_0^u |X^*(v)|v^{-\beta} dv + \int_0^u \int_{-v}^{0}  \frac{|X(u+s) - X(v+s)|}{(u-v)^{\beta+1}} ds dv \right) \nonumber \\
 \leq  & C\left(\! 1 \!
+\! \int_0^u |X^*(v)|v^{-\beta} dv + \int_0^u \int_{-v}^{0}  
\int_{v+s}^{u+s} \frac{|X^*(a)|}{ (a-(v+s))^{-\beta} (u-v)^{\beta +1}} da dsdv \right. \nonumber \\ 
& \qquad \left.+
 \int_0^u \int_{-v}^{0}   \int_{v+s}^{u+s} \int_{v+s}^a \frac{|X(a) -X(z)|+|Y(a) -Y(z)|}{(a-z)^{\beta+1} (u-v)^{\beta+1}} dzda ds dv \right). 
\end{align}
For the last integral in \eqref{lema6-1}, a change of variables followed by direct computation yields
 \begin{align}
\label{lem15-1}
 & \int_0^u \int_{-v}^{0}   \int_{v+s}^{u+s} \int_{v+s}^a \frac{|X(a) -X(z)|+|Y(a) -Y(z)|}{(a-z)^{\beta+1} (u-v)^{\beta+1}} dzda ds dv \nonumber\\
  =&  \int_{0}^{u} \int_{0}^{a} \int_{a-u}^{0} \int_{-s}^{z-s} 
\frac{|X(a)-X(z)|+|Y(a) -Y(z)|}{(a-z)^{\beta+1} (u-v)^{\beta+1}}  dv ds  dz  da\nonumber\\
\leq & C\int_{0}^{u} \int_{0}^{a} \int_{a-u}^{0}  \frac{|X(a)-X(z)|+|Y(a) -Y(z)|}{(a-z)^{\beta+1} (u-(z-s))^{\beta}}  ds  dz  da\nonumber\\
\leq& C \int_{0}^{u} \int_{0}^{a} (u-z)^{1-\beta} \frac{|X(a)-X(z)|+|Y(a) -Y(z)|}{(a-z)^{\beta+1} }   dz  da \nonumber  \\
\leq& C  \int_{0}^{u} \int_{0}^{a} (u-a)^{-\beta}
\frac{|X(a)-X(z)|+|Y(a) -Y(z)|}{(a-z)^{\beta+1} }   dz  da,
\end{align}
the last inequality is due to $z <a$ implies $u-a < u-z$, and then $\frac{1}{u-z} < \frac{1}{u-a}$.

We study the second  integral of \eqref{lema6-1}, making a substitution, then changing the limits of the integral and taking into account that  $s$ is non-positive, we have 
\begin{align}
\label{lem15-2}
&\int_0^u \int_{-v}^{0}  
\int_{v+s}^{u+s} \frac{|X^*(a)|}{ (a-(v+s))^{-\beta} (u-v)^{\beta +1}} da dsdv\nonumber\\
=&\int_0^u \int_{-v}^{0}  
\int_{v}^{u} \frac{|X^*(a+s)|}{ (a-v)^{-\beta} (u-v)^{\beta +1}} da dsdv = \int_{0}^{u}\int_{-a}^{0}\int_{-s}^{a}
\frac{|X^*(a+s)|}{(a-v)^{\beta}(u-v)^{\beta+1}}\,dv\,ds\,da\nonumber\\
\leq & \int_{0}^{u}\int_{-a}^{0}\int_{0}^{a}
\frac{|X^*(a)|}{(a-v)^{\beta}(u-v)^{\beta+1}}\,dv\,ds\,da
\leq C\int_{0}^{u} \int_{-a}^{0}|X^*(a)| (u-a)^{-2\beta} dsda\nonumber\\
\leq& C\int_{0}^{u} |X^*(a)| (u-a)^{-2\beta} da.
\end{align}

Combining \eqref{lema6-1}, \eqref{lem15-1} and \eqref{lem15-2} yields
\begin{align*}
\int_0^u\! \int_{-r}^0 \! \Gamma_{u,v}(s) ds{d}v & 
 \leq C\left( 1 
+ \int_0^u |X^*(v)|v^{-\beta} dv + 
\int_{0}^{u} |X^*(a)| (u-a)^{-2\beta} da.
 \right.\nonumber  \\
&+\!
\int_0^u (u-a)^{-\beta}\!\! \int_0^a\!
\frac{|X(a) - X(z)|}{(a-z)^{\beta+1} }  dz da \!+\!
\left. \!\int_0^u \!\!(u-a)^{-\beta}\!\! \int_0^a\!
\frac{|Y(a) - Y(z)|}{(a-z)^{\beta+1} }  dz da \right).
\end{align*}
Finally, from \eqref{CA-Y-2-} and Lemma \ref{DifX-Casos},
\begin{align*}
      \int_0^u\! \int_{-r}^0 \!\Gamma_{u,v}(s) ds{d}v  
&\leq C\left( 1 
+ \int_0^u |X^*(v)|v^{-\beta} dv + 
\int_{0}^{u} |X^*(a)| (u-a)^{-2\beta} da.
 \right.\nonumber  \\
&+\!
\int_0^u\!\! (u-a)^{-\beta}\!\! \int_0^a\!
\frac{|X(a) - X(z)|}{(a-z)^{\beta+1} }  dz da \!+\!
\left. \int_0^u\! (u-a)^{-\beta}\!  \int_0^a \!\int_{-a}^0 \Gamma_{a,z}(s) ds dz da \right).
\end{align*}

\end{proof}

\section*{Funding}

\noindent
Alexander Abreu acknowledges financial support through a doctoral fellowship in Mathematics at the  PUCV, Chile.
Héctor Araya was supported by FONDECYT Project No. 11230051, ECOS project ECOS210037 (C21E07) and the MathAmsud project AMSUD210023.
Lisandro Fermin was supported by the Aix-Marseille School of Economics (AMSE), funded by the French National Research Agency (ANR) under grant ANR-17-EURE-0020 and by the Excellence Initiative of Aix-Marseille University - A*MIDEX.
Johanna Garz\'on was partially supported by HERMES project 58557 and Mathamsud EXPLORE-SDE AMSUD240037.
Soledad Torres was partially supported by Fondecyt Reg. project No. 1230807, Fondecyt Reg. project No. 1221373, Mathamsud SMILE AMSUD230032, Mathamsud 
EXPLORE-SDE AMSUD240037 and Mathamsud SiJaVol AMSUD240024. 
This work was supported by Centro de Modelamiento Matemático (CMM) BASAL fund FB210005 for center of excellence from ANID-Chile.

\bibliographystyle{acm}
\bibliography{Referencias}

\end{document}